\documentclass{amsart}
\usepackage{paper}

\title[Invariant splitting principles]{Invariant splitting principles for the Lipshitz--Ozsv\'{a}th--Thurston correspondence}

\author{Gary Guth}
\address{Department of Mathematics, Stanford University, Palo Alto, CA 94301}
\email{gmguth@stanford.edu}

\author{Sungkyung Kang}
\address{Mathematical Institute, University of Oxford, Andrew Wiles Building, Radcliffe Observatory Quarters, Woodstock Road, Oxford, OX2 6GG, United Kingdom}
\email{sungkyung3838@gmail.com}

\begin{document}

\maketitle

\begin{abstract}
We prove that the Lipshitz-Ozsv\'{a}th-Thurston correspondence between extended type D structures of knot complements and $\F[U, V]/(UV)$ knot Floer complexes can be arranged so that $\iota_K$-invariant splittings of knot Floer chain complexes correspond to $\iota_{S^3 \setminus K}$-invariant splittings of bordered Floer homology of knot complements. This provides a novel way to compute the involutive knot Floer homology of satellites from that of their companions. As a topological application, we show that our results can be applied to construct infinitely many examples of exotic pairs of contractible 4-manifolds which remain exotic after one stabilization. Along the way, we also establish first order naturality of bordered Floer homology. 
\end{abstract}

\tableofcontents

\section{Introduction}\label{sec: introduction}

\emph{Heegaard Floer homology} is an expansive kit of tools for low dimensional topologists introduced by Ozsv\'ath and Szab\'o in \cite{os_holodisks,ozsvath2004holomorphicproperties}. Heegaard Floer homology assigns invariants to 3-manifolds as well as to links inside of them \cite{rasmussen_knotcompl,os_knotinvts,os_linkinvts}; moreover, Heegaard Floer homology has the structure of a TQFT, assigning maps to 4-dimensional cobordisms between 3-manifolds as well as to surface cobordisms between links \cite{os_holotri,zemke_graphcob,zemke_linkcob}. 

\emph{Bordered Floer homology} is a collection of invariants for 3-manifolds with boundary, due to \LOT \cite{LOT_bordered_HF}. In this theory, each 3-manifold $Y$ is assigned a pair of invariants, $\CFAh(Y)$ and $\CFDh(Y)$ (an $\cA_\infty$-module and a differential graded module respectively). Importantly, there is a \emph{pairing theorem}, which states that if $Y_1$ and $Y_2$ have common boundary, then
\[
\CFh(Y_1 \cup Y_2)\simeq \CFAh(Y_1) \boxtimes\CFDh(Y_2),
\]
where $\CFAh(Y_1) \boxtimes \CFDh(Y_2)$ is a convenient model for the derived tensor product $\CFAh(Y_1)\widetilde{\otimes}\CFDh(Y_2)$.

\emph{Involutive Heegaard Floer homology} is an extension of these invariants introduced by Hendricks and Manolescu \cite{hendricks_manolescu_Invol}.
Analogous to the intrinsic $Pin(2)$-symmetry in Seiberg--Witten Floer homology, Heegaard Floer homology has an intrinsic $\Z/2$-symmetry which arises from the conjugation of $\mathrm{Spin}^c$ structures on the underlying manifolds. This induces an action on $\CF^-(Y)$, usually denoted $\iota$, which is used to define involutive Heegaard Floer homology. In a similar vein, Hendricks--Manolescu define an action on the knot Floer complex $\CFK_{\mathbb{F}[U,V]}(S^3,K)$, denoted $\iota_K$, and define a new knot invariant, \emph{involutive knot Floer homology}. 

These two extensions of the original Heegaard Floer homology package have been utilized very successfully in the study of 3.5-dimensional topology. In the realm of concordance and homology cobordism, involutive Floer homology has been used to prove the 3-dimensional homology cobordism group contains a $\Z^\infty$ summand \cite{DHSL_infinite_rank} and further, is not even generated by Seifert fibered spaces \cite{HHSZ_surgery_exact_invol}. Involutive knot Floer homology has been used by Hom--Park--Stoffregen and the second author to prove the existence of an infinite family of rationally slice knots which are linearly independent in the smooth concordance group \cite{HKPS_lin_ind_rationally_slice}. Bordered Floer homology has proven to be quite adept at detecting exotic behavior: the first author used these invariants to produce stably-exotic surfaces in the 4-ball \cite{guth_one_not_enough_exotic_surfaces} (i.e. surfaces which remain exotic after attaching arbitrarily many 1-handles); work of the authors with Hayden-Park showed that vast families of exotic surfaces in $B^4$ can be constructed utilizing similar satellite techniques \cite{guth2023doubled}. The second author combined both involutive and bordered Floer homology to prove the existence of a pair of stably exotic 4-manifolds with boundary \cite{kang2022stabilization} (i.e. which remain exotic after taking a connected sum with $S^2\times S^2)$. Additionally, Cohen \cite{cohen2023composition} has given an algorithm for computing cobordism maps on $\HFh$ using bordered Floer homology, extending the algorithm for computing $\HFh$ for 3-manifolds in \cite{LOT_factoring_mapping_classes}.

\subsection{Satellites and Spin$^c$ conjugation:}

Several of these recent results rely on \emph{satellite operations}, which provide a controlled means of introducing topological complexity; it is therefore sensible to understand their effect on topological invariants. In particular, in this article we consider the question of understanding the action of $\iota_K$ on the knot Floer homology of a satellite knot. Over the ring $\cR := \F[U,V]/(UV)$, understanding the knot Floer complex of satellites is quite tractable via bordered Floer homology \cite{LOT_bordered_HF}; despite this, understanding the involution is quite difficult. 

The first computation of involutive knot Floer homology for a satellite knot appeared in \cite{HKPS_lin_ind_rationally_slice}, where the $\iota_K$-action was computed for $(2n+1,1)$-cables of the figure-eight knot. The computation relies the fact that $\iota^2_K$ is chain homotopic to the Sarkar map, $1+\Phi\Psi$, where $\Phi$ and $\Psi$ are formal derivatives of the differential of the full $CFK_{\F[U,V]}$ complex with respect to the formal variables $U$ and $V$, respectively. Later, Collins computed the involutive knot Floer homology of the Mazur satellite of $4_1$ using a similar set of ideas \cite{collins2022homology}. However, this strategy has clear limitations: this strategy often leads to an impossibly big case-by-case analysis unless $K$ is simple (i.e.,  $CFK_{\F[U,V]}(S^3,K)$ is not too big.) Moreover, this strategy does not import any information from the involutive action on the original (companion) knot. The aim of the present paper is to develop a more efficient method for understanding the interaction between satellite operations and the Spin$^c$ conjugation action by taking into account the action of $\iota_K$ on the companion knot.

\subsection{Satellite Splittings}

As we alluded above, for a companion knot $K$ in $S^3$ and pattern knot $P$ in the solid torus, the $\cR$-coefficient knot Floer homology, $\CFK_\cR(S^3,P(K))$, can be understood by applying the techniques of bordered Floer homology \cite{LOT_bordered_HF}. When $P$ is relatively simple (for example, when $P$ is a cabling pattern, a doubling pattern, a Mazur pattern, etc.) $\CFK_\cR(S^3, P(K))$ is quite computable; consequently, the knot Floer homology of these satellites is well understood \cite{homcabling,hedden_cables_I,hedden_cables_II,hedden_whitehead,levine_doubling_operators,levine_non_surjective,hanselman_watson_cabling,chen_satellites_one_diagrams,chen_hanselman2023satellite}.

This strategy is particularly efficient when the knot Floer complex of $K$ splits into manageable pieces: 
\[
\CFK_\cR(S^3,K) \cong C_1 \oplus \hdots \oplus C_n.
\]
By \cite[Chapter 11]{LOT_bordered_HF}, the $\cR$-coefficient knot Floer complex for $K$ determines the type D structure of its complement, $\CFDh(\KC)$ (and vice versa). We refer to this as the \emph{\LOT correspondence}. Moreover, splittings of $\CFK_\cR(S^3,K)$ induce splittings of $\CFDh(\KC)$. The type D structure for the complement of $P(K)$ is obtained from $\CFDh(\KC)$ by tensoring with the type DA bimodule associated to the pattern $P$: 
\[
\CFDh(S^3\smallsetminus P(K)) \simeq \CFDAh((S^1\times D^2) \smallsetminus P)\boxtimes \CFDh(\KC).
\]
The induced splitting of $\CFDh(\KC)$ produces a splitting of $\CFDh(S^3\smallsetminus P(K))$, which in term gives rise to a splitting of the knot Floer complex of $P(K)$: 
\[
\CFK_\cR(S^3,P(K)) \cong P(C_1) \oplus \hdots \oplus P(C_n),
\]
where $P(C_k)$ is determined by $C_k$ and $P$\footnote{It is actually not obvious that splittings of $\CFDh(S^3\smallsetminus P(K))$ determine splittings of $\CFK_\cR(S^3,P(K))$. 
This follows from \Cref{lem:summand-extendable}. See \Cref{sec:basepoints and extensions,sec: doubling} for a more detailed discussion.}. Essential information about $K$ (e.g., its local equivalence class or torsion order) is often contained in a single summand of $\CFK_\cR(S^3,K)$, and consequently, the analogous statement is often true for $P(K)$ as well. Leveraging this fact can greatly simplify the complexity of computations needed to extract information about $P(K)$. 

A priori, it is unclear whether this strategy can used to understand the $\iota_K$ action on the Floer complex of satellite knots. Given an $\iota_K$ equivariant direct sum decomposition of $\CFK_\cR(S^3,K)$, it is not guaranteed that the induced splitting of $\CFK_\cR(S^3,P(K))$ will be $\iota_{P(K)}$-equivariant. Our first result states that, up to homotopy, we can take this to be the case. 

\begin{thm}\label{thm: induced splittings}
Let $K$ be a knot in $S^3$ and let $P$ be a satellite pattern in the solid torus. Moreover, let 
\[
\CFK_\cR(S^3,K) \cong C_1 \oplus \hdots \oplus C_n
\]
be an $\iota_K$-equivariant splitting. Then, up to homotopy, the decomposition 
\[
\CFK_\cR(S^3,P(K)) \cong P(C_1) \oplus \hdots \oplus P(C_n)
\]
given by \LOT is $\iota_{P(K)}$-equivariant.
\end{thm}

\subsection{The Invariant Splitting Principles}

Our strategy will be to understand the relation between $\iota_K$ and the corresponding action on $\CFDh(\KC)$, which takes the form of a homotopy equivalence
\begin{align}
    \iota_{\KC}: \widehat{CFDA}(\AZ)\boxtimes \widehat{CFD}(\KC) \ra \CFDh(\KC).
\end{align}
Here, $\AZ$ denotes a particular diagram, the \emph{Auroux-Zarev piece}, considered independently in \cite{auroux_symprod_BFH,zarev2010joining}.

\begin{defn}
    We say that a splitting $\CFDh(\KC) \cong M_1 \oplus \hdots \oplus M_n$ is $\iota_{\KC}$-equivariant if there are homotopy equivalences
    \[
    \frak{I}_j:\widehat{CFDA}(\AZ)\boxtimes M_i \rightarrow M_i
    \]
    such that $\iota_{\KC} \sim \frak{I}_1 \oplus \cdots \oplus \frak{I}_n$, i.e. the bordered involution of $\KC$ decomposes along the given splitting up to homotopy.
\end{defn}

We derive \Cref{thm: induced splittings} as a corollary of the following two theorems, which we call the \emph{invariant splitting principles} for the \LOT correspondence.

\begin{thm}\label{thm:CFK-to-CFA}
    Given a knot $K$ and an $\iota_K$-invariant splitting of $\CFK_\cR(S^3, K),$
    \[
    CFK_\mathcal{R}(S^3,K) \simeq C_1 \oplus \cdots \oplus C_n,
    \]
    there is a corresponding $\iota_{\KC}$-equivariant splitting of $\CFDh(\KC)$,
    \[
    \widehat{CFD}(S^3 \setminus K)\simeq M_1 \oplus \cdots \oplus M_n
    \]
    such that each $M_i$ corresponds to $C_i$ under the \LOT correspondence.
\end{thm}

\begin{thm}\label{thm:CFA-to-CFK}
    Given a knot $K$ and an $\iota_{S^3 \setminus K}$-invariant splitting
    \[
    \widehat{CFD}(S^3 \setminus K)\simeq M_1 \oplus \cdots \oplus M_n,
    \]
    there is an $\iota_K$-invariant splitting
    \[
    CFK_\mathcal{R}(S^3,K) \simeq C_1 \oplus \cdots \oplus C_n
    \]
    such that each $C_i$ corresponds to $M_i$ under the \LOT correspondence.
\end{thm}

Using \Cref{thm:CFK-to-CFA} and \Cref{thm:CFA-to-CFK}, we prove the following \emph{relative splitting principles}. Recall that if $K$ and $K'$ are knots in $S^3$, we say that $\CFK_\cR(S^3,K)$ and $\CFK_\cR(S^3,K')$ are locally equivalent if there is a map $f: \CFK_\cR(S^3,K) \ra \CFK_\cR(S^3,K')$ which induces an isomorphism $H_*((U,V)^{-1}\CFK_\cR(S^3,K)) \ra H_*((U,V)^{-1}\CFK_\cR(S^3,K'))$. 

Given a map \[f: C_1 \oplus \hdots \oplus C_n \ra D_1 \oplus\hdots \oplus D_m\] represented as an $m$ by $n$ matrix, we denote its $ij$-component by $f_{ij} = \pi_{D_i }\circ f|_{C_j}$.

\begin{thm}\label{thm:relative CFK-to-CFA}
    Let $K$ and $K^\prime$ be knots in $S^3$ whose knot Floer complexes are locally equivalent over $\mathcal{R}$, and fix splittings
    \[
        CFK_\mathcal{R}(S^3,K) \simeq C_1 \oplus \cdots \oplus C_m, \quad CFK_\mathcal{R}(S^3,K^\prime) \simeq D_1 \oplus \cdots \oplus D_n,
    \]
    as well as a bidegree-preserving chain map $f:CFK_\mathcal{R}(S^3,K)\rightarrow CFK_\mathcal{R}(S^3,K^\prime)$. If $f$ is $\iota_K$-equivariant, i.e. $f\iota_K \sim \iota_{K^\prime} f$, then there exist splittings
    \[
        \CFDh(S^3 \smallsetminus K) \simeq M_1 \oplus \cdots \oplus M_m, \quad
        \CFDh(S^3 \smallsetminus K^\prime) \simeq N_1 \oplus \cdots \oplus N_n,
    \]
    such that each $C_i$ corresponds to $M_i$ and $D_j$ corresponds to $N_j$ under the \LOT correspondence, and a degree-preserving type D morphism (with respect to the refined grading defined in \cite[Section 10.5]{LOT_bordered_HF})
    \[
    f^D:\widehat{CFD}(S^3 \smallsetminus K)\rightarrow \CFDh(S^3 \smallsetminus K^\prime)
    \]
    satisfying $f^D \circ \iota_{S^3 \smallsetminus K} \sim \iota_{S^3 \smallsetminus K^\prime}\circ (\mathrm{id}_{\widehat{CFDA}(\AZ)}\boxtimes f^D)$, such that $f_{ij} \sim 0$ if and only if $f^D_{ij} \sim 0$ and  $f_{ij}$ is a homotopy equivalence if and only if $f^D_{ij}$ is a homotopy equivalence.
\end{thm}

\begin{thm}\label{thm:relative CFA-to-CFK}
    Let $K$ and $K^\prime$ be knots in $S^3$ whose knot Floer complexes are locally equivalent over $\mathcal{R}$, and fix splittings
    \[
        \CFDh(S^3 \smallsetminus K) \simeq M_1 \oplus \cdots \oplus M_m, \quad
        \CFDh(S^3 \smallsetminus K^\prime) \simeq N_1 \oplus \cdots \oplus N_n,
    \]
    and a degree-preserving type-D morphism $f^D:\widehat{CFD}(S^3 \smallsetminus K)\rightarrow \CFDh(S^3 \smallsetminus K^\prime)$. Suppose that $f^D$ satisfies $f^D \circ \iota_{S^3 \smallsetminus K} \sim \iota_{S^3 \smallsetminus K^\prime}\circ (\mathrm{id}_{\widehat{CFDA}(\AZ)}\boxtimes f^D)$. Then, there exist splittings
    \[
        CFK_\mathcal{R}(S^3,K) \simeq C_1 \oplus \cdots \oplus C_m, \quad CFK_\mathcal{R}(S^3,K^\prime) \simeq D_1 \oplus \cdots \oplus D_n,
    \]
    such that each $C_i$ corresponds to $M_i$ and $D_j$ corresponds to $N_j$ under the \LOT correspondence, and a bidegree-preserving $\iota_K$-equivariant chain map
    \[
    f:CFK_\mathcal{R}(S^3,K)\rightarrow CFK_\mathcal{R}(S^3,K^\prime)
    \]
    such that $f_{ij}\sim 0$ if and only if $f^D_{ij}\sim 0$ and $f_{ij}$ is a homotopy equivalence if and only if $f^D_{ij}$ is a homotopy equivalence.
\end{thm}

Along the way, in Subsection \ref{sec:naturality}, we prove a naturality statement for bordered Floer homology. 

\begin{thm}\label{thm:intro_naturality}
    The bordered type A and D invariants are natural.
\end{thm}

In particular, this implies that bordered involutions considered in \cite{HL_Inv_bordered_floer} are well defined, as is the bordered contact invariant of \cite{alishahi_friendly_contact}.

\subsection{Topological Application}

Finally, using \Cref{thm: induced splittings}, we show the following generalization of the results in \cite{kang2022stabilization}.

\begin{thm} \label{thm: one stab inf family}
    Given any odd integers $n,m\ge 3$, let $K_{m,n}$ be the knot $(2T_{2n,2n+1}\# T_{2n,4n+1})_{m,1}$. Then $Y_{m,n}=S^3 _{+1}(4K_{m,n}\# -4K_{m,n})$ bounds a pair of smooth contractible 4-manifolds which are homeomorphic but not diffeomorphic even after one stabilization, i.e. connected summing with one copy of $S^2 \times S^2$. 
\end{thm}

\subsection{Structural Overview}

To motivate the sections that follow, we briefly outline the strategy for the proof of Theorem \ref{thm: induced splittings}. Recall that by \cite[Theorem 11.26]{LOT_bordered_HF}, the data of $\CFK_{\cR}(S^3,K)$ is equivalent to the data of a type D structure for $S^3 \smallsetminus K$. Moreover, by \cite{HL_Inv_bordered_floer}, the bordered module $\CFDh(S^3 \smallsetminus K)$ may be equipped with an involutive structure in the form of a map 
\[
\iota_{S^3 \smallsetminus K}: \widehat{CFDA}(\AZ)\boxtimes \CFDh(\KC) \ra \CFDh(\KC).
\]
Under the bordered Floer homology pairing theorem, this map recovers the $\iota_K$-action on $\CFKh(S^3,K)$ \cite{kang_bordered_involutive_HFK}. Given an $\iota_{\KC}$-equivariant splitting $\CFDh(\KC) = M \oplus N$ into extendable type D structures (see \Cref{def:extendable summand}), we obtain a splitting $\CFK_\cR(S^3,K) = C \oplus D$ by tensoring (the extended modules) with the (extended) type A structure for the $S^1$-fiber in the solid torus (which we denote by $(T_\infty, \nu)$ for brevity). See \Cref{fig:T_inf_hd} for a bordered Heegaard diagram. There are two steps in showing that this splitting respects the action of $\iota_K$. 

First, in Section \ref{sec: cobs}, we show that there is a natural cobordism which gives rise to a map
\[
\Lambda: \End^{\cA}(\CFDh(\KC))\ra \End_\cR(\CFK_\cR(S^3,K)).
\]
This cobordism map has the property that if $p_M \in \Mor^{\cA}(\CFDh(\KC), \CFDh(\KC))$ is projection to $M$, then $\Lambda(p_M)$ is (homotopic to) a projection map on $\CFK(S^3,K)$. Furthermore, the construction of $\Lambda$ implies that it intertwines the conjugation actions by $\iota_{\KC}$ and $\iota_K$. 

It turns out that the summand $\Lambda(M):=\Lambda(p_M)(\CFK(S^3,K))$ has a simple bordered Floer theoretic description. There is another natural projection map obtained by tensoring $p_M$ with the identity:
\[
\CFAx \boxtimes\CFDh(\KC) \xrightarrow{\mathbb{I_\bX}\boxtimes p_M} \CFAx \boxtimes \CFDh(\KC) 
\]
where $\bX$ is a bordered Heegaard diagram for $(T_\infty, \nu)$ with a free basepoint on $\partial T_\infty$; see \Cref{fig:x_hd}. The image of this projection is the summand $\CFAx \boxtimes M$. A key result of Section \ref{sec: cobs} is that the summands $\Lambda(M)\oplus \Lambda(M)$ and $ \CFAx \boxtimes M$ are actually homotopy equivalent. 

The second step is to relate $\CFAx \boxtimes M$ to the summand of $\CFK_\cR(S^3,K)$ determined by the \LOT correspondence. In Section \ref{sec: doubling}, we show that for any direct summand of $\CFDh(\KC)$ (for any knot $K$), we have
\[
\CFAx \boxtimes M \simeq\cC_M\oplus\cC_M,
\]
where $\cC_M$ is the $\cR$-complex which corresponds to $M$ under the \LOT correspondence (from which it follows that $\cC_M$ can be identified with $\Lambda(M)$). The invariant splitting principles, \Cref{thm:CFK-to-CFA} and \Cref{thm:CFA-to-CFK} follow from these two observations, and their relative versions, \Cref{thm:relative CFK-to-CFA} and \Cref{thm:relative CFA-to-CFK} follow similarly. We then use \Cref{thm:CFK-to-CFA} and \Cref{thm:CFA-to-CFK} to prove \Cref{thm: induced splittings} in \Cref{sec:maintheoremproof}. We conclude the paper by using these results to prove our main topological application, \Cref{thm: one stab inf family}.

\begin{rem}
    We note that in \cite{HL_Inv_bordered_floer,kang_bordered_involutive_HFK,kang2022stabilization,kang2022torsion}, it was necessary to work with a particular choice of bordered involution (determined by some auxiliary data), since it had not been established that the bordered involutions were well defined. 
\end{rem}

\subsection*{Notations} We use the following distinction between the symbols $\sim$ and $\simeq$: $X\simeq Y$ means that objects $X$ and $Y$ are homotopy equivalent, while $f\sim g$ means that maps $f$ and $g$ are homotopic. When we deal with additional (free) basepoints, we will denote them as $p$. The unknot will always be written as $U$.

\subsection*{Relation to the ``one is not enough'' result of \cite{kang2022stabilization}} The results of \cite{kang2022stabilization} depend on some of the propositions and lemmas in this paper. In particular, \Cref{prop:doubling-CFA-to-CFK}, \Cref{prop: bijective on projection}, and \Cref{lem: lambda commutes with involutions} are the results that are crucial in the proof of \cite[Theorem 1.1]{kang2022stabilization}.

\subsection*{Acknowledgments} The authors would like to thank Robert Lipshitz for numerous helpful discussions as well as Jesse Cohen and Ciprian Manolescu for helpful conversations. They are also grateful for the careful reading and suggestions of the anonymous referees. GG is supported by a Simons Collaboration Grant on New Structures in Low Dimensional Topology.

\section{Background} \label{sec: background}

In this section, we provide a brief overview of the relevant aspects of bordered and involutive Floer theory; for a more thorough overview, see \cite{LOT_bordered_HF,LOT_notes,LOT_bimodules, hendricks_manolescu_Invol,HL_Inv_bordered_floer}. 

\subsection{Bordered Floer homology}\label{sec: bordered background}

Bordered Floer homology is a package of invariants of 3-manifolds with parametrized boundary. To a surface $F$, bordered Floer homology associates a differential graded algebra $\cA(F)$, and to a 3-manifold $Y$ with boundary with an identification $\phi: \partial Y \ra F$, a type D structure over $\cA(-F)$, called $\CFDh(Y, \phi)$, as well as a type A structure (a right $\cA_\infty$-module) over $\cA(F)$ called $\CFAh(Y, \phi)$. We will work almost exclusively with the algebra associated to the torus, which can be described succinctly as a quotient of a certain path algebra: 
\begin{equation}\label{eq:torus-alg}
  \xymatrix{   \iota_0\,\bullet\ar@/^1pc/[r]^{\rho_1}\ar@/_1pc/[r]_{\rho_3} & \circ\,\iota_1\ar[l]_{\rho_2}
  }/(\rho_2\rho_1=\rho_3\rho_2=0).
\end{equation}
To avoid cluttering the notation, we will usually write $\cA$ instead of $\cA(F)$.

Bordered Floer homology has a pairing theorem \cite[Theorem 1.3]{LOT_bordered_HF}, which recovers the hat-version of the Heegaard Floer homology of the manifold obtained by gluing bordered manifolds along their common boundary. Given 3-manifolds $Y_1$ and $Y_2$ with $\partial Y_1 \cong F \cong -\partial Y_2$, there is a homotopy equivalence
\[
	\CFh(Y_1\cup Y_2) \simeq \CFAh(Y_1) \boxtimes_{\cA} \CFDh(Y_2).
\]
Bordered Floer theory also recovers knot Floer homology \cite[Theorem 11.21]{LOT_bordered_HF}. Given a doubly pointed bordered Heegaard diagram $(\mathcal{H}_1, w, z)$ for $(Y_1, \partial F, K)$ and a bordered Heegaard diagram $(\mathcal{H}_2, z)$ with $\partial Y_1 \cong F \cong -\partial Y_2$, then 
\[
	\CFKh(Y_1 \cup Y_2, K) \simeq \CFAh(\mathcal{H}_1, w, z) \boxtimes_{\cA} \CFDh(\mathcal{H}_2, z).
\]
Typically in the literature, one basepoint lies on the boundary of the Heegaard surface, while the other lies in the interior. Our bordered Heegaard diagrams will always have a basepoint on the boundary, but we will allow for more than one interior basepoint. 

There is also a morphism formulation of the pairing theorem for singly based bordered three-manifolds: 
\[
\Mor^{\cA}(\CFDh(Y_1), \CFDh(Y_2)) \simeq \CFh(-Y_1 \cup Y_2),
\]
see \cite{LOT_HF_as_morphism}. Moreover, the composition map 
\[
\Mor^{\cA}(\CFDh(Y_1), \CFDh(Y_2)) \otimes \Mor^{\cA}(\CFDh(Y_2), \CFDh(Y_3)) \ra \Mor^{\cA}(\CFDh(Y_1), \CFDh(Y_3))
\]
is induced by a cobordism, which we will typically refer to as the \emph{pair of pants} cobordism \cite{cohen2023composition}. For an explicit description of this four-manifold, see Section \ref{sec: cobs}.\\

\subsection{Involutive Floer homology and its bordered formulation}

Involutive Heegaard Floer homology was introduced by Hendricks and Manolescu in \cite{hendricks_manolescu_Invol}. Given a Heegaard diagram $\cH$ representing a based three manifold $Y$, let $\overline{\cH}$ denote the \emph{conjugate diagram} obtained by reversing the orientation of the Heegaard surface and reversing the roles of the alpha and beta circles. There is a canonical map 
\[
\CF(\cH) \xrightarrow{\eta} \CF(\overline{\cH}),
\]
taking tuples of intersection points to the same tuple in the conjugate diagram. Since $\cH$ and $\overline{\cH}$ represent the same 3-manifold, they are related by a sequence of Heegaard moves, which produce a map 
\[
\Phi: \CF(\overline{\cH}) \ra \CF(\cH),
\]
by the naturality of Heegaard Floer homology. The composition of these two maps is usually denoted $\iota$, and the \emph{involutive Heegaard Floer homology} of $Y$ is defined to be the homology of 
\[
\Cone(\CF(Y) \xrightarrow{Q(1 + \iota)} Q \cdot \CF(Y)).
\]
When $\cH$ is a doubly pointed Heegaard diagram representing a knot (say, in $S^3$, for concreteness), a similar construction can be carried out; the only difference is that $\cH$ and $\overline{\cH}$ represent the same knot but with the roles of the two basepoints reversed. We therefore must pushforward the Heegaard diagram $\overline{\cH}$ by a half Dehn twist along the knot before choosing a sequence of Heegaard moves back to the original Heegaard diagram. The composite map is denoted $\iota_K$. \\

Hendricks and Lipshitz formulated a bordered involutive theory in \cite{HL_Inv_bordered_floer}. Given a three-manifold $Y = Y_1 \cup_F Y_2$ represented by a pair of glued bordered Heegaard diagrams $\cH_1 \cup \cH_2$, we can try to emulate the construction above. There is a technical issue however, in that there is an essential asymmetry in a bordered Heegaard diagram -- a bordered Heegaard diagram is \emph{either} an alpha-bordered diagram or a beta-bordered diagram, and an alpha-bordered Heegaard diagram cannot be related to a beta-bordered Heegaard diagram by bordered Heegaard moves. Therefore, when the glued diagram $\cH_1\cup \cH_2$ is conjugated, the diagrams $\overline{\cH}_1$ and  $\overline{\cH}_2$ become $\beta$-bordered diagrams, and therefore cannot be related to $\cH_1$ and $\cH_2$ by bordered Heegaard moves. Hendricks and Lipshitz resolve this issue by inserting a particular interpolating piece, called the \emph{Auroux-Zarev piece}, which we will denote by $\AZ$. $\AZ$ is a diagram with two boundary components; one side is alpha-bordered and the other is beta-bordered. Moreover, it is shown in \cite{LOT_HF_as_morphism} that the diagram $\overline{\AZ}\cup\AZ$ represents the product cobordism $F \times I$. Furthermore, $\overline{\cH}_1\cup \overline{\AZ}$ and $\overline{\cH}_2\cup \AZ$ are related to $\cH_1$ and $\cH_2$ by Heegaard moves. Then, as above, a choice of a sequence of Heegaard moves $\frak{H}_1$ and $\frak{H}_2$
\[
\overline{\cH}_1\cup \bI \cup \overline{\cH}_2 \sim \overline{\cH}_1\cup \overline{\AZ}\cup \AZ \cup \overline{\cH}_2 \sim \cH_1 \cup H_2
\]
induces maps
\[
\iota_{Y_1,\frak{H}_1}: \CFAh(Y_1) \boxtimes \CFDAh(\overline{\AZ}) \ra \CFAh(Y_1)
\]
\[
\iota_{Y_2,\frak{H}_2}: \CFDAh(\AZ) \boxtimes\CFDh(Y_2) \ra \CFDh(Y_2),
\]
which, when tensored together, give a model for the action of $\iota$. Here, we emphasize that in order to show that the maps above are independent of the additional data of the sequence of bordered Heegaard moves, one needs to establish naturality of the bordered modules. However, even without bordered naturality, this framework is quite useful; for any choice of Heegaard moves $\frak{H}_1$ and $\frak{H}_2$, the box tensor product $\iota_{Y_1,\frak{H}_1}\boxtimes\iota_{Y_2,\frak{H}_2}$ is well-defined! Hence, if one is only interested in the action of $\iota_{Y_1\cup Y_2}$, one can safely pick \emph{any} candidate for $\iota_{Y_i,\frak{H}_i}$, since after gluing, all choices will agree. For this reason, the applications in \cite{HL_Inv_bordered_floer,kang2022torsion,kang2022stabilization} do not actually rely on the naturality of the bordered invariants. 

Despite the various work-arounds, we nevertheless do establish (first order) naturality of the bordered invariants in \Cref{sec:naturality}. The following is an immediate consequence of that result.

\begin{thm}
    The homotopy classes of the type A and D morphisms $\iota_{Y_1}$ and $\iota_{Y_2}$ are well defined; i.e., they are independent of the choices of Heegaard moves used in their definition up to homotopy.
\end{thm}

The extension of involutive bordered theory to knots is due to the second author \cite{kang_bordered_involutive_HFK}. Given a knot $K$ in $S^3$, the construction above gives rise to a bordered involution 
\[
\iota_{S^3\smallsetminus K}: \CFAh(S^3\smallsetminus K) \boxtimes \CFDAh(\overline{\AZ}) \ra \CFAh(S^3\smallsetminus K).
\]
By \cite[Theorem 1.3]{kang_bordered_involutive_HFK}, the involution $\iota_{\KC}$ determines the action of $\iota_K$ on $\CFKh(S^3,K)$, i.e. there is a particular choice of map $f \in \Mor^{\cA(T^2)}(\CFDh(T_\infty, \nu), \CFDAh(\AZ) \boxtimes \CFDh(T_\infty, \nu))$ such that the diagram
\begin{align*}
	\begin{tikzcd}[ampersand replacement=\&]
		\CFAh(S^3 \smallsetminus K) \boxtimes \CFDh(T_\infty, \nu) \ar[r,"\simeq"] \ar[d,"\iota_{S^3 \smallsetminus K}^{-1}\boxtimes f"] \&
		\CFKh(S^3,K) \ar[dd,"\iota_K^\varepsilon"] \\
		\overunderset{\hspace{-3mm}\displaystyle\CFAh(S^3 \smallsetminus K) \boxtimes \CFDAh(\overline{\AZ})}{\hspace{0.7mm}\displaystyle\CFDAh(\AZ)\boxtimes\CFDh(T_\infty, \nu)}{\boxtimes} \ar[d,"\simeq"] \& \\
		\CFAh(S^3 \smallsetminus K) \boxtimes \CFDh(T_\infty, \nu) \ar[r,"\simeq"]  \&
		\CFKh(S^3,K),
	\end{tikzcd}
\end{align*}
commutes up to homotopy. $\CFDh(T_\infty, \nu)$ is the type D structure associated to the bordered Heegaard diagram shown in \Cref{fig:T_inf_hd}, which represents the pair $(T_\infty, \nu)$, where $T_\infty$ is the infinity framed solid torus and $\nu$ is an $S^1$-fiber.\footnote{We note that this is an abuse of notation; $\CFDh(T_\infty, \nu)$ depends not only on $\nu$, but a choice of an arc connecting $\nu$ to $\partial T^\infty$. Therefore, when we refer to  $\CFDh(T_\infty, \nu)$, we always refer to the particular diagram in \Cref{fig:T_inf_hd}.} Here, $\varepsilon \in \{\pm 1\}$; surprisingly, from the arguments of \cite{kang_bordered_involutive_HFK} it is unclear whether one recovers $\iota_K$ or its inverse from the bordered perspective. This is ultimately of little importance, since the map $(\CFK(K), \iota_K) \xra{\iota^2} (\CFK(K), \iota_K^{-1})$ is a homotopy equivalence, which intertwines the action of $\iota_K$ and its inverse. Hence, no essential information is lost.

We also note that in \cite{kang_bordered_involutive_HFK}, it was necessary to work with a particular choice of involution (as $\iota_{S^3 \smallsetminus K}$ was not well defined). Though, as we will establish naturality for the bordered invariants in the next section, we assume here that $\iota_{S^3 \smallsetminus K}$ is well-defined.

Having reviewed the basic notions of bordered Floer homology, we turn to the question of naturality.

\section{Naturality in bordered Floer homology} \label{sec:naturality}

We first briefly recall the naturality result for ordinary Heegaard Floer homology and some of the formalism set up in \cite{JTZ_naturality_mapping_class_groups}. 
\begin{defn}\label{def:transitive_sys}
    Let $\cC$ be a category and let $\mathcal{I}$ be a set. A \emph{transitive system in $\cC$ indexed by $\mathcal{I}$} is a collection of objects $(X_i)_{i \in \mathcal{I}}$ as well as a distinguished morphism $\Psi_{i\ra j}: X_i \ra X_j$ for every $(i, j) \in \mathcal{I} \times \mathcal{I}$ such that 
    \begin{enumerate}
        \item $\Psi_{j \ra k} \circ \Psi_{i \ra j} = \Psi_{i \ra k}$, and 
        \item $\Psi_{i\ra i} = \id_{X_i}$.
    \end{enumerate}
\end{defn}
A transitive system gives rise to a single object, $X$, by taking the colimit
\[
X := \mathrm{colim}_{i\in\mathcal{I}} X_i. 
\]
In the context of Heegaard Floer homology, Juh\'asz-Thurston-Zemke show that given two Heegaard diagrams, $\cH$ and $\cH'$ representing a based 3-manifold $(M,z)$, there are canonical maps
\[
\Psi_{\cH \ra \cH'}: \HF^\circ(\cH) \ra \HF^\circ(\cH')
\]
which form such a transitive system, and hence, produce a single invariant, $\HF^\circ(M,z).$ Throughout, $\circ \in \{+,-,\infty, \widehat{\hspace{.3cm}}\}$.

More precisely, given a closed 3-manifold $M$ with a basepoint $z\in M$ (additional basepoints may be considered as well), Juh\'asz-Thurston-Zemke consider a graph $\mathcal{G}_{M,z}$:
\begin{itemize}
    \item The vertices of $\mathcal{G}_{M,z}$ are \emph{isotopy classes of embedded Heegaard diagrams}, i.e. oriented Heegaard surfaces $\Sigma\subset M$, specifying a Heegaard splitting $M=H_\alpha \cup_\Sigma H_\beta$ equipped with a basepoint $z\in \Sigma$ and a pair of $g(\Sigma)$-tuples of $\alpha$- and $\beta$-curves, $\boldsymbol\alpha,\boldsymbol\beta$, such that the $\alpha$-curves bound disjoint disks in $H_\alpha$ and the $\beta$-curves bound disjoint disks in $H_\beta$. These two sets of curves are consider up to isotopy\footnote{Working with isotopy diagrams sidesteps admissibility issues.};
    \item The edges of $\mathcal{G}_{M,z}$ encodes Heegaard moves between embedded Heegaard diagrams. Concretely, there is an edge between vertices $\cH_0$ and $\cH_1$ if $\cH_0$ and $\cH_1$ differ by a handleslide, an ambient stabilization or destabilization, or by an element of $\mathrm{Diff}_0^+(M,z)$, the subgroup of $\mathrm{Diff}^+(M,z)$ consisting of those diffeomorphisms isotopic to the identity.
\end{itemize}
A \emph{weak Heegaard invariant} is a map of graphs $F: \cG_{M,z} \ra \cC$, for some category $\cC$, with the property that for every edge $e$ of $\cG_{M,z}$, $F(e)$ is an isomorphism. In this language, $\HF^\circ$ is a weak Heegaard invariant by \cite{os_holodisks}. Given a path $\eta$ in $\cG_{M,z}$, given by
\[
\cH_0 \xra{e_1} \cH_1 \xra{e_2}\hdots \xra{e_n} \cH_n,
\]
we may define a map $F(\eta)$, as the composition
\[
F(e_n) \circ \hdots \circ F(e_1): F(\cH_0) \ra F(\cH_n).
\]
For a weak Heegaard invariant, this map \emph{may} depend on the choice of path $\eta$. In particular, loops in $\cG_{M,z}$ may have nontrivial monodromy. 

A \emph{strong Heegaard invariant} must satisfy the following additional axioms:
\begin{enumerate}
    \item \emph{Functoriality}: The restrictions of $F$ to the subgraph consisting of only (alpha or beta) handleslides or diffeomorpisms isotopic to the identity are maps of graphs. Moreover, if $e: \cH_0 \ra \cH_1$ is a stabilization, and $e'$ is the corresponding destabilization, then $F(e') = F(e)^{-1}$;
    \item \emph{Commutativity}: For every distinguished rectangle 
    \begin{align*}
        \begin{tikzcd}[ampersand replacement = \&]
            \cH_0 \ar[r,"e"]
             \ar[d,"f"]
             \& 
             \cH_1 
              \ar[d,"f"]\\
            \cH_2 \ar[r,"h"]\&
            \cH_3,
        \end{tikzcd}
    \end{align*}
    in $\cG_{M,z}$, we have that $F(g) \circ F(e) = F(h) \circ F(f)$. See \cite[Definition 2.29]{JTZ_naturality_mapping_class_groups} for an enumeration of distinguished rectangles. 
    \item \emph{Continuity}: If $\cH \in \cG_{M,z}$ and $e$ is an edge corresponding to a diffeomorphism of $\cH$ isotopic to the identity, then $F(e) = \id_{F(\cH)}$. 
    \item \emph{Handleswap Invariance}: For every simple handleswap
    \begin{align*}
        \begin{tikzcd}[ampersand replacement = \&]
            \cH_0 \ar[rd,"e"]
             \& \\
            \cH_2 \ar[u,"g"]\&
            \cH_1 \ar[l,"f"]
        \end{tikzcd}
    \end{align*}
    in $\cG_{M,z}$, we have $F(g)\circ F(f)\circ F(e) = \id_{F(\cH_0)}$. See \cite[Definition 2.31]{JTZ_naturality_mapping_class_groups} for the definition of a simple handleswap.
\end{enumerate}
The fact that $\HF^\circ$ is a \emph{strong} Heegaard invariant is proven in \cite{JTZ_naturality_mapping_class_groups}. We recall the precise statements of these axioms for completeness, and to motivate the definition of a strong bordered Heegaard invariant, though, ultimately, they are rather unimportant to our argument, which simply relies on the fact that the Heegaard Floer invariants of \cite{os_holodisks} are themselves strong in this sense. 

The key result of \cite{JTZ_naturality_mapping_class_groups}, from which naturality follows, is the following theorem.

\begin{thm}{\cite[Theorem 2.38]{JTZ_naturality_mapping_class_groups}}\label{thm: JTZ}
    Let $F: \cG_{M,z} \ra \cC$ be a strong Heegaard invariant. For $\cH, \cH' \in \cG_{M,z}$ and for two oriented paths $\eta, \gamma$ in $\cG_{M,z}$ from $\cH$ to $\cH'$, we have 
    \begin{align*}
        F(\eta) = F(\gamma).
    \end{align*}
\end{thm}
In fact, the complex obtained from $\cG_{M,z}$ by attaching 2-cells along the various polygons given in the definition of a strong Heegaard invariant is simply connected \cite[Remark 2.39]{JTZ_naturality_mapping_class_groups}.\\

It follows from this deep result that given two equivalent Heegaard diagram $\cH$ and $\cH'$, any choice of Heegaard moves relating the two diagrams will give rise to map, which we denote
\[
\Psi_{\cH \ra \cH'}: \CF^\circ(\cH) \ra \CF^\circ(\cH'),
\]
which is well-defined up to homotopy. Moreover, for any $\cH_0, \cH_1, \cH_2 \in \cG_{M,z}$, we have that  
\[
\Psi_{\cH_1 \ra \cH_2}\circ \Psi_{\cH_0 \ra \cH_1} \sim \Psi_{\cH_0 \ra \cH_2}.
\]
This is \cite[Corollary 2.41]{JTZ_naturality_mapping_class_groups}. Moreover, by the \emph{continuity axiom}, 
\[
\Psi_{\cH\ra\cH} \sim \id_{\CF^\circ(\cH)}.
\]
Therefore, $(\{\CF^\circ(\cH)\}_{\cH\in \cG_{M,z}}, \{\Psi_{\cH \ra \cH'}\}_{\cH, \cH'\in \cG_{M,z}})$ form a transitive system in the sense of \Cref{def:transitive_sys}, giving rise to the well-defined object
\[
\HF^\circ(M, z),
\]
\emph{the} Heegaard Floer homology invariant of $(M, z)$.

In the case of the Heegaard Floer functors, \Cref{thm: JTZ} may be reframed as follows: consider the map 
\[
\rho_{M,z}: \pi_1 (\cG_{M,z}) \ra H_*\End(\CF^\circ(M,z)),
\]
taking a loop $\eta$ of Heegaard moves based at $\cH_0$ to the class $[\Psi(\eta)]$. \Cref{thm: JTZ} says that this representation is trivial. We will make use of this formulation of their result in the proof of Proposition \ref{prop: bordered invariants have trivial monodromy}.  \\

We now turn to the bordered setting. Given a bordered 3-manifold $M$ with one boundary component (which contains the basepoint $z$), $(M, z)$ can be represented by a bordered Heegaard diagram, which is a tuple $(\Sigma, \bm \alpha^a, \bm \alpha^c, \bm \beta, z)$, where $\Sigma$ is a genus $g$ surface with a single boundary component, $\bm \beta$ is a collection of $g$ pairwise disjoint circles in the interior of $\Sigma$, $\bm \alpha^a$ is a collection of $2k$ pairwise disjoint arcs in $\Sigma$ with boundary in $\partial \Sigma$, and $\bm \alpha^c$ a collection of $(g-k)$ pairwise disjoint circles in the interior of $\Sigma$ disjoint from $\bm \alpha^a$. 

Any two bordered Heegaard diagrams for equivalent bordered 3-manifolds can be related by a sequence of elementary bordered Heegaard moves, which consist of the following:
\begin{enumerate}
    \item isotopies of the $\alpha$- and $\beta$-curves not crossing the boundary of $\Sigma$;
    \item handleslides of the $\alpha$ curves over $\alpha$ circles\footnote{I.e., both $\alpha$ arcs and circles can slide over $\alpha$-circles, but no $\alpha$ curves can slide over $\alpha$-arcs.} and of the $\beta$ circles over $\beta$ circles;
    \item stabilizations and destabilizations in the interior of $\Sigma$.
\end{enumerate}
Therefore, just as in the closed case, we may consider the graph $\mathcal{G}_M^\partial$, defined as follows:
\begin{itemize}
    \item Vertices of $\mathcal{G}_M^\partial$ are isotopy classes of bordered Heegaard diagrams (again, here we mean that the $\alpha$ and $\beta$ curves are considered up to isotopy fixing the boundary;
    \item The edges of $\mathcal{G}_M^\partial$ are given by handleslides, ambient stabilizations and destabilizations, and elements of $\mathrm{Diff}^+_0(M,\partial M)$ (i.e.  those isotopic to the identity). Note that we include handleslides of $\alpha$-arcs over $\alpha$-circles.
\end{itemize}
It is established in \cite{LOT_bordered_HF} that $\CFDh(-)$ and $\CFAh(-)$ are both weak Heegaard invariants, i.e. to every elementary bordered Heegaard move corresponding to an edge $e$ in $\cG_M^\partial$ from $\cH_0$ to $\cH_1$, they define a homotopy equivalence $\Psi^D(e): \CFDh(\cH_0) \ra \CFDh(\cH_1)$ and $\Psi^A(e): \CFAh(\cH_0) \ra \CFAh(\cH_1).$ In particular, there are well-defined representations
\[
\rho_M^D: \pi_1(\cG_M^\partial) \ra H_*\End^\cA(\CFDh(M)),
\]
\[
\rho_M^A: \pi_1(\cG_M^\partial) \ra H_*\End_\cA(\CFAh(M)),
\]
given by taking a loop of bordered Heegaard moves $\eta$ to the composite of the maps induced by those moves. A priori, it is unclear whether these representations are trivial. 

We shall not define strong bordered invariants, i.e. we will not pin down a minimal set of relations which must be checked to determined whether a map of graphs $F: \cG_{M, z}^\partial \ra \cC$ gives rise to transitive system. Rather, we show that any loop of elementary bordered Heegaard diagrams (between embedded diagrams) must induce the identity morphism. 

According \cite[Proof of Corollary 4.2]{lipshitz2013faithful}, the functor $\widehat{CFAA}(\mathbb{I})$ induces an equivalence of categories; therefore we only need to establish naturality for one of $\CFDh$ or $\CFAh$.  We shall work with type D structures. \\

Before proceeding, we make a few algebraic observations. Let $(M, \delta^1)$ be a type $D$ structure over an dg algebra $\cA$. Let $\overline{M} = \Hom_\F(M, \F)$ and let $\overline{\delta}^1: \overline{M} \ra \overline{M} \otimes \cA$ be induced by the transpose of $\delta^1$. This makes $(\overline{M}, \overline{\delta}^1)$ a (right) type D structure over $\cA$, which is the dual type D structure to $(M, \delta^1)$. If $M$ and $N$ are two type D structures over $\cA$, then the chain complex $\Mor^\cA(M, N)$ of module homomorphisms from $M$ to $N$ is isomorphic to $\overline{M}\boxtimes_\cA \cA \boxtimes_\cA N$ \cite[Proposition 2.7]{LOT_HF_as_morphism}. We will write $M^\vee: = \overline{M}\boxtimes_\cA \cA$, which is the type A structure dual to the type D structure $M$. These isomorphisms respect composition in the following sense. 

\begin{lem}\label{lem:composition_to_boxtensor}
    Suppose that $M$, $N$, and $P$ are type $D$ structures over $\cA$ and that $f: N \ra P$ is a module homomorphism. Then, the diagram
    \begin{align*}
        \begin{tikzcd}[ampersand replacement = \&]
            \Mor^\cA(M, N) \ar[r,"f\circ {-}"] \ar[d,"\cong"] \&
            \Mor^\cA(M, P) \ar[d,"\cong"]\\
            M^\vee\boxtimes_\cA N \ar[r,"\bI \boxtimes f"] \& 
            M^\vee \boxtimes_\cA P
        \end{tikzcd}
    \end{align*}
    commutes.
\end{lem}
\begin{proof}
    A module map $g$ is \emph{basic} if there are generators $x$ and $y$ so that $g(x) = \sigma y$ for some algebra element $\sigma \in \cA$ and vanishes for all other generators. The vertical isomorphisms take such a morphism $g$ to the element $\overline{x}\otimes \sigma \otimes y \in M^\vee \boxtimes N = \overline{M}\boxtimes_\cA \cA \boxtimes_\cA N$. Since every morphism of $\Mor^\cA(N, P)$ is a linear combination of basic morphisms, we may assume that $f$ is basic, i.e. $f(x) = \sigma y$ and vanishes for all other generators of $N$. Suppose $g \in \Mor^\cA(M, N)$ is basic as well, and composes nontrivially with $f$, so that $g(v) = \rho x$. The clockwise composition takes $g$ to $f \circ g$ (which is also basic, $v \mapsto \rho \sigma y$), and hence $f \circ g$ to the element $\overline{v}\otimes \rho \sigma \otimes y.$ The counterclockwise composition takes $g$ to $\overline{v}\otimes \rho \otimes x$; $(\bI \boxtimes f)(\overline{v}\otimes \rho \otimes x)$ is precisely $\overline{v}\otimes  \rho \sigma \otimes y$. 
\end{proof}

\begin{lem}\label{lem: M to MvM is injective}
    Let $\End^\cA(M) := \Mor^\cA(M, M)$. The map 
    \begin{align*}
        H_*(\End^\cA(M)) \ra H_*(\End_\F(M^\vee \boxtimes M)),
    \end{align*}
    induced by $ f \mapsto \bI \boxtimes f$ is injective.
\end{lem}
\begin{proof}
    According to the previous lemma, under the identification $M^\vee \boxtimes M \cong \End^\cA(M)$, this is equivalent to a map 
    \begin{align*}
        \Phi: H_*(\End^\cA(M)) \ra H_*(\End_\F(\End^\cA(M))),
    \end{align*}
    induced by the map $f \mapsto f \circ (-).$ So, suppose that $\Phi([f]) = 0$, i.e. assume that for every $g \in \End^\cA(M)$, $f\circ g$ is nullhomotopic. In particular, when $g = \id_M$, we have that $f \circ \id_M  = f$ is nullhomotopic. Hence, $[f] = 0$, and thus, $\Phi$ is injective.
\end{proof}

It will be useful to have a slightly more general version of the previous lemma. 

\begin{cor}\label{lem:M to N M is injective}
    Let $M$ be a bounded type $D$ structure and let $N$ be a type A structure which is homotopy equivalent to $M^\vee$. Then, map 
    \begin{align*}
        H_*(\End^\cA(M)) \ra H_*(\End_\F(N \boxtimes_\cA M))
    \end{align*}
    induced by $f \mapsto \bI_N \boxtimes f$ is injective. 
\end{cor}
\begin{proof}
    Fix a homotopy equivalence $\varphi: N \ra M^\vee$ with inverse $\psi: M^\vee \ra N$. There is a diagram:
    \begin{align*}
        \begin{tikzcd}[ampersand replacement = \&]
            \End^\cA(M) \ar[r]\ar[rd] \& \End_\F(N \boxtimes_\cA M) \ar[d] \\
            \& \End_\F(M^\vee \boxtimes_\cA M).
        \end{tikzcd}
    \end{align*}
    The horizontal arrow sends $f$ to $\bI_N \boxtimes f$ while the diagonal arrow sends $f$ to $\bI_{M^\vee} \boxtimes f$. The horizontal arrow takes a map $f \boxtimes g$ to $(\psi \boxtimes \bI_M)\circ (f \boxtimes g)\circ (\varphi \boxtimes \bI_M)$. This diagram is homotopy commutative by \cite[Lemma 2.32]{LOT_bordered_HF}. Hence, when we pass to homology, the diagram commutes. By \Cref{lem: M to MvM is injective}, the diagonal arrow induces an injection in homology, from which it follows that the horizontal arrow also induces an injection in homology as well. This proves the claim.
\end{proof}

With these preliminaries out of the way, way may proceed to the proof of naturality.

\begin{prop}\label{prop: bordered invariants have trivial monodromy}
    Let $\eta, \gamma$ be paths in $\cG_M^\partial$ between embedded bordered Heegaard diagrams $\cH$ and $\cH'$. Then, 
    \[
        \Psi^D(\eta) \sim \Psi^D(\gamma).
    \]
\end{prop}
\begin{proof}
    It suffices to prove that for any loop $\eta$ in $\cG_M^\partial$ based at $\cH$
    \[
    \rho_{\cG_M^\partial}(\eta):= [\Psi^D(\eta)] = [\bI_{\CFDh(\cH)}].
    \]
    To do so, consider the double of $M$, which we denote $D(M) := (-M) \cup M.$ Fix a bordered Heegaard diagram $\cH_0$ for $-M$. Since we are working with isotopy classes of bordered Heegaard diagrams we assume that $\cH$ is only provincially admissible (every periodic domain in the interior of $\Sigma$ has both positive and negative multiplicities). Though, in order to invoke the pairing theorem, we make the stronger assumption that $\cH_0$ is admissible (every periodic domain in $\Sigma$, including those which go out to the boundary, have both positive and negative multiplicities).
    
    The fixed diagram $\cH_0$ gives rise to a map 
    \[
    \cF: \cG_{M}^\partial \ra \cG_{D(M),z}.
    \]
    On vertices, $\cF$ acts as $\cH \mapsto \cH_0 \cup \cH$. We define $\cF$ on edges as follows: let $e$ be an edge of $\cG_{M}^\partial$ corresponding to a bordered Heegaard move from $\cH$ to $\cH'$; define $\cF(e)$ to be the Heegaard move from $\cH_0 \cup \cH$ to $\cH_0 \cup \cH'$ given by performing $e$ on the glued diagram. The map $\cF$ induces a map
    \[
    \cF_*: \pi_1(\cG_{M}^\partial,\cH) \ra \pi_1(\cG_{D(M),z}, \cH_0 \cup \cH).
    \]
    There is another natural map
    \begin{align*}
        \Phi: \End^\cA(\CFDh(\cH)) \ra &\End_\F(\CFh(\cH_0 \cup \cH))\\  \simeq& \End_\F(\CFAh(\cH_0)\boxtimes \CFDh(\cH)), 
    \end{align*}
    given by $f \mapsto \bI_{\CFAh(\cH_0)}\boxtimes f.$ We emphasize here that since $\cH_0$ is assumed to be strongly admissible, pairing here is indeed permissible. These two maps are compatible with the monodromy representations in the sense that the diagram
    \begin{align}\label{diagram:naturality}
        \begin{tikzcd}[ampersand replacement=\&]
            \pi_1(\cG_{M}^\partial,\cH) \ar[d,"\cF_*"]\ar[rr,"\rho^D_M"] \&\&
             H_\ast \mathrm{End}^\cA(
    \widehat{CFD}(\cH))\ar[d,"\Phi"]\\
            \pi_1(\cG_{D(M),z}, \cH_0 \cup \cH) \ar[rr,"\rho_{D(M),z}"] \&\&
            H_\ast\mathrm{End}(\CFAh(\cH_0)\boxtimes \CFDh(\cH)).
        \end{tikzcd}
    \end{align}
    is commutative. 

    To see that this diagram commutes, let $e$ be an edge in $\cG_{M}^\partial$ corresponding to an elementary bordered Heegaard move between some $\cH$ and $\cH'$. By \cite[Lemma 5.6, Section 10]{HL_Inv_bordered_floer}, there is a homotopy commutative diagram
    \begin{align}\label{diagram:bordered moves pairing}
    \begin{tikzcd}[column sep = huge,ampersand replacement=\&]
        \CFAh(\cH_0) \boxtimes \CFDh(\cH) \ar[rr,"\bI_{\CFAh(\cH_0)}\boxtimes \Psi^D(e)"] \ar[d,"\simeq"] \& \&
        \CFAh(\cH_0) \boxtimes \CFDh(\cH') \ar[d,"\simeq"]\\
        \CFh(\cH_0\cup\cH) \ar[rr,"\Psi(\cF(e))"] \& \&
        \CFh(\cH_0\cup \cH')
    \end{tikzcd}
    \end{align}
    where the vertical arrows are given by the pairing theorem. Let $\eta \in \pi_1(\cG_M, \cH)$ be a loop given by
    \[
    \cH \xra{e_1} \cH_1 \xra{e_2} \hdots \xra{e_{n-1}} \cH_{n-1} \xra{e_n} \cH.
    \]
    The clockwise composition in Diagram (\ref{diagram:naturality}) is 
    \[
        \phi(\rho_{M}^D(\eta)) = [\bI_{\CFAh(\cH_{n-1})}\boxtimes \Psi^D(e_n) \circ \hdots \circ \bI_{\CFAh(\cH)}\boxtimes \Psi^D(e_1)].
    \]
    By the commutativity of Diagram (\ref{diagram:bordered moves pairing}), this is equal to
    \[
        [\Psi(\cF(e_n))\circ \hdots \circ \Psi(\cF(e_1))] = \rho_{D(M),z}(\cF_*(\eta)),
    \]
    as claimed. \\

    By \cite[Theorem 2.38]{JTZ_naturality_mapping_class_groups}, the counterclockwise composition in Diagram (\ref{diagram:naturality}) is constant; hence, the clockwise composition is as well. In particular, the image of $\rho_{M}^D$ is contained in $\mathfrak{f}^{-1}(1).$ Therefore, it suffices to prove that the map $\mathfrak{f}$ is injective. 

    According to \Cref{lem: M to MvM is injective}, the map 
    \begin{align*}
        H_*(\End^\cA(\CFDh(\cH))) \ra H_*(\End_\F(\CFDh(\cH)^\vee \boxtimes \CFDh(\cH))),
    \end{align*}
    is injective. But, by \cite[Theorem 2]{LOT_HF_as_morphism}, $\CFDh(\cH)^\vee$ is homotopy equivalent to $\CFAh(-\cH)$, which is a bordered Heegaard diagram for $-M$. But, then, by \cite[Theorem 2]{LOT_bordered_HF}, $\CFAh(-\cH)$ is homotopy equivalent to $\CFAh(\cH_0)$. Hence, by \Cref{lem:M to N M is injective}, the map 
    \begin{align*}
        \Phi: H_*(\End^\cA(\CFDh(\cH))) \ra H_*(\End_\F(\CFAh(\cH_0) \boxtimes \CFDh(\cH)))
    \end{align*}
    is injective, completing the proof.
\end{proof}

\begin{rem}
    There is a minor subtlety in the commutativity of Diagram \ref{diagram:bordered moves pairing} when the edge $e$ corresponds to a handleslide involving one of the $\alpha$ curves. In the bordered setting, this corresponds to a situation in which one must count holomorphic triangles in a diagram with two different sets of $\alpha$-arcs. A general theory of pairing triangle in this setting does not currently exist in the literature. Regardless, in the special case of handleslides, the square does commute, as discussed in \cite[Section 10]{HL_Inv_bordered_floer}. 
\end{rem}

It follows immediately from \Cref{prop: bordered invariants have trivial monodromy} that if $\cH$ and $\cH'$ are equivalent, embedded bordered Heegaard diagrams, we can define
\[
\Psi^D_{\cH \ra \cH'}:= \Psi^D(\eta)
\]
for any path $\eta$ in $\cG_{M}^\partial$; up to homotopy, this map does not depend on $\eta$.
    
\begin{thm}\label{thm:bordered naturality one boundary}
    The collection $(\{\CFDh(\cH)\}_{\cH\in \cG_M^\partial}, \{\Psi^D_{\cH \ra \cH'}\}_{\cH, \cH'\in \cG_{M}^\partial})$ forms a transitive system. In particular, the type D bordered Floer invariants are natural. 
\end{thm}
\begin{proof}
    Part (1) of \Cref{def:transitive_sys} is immediate from \Cref{prop: bordered invariants have trivial monodromy}. If $\cH_0$, $\cH_1$, and $\cH_2$ are equivalent embedded bordered Heegaard diagrams, then by \Cref{prop: bordered invariants have trivial monodromy}
    \[
        (\Psi^D_{\cH_1 \ra \cH_2}\circ \Psi^D_{\cH_0 \ra \cH_1})\circ (\Psi^D_{\cH_0 \ra \cH_2})^{-1} \simeq \bI_{\CFDh(\cH_0)},
    \]
    from which Part (2) of \Cref{def:transitive_sys} follows.
\end{proof}

\begin{thm}
    The collection $(\{\CFAh(\cH)\}_{\cH\in \cG_M^\partial}, \{\Psi^A_{\cH \ra \cH'}\}_{\cH, \cH'\in \cG_{M}^\partial})$ forms a transitive system. In particular, the type A bordered Floer invariants are natural. 
\end{thm}
\begin{proof}
    This follows from \Cref{prop: bordered invariants have trivial monodromy} and the fact that $\CFAAh(\bI)\boxtimes (-)$ is an equivalence of categories \cite{lipshitz2013faithful}.
\end{proof}

The case of multiple boundary components follows from the single boundary case. We recall the definition of a \emph{strongly bordered 3-manifold:}
\begin{defn}
    Fix connected surfaces $(F_1, D_1, z_1'), \hdots, (F_k, D_k, z_k')$ with preferred disks $D_i$ with basepoints $z_i'$ on their boundaries. A \emph{strongly bordered} 3-manifold $M$ with boundary components $F_1, \hdots, F_k$ is a 3-manifold $Y$ equipped with the following data: 
    \begin{enumerate}
        \item Preferred disks $\Delta_1, \hdots \Delta_k$ and basepoints $z_i \in \partial \Delta_i$ in each of its boundary components;
        \item A diffeomorphism $\psi:\amalg_i (F_i, D_i, z_i') \ra (\partial Y, \amalg_i \Delta_i, \amalg_i z_i)$;
        \item A framed tree, $\mathbf{z}$, consisting of an interior basepoint $z$ and arcs connecting $z$ to each $z_i$.
    \end{enumerate}
\end{defn}
Compare with \cite[Definition 1.3]{LOT_bimodules} or \cite[Section 2]{hanselman_graph}.\\

Any strongly bordered 3-manifold can be represented by an \emph{arced bordered Heegaard diagram with $k$ boundary components} \cite[Definition 2.2]{hanselman_graph}.  Given a strongly bordered 3-manifold $(M,\mathbf{z})$ with $k$ boundary components, define a graph, $\cG_{M,\mathbf{z}}^{k\partial}$ as follows:
\begin{itemize}
    \item Vertices of $\cG_{M,\mathbf{z}}^{k\partial}$ are embedded strongly arced bordered Heegaard diagrams;
    \item Edges are arced bordered Heegaard moves: isotopies, handleslides, ambient stabilizations and destabilizations, and elements of $\Diff^+_0(M, \partial M, \mathbf{z}).$
\end{itemize}

To a strongly bordered 3-manifold $M$ with $k$ boundary components, the techniques from \cite{LOT_bimodules} determine a \emph{multimodule},$\CFDh^k(M)$, over $\cA_1 \otimes \hdots \otimes \cA_k$, where $\cA_i := \cA(-F_i)$ as follows. Fix an arced bordered Heegaard diagram, $\cH$, for $M$. Let $\cH_{dr}$ be the bordered Heegaard diagram (with a single boundary component) obtained by deleting a neighborhood of the tree $\mathbf{z}$ and let $M_{dr}$ be the 3-manifold associated to $\cH_{dr}$. For each arc, $\lambda_i$, $i\in \{2, \hdots, k\}$, in $\mathbf{z}$ from $z$ to $z_i$, fix a pair of points $p_i^\pm$ in $\partial \nu(\mathrm{int}(\lambda))$ (one in each component). See \cite[Definition 5.5]{LOT_bimodules}. The drilling operation gives rise to a simplicial map
\[
\cD: \cG_{M,\mathbf{z}}^{k\partial} \ra \cG_{M_{dr},z_k^+}^{\partial},
\]
in the obvious way. The choice of basepoint in $M_{dr}$ is unimportant, but we fix it for concreteness. Drilling has a natural inverse, \emph{filling}, given by attaching $(k-1)$ bands to $\cH_{dr}$ along $p_i^\pm$ \cite[Definition 5.8]{LOT_bimodules}. In particular, $\cD$ is an isomorphism of graphs. 

Note, that $M_{dr}$ has boundary $\#^i F_i$. Let $\cA^\#:=\cA(\#^i -F_i)$. There is a natural projection 
\[
\rho: \cA^\# \ra \cA_1 \otimes \hdots \otimes \cA_k,
\]
which gives rise to an induction functor 
\[
\mathbf{Induct}^{\cA_1,\hdots,\cA_k}: \prescript{\cA^\#}{}{\mathbf{Mod}} \ra \prescript{\cA_1,\hdots,\cA_k}{}{\mathbf{Mod}}.
\]
The multimodule associated to $M$ is given by
\[\CFDh^k(M):= \mathbf{Induct}^{\cA_1,\hdots,\cA_k}(\CFDh(M_{dr})).\]
Since the drilling functor is an isomorphism of categories, invariance of the mutimodules follows immediately from invariance of the usual bordered modules. As we now prove, naturality follows easily as well. 

\begin{prop}\label{prop: multimodule invariants have trivial monodromy}
    Let $\eta, \gamma$ be paths in $\cG_{M,\mathbf{z}}^{k\partial}$ from $\cH$ to $\cH'$. Then, 
    \[
        \Psi(\eta) \sim \Psi(\gamma).
    \]
\end{prop}
\begin{proof}
    As in the single boundary case, it suffices to prove that the monodromy representation
    \[
    \rho^{D^k}_{M,\mathbf{z}}: \pi_1 (\cG_{M,\mathbf{z}}^{k\partial}) \ra H_*\End^{\cA_1, \hdots, \cA_k}(\CFDh^k(\cH)),
    \]
    is trivial. Consider the following diagram: 
    \begin{align}\label{diagram:naturality-multi}
        \begin{tikzcd}[ampersand replacement=\&]
            \pi_1 (\cG_{M,\mathbf{z}}^{k\partial}) \ar[rr,"\rho^{D^k}_{M,\mathbf{z}}"] \&\&
             H_\ast \mathrm{End}^{\cA_1, \hdots, \cA_k}(
    \widehat{CFD}(\cH))\\
            \pi_1(\mathcal{G}_{M_{dr},z}) \ar[rr,"\rho_{M_{dr},z}"]\ar[u] \&\&
            H_\ast\End^{\cA^\#}(\CFDh(\cH_{dr})) \ar[u],
        \end{tikzcd}
    \end{align}
    where the left vertical arrow is induced by the filling construction and the right vertical arrow is given by the induction functor. The square commutes by the definition of the multimodules (the maps associated to arced bordered Heegaard moves are, by definition, obtained by applying the induction functor to the maps induced by the corresponding Heegaard moves on the drilled diagram \cite[Proposition 6.5]{LOT_bimodules}.) By \Cref{prop: bordered invariants have trivial monodromy}, the counterclockwise composition is constant, hence, the clockwise composition is constant as well. But, the left arrow is an isomorphism; hence, $\rho_{M}^{D^k}$
    is the trivial representation.
\end{proof}

Naturality then follows just as in the single boundary case. 

\begin{thm}\label{thm: naturality for multimodules}
    The collection $(\{\CFDh^k(\cH)\}_{\cH\in \cG_{M,\mathbf{z}}^{k\partial}}, \{\Psi_{\cH \ra \cH'}\}_{\cH, \cH'\in \cG_{M,\mathbf{z}}^{k\partial}})$ forms a transitive system. In particular, the type D multimodule bordered Floer invariants are natural. 
\end{thm}
The case of other type of multimodules follows by box-tensoring with $\CFAAh(\bI)$.

\section{Extensions and basepoints}\label{sec:basepoints and extensions}

In this section, we consider two methods of extracting $\CFK_\cR(K)$ from $\CFDh(\KC)$: the first method involves working over an extension of the usual torus algebra (as in \cite{hanselman2023bordered}) while the second involves adding in an additional basepoint (as in \cite{kang_bordered_involutive_HFK}). 

\subsection{Extensions of the Torus Algebra}

The bordered invariants for the complement of a knot $K\subset S^3$ contain equivalent data to the knot Floer complex for $K$ over the ring $\cR$. The correspondence passes through the formalism of \emph{extended type D structures}, which we briefly describe now. The torus algebra has an extension, given by
\begin{equation}\label{eq:ext_torus-alg}
\widetilde{\cA} =
  \xymatrix{
\iota_0\,\bullet\ar@/^1pc/[r]^{\rho_1}\ar@/_2pc/[r]_{\rho_3} 
    & \circ\,\iota_1\ar@/^1pc/[l]_{\rho_2}\ar@/_2pc/[l]_{\rho_0}}
    /(\rho_{i+1}\rho_i = \rho_{0}\rho_{1}\rho_{2}\rho_{3}\rho_{0}=0), \hspace{1cm} i \in \Z/4.
\end{equation}
An extended type D structure is a curved module over the extended torus algebra $\tilde{\mathcal{A}}$, i.e. it is a finitely generated vector space $V=V_0 \oplus V_1$ over $\F$, on which the subalgebra of idempotents $\mathcal{I}=\mathrm{Span}(\iota_0,\iota_1)$ acts by projections to $V_0$ and $V_1$, respectively, together with a differential
\[
\delta^1: V\rightarrow \tilde{\mathcal{A}}\otimes_{\mathcal{I}} V
\]
such that the squared differential
\[
V\xrightarrow{\delta^1}\tilde{\mathcal{A}}\otimes V\xrightarrow{\id_{\tilde{\mathcal{A}}}\otimes \delta^1} \tilde{\mathcal{A}}\otimes_{\mathcal{I}}\tilde{\mathcal{A}}\otimes_{\mathcal{I}} V\xrightarrow{\tilde{\mu}\otimes \id_V}\tilde{\mathcal{A}} \otimes_{\mathcal{I}} V
\]
is given by $\partial^2 = \mathbb{U}\otimes \mathrm{id}_V$. Here, $\mathbb{U}$ denotes the central element of $\tilde{\mathcal{A}}$, defined as 
\[
\mathbb{U} =  \rho_{0123}+\rho_{1230}+\rho_{2301}+\rho_{3012} = \sum_{\vert I \vert =4}\rho_I.
\]
We say that an extended type $D$ structure $V$ is an \emph{extension} of a type D structure $V'$ if 
\[
V' \simeq \cA \otimes_{\cAt} V
\]
as differential modules. Conversely, a type D structure $V'$ is \emph{extendable} if there exists such a $V.$

By counting additional rigid holomorphic curves which are allowed to cover the region containing $\rho_0$ at most once, one can define an extended type D structure $\CFDt(Y)$ for bordered 3-manifolds $Y$ with torus boundary. $\CFDt(Y)$ is genuinely an extension of $\CFDh(Y)$, i.e.
\[
\CFDh(Y) \simeq \cA \otimes_{\cAt} \CFDt(Y).
\]
Moreover, the fact that this process defines a well-defined extended type D structure proves that type D structures for bordered 3-manifold with torus boundary are always extendable. Moreover, by \cite[Proposition 4.16]{hanselman2023bordered}, when an extension exists, it is unique up to homotopy equivalence. \\

In \cite[Chapter 11]{LOT_bordered_HF}, \LOT show that there is a correspondence 
\[
\{\text{reduced models for }\CFK_\cR(S^3,K)\} \iff \{\text{reduced models for } \CFDt(S^3\smallsetminus K)\}.
\]
Given a reduced module for $\CFK_\cR(S^3,K)$, they show that the holomorphic curves defining the differential on $\CFK_\cR(S^3,K)$ determine those defining the differential for $\CFDt(S^3\smallsetminus K)$. Conversely, given $\CFDt(S^3\smallsetminus K)$, the knot Floer complex can be recovered by tensoring with the extended type $A$ structure, $\CFAt(T_\infty, \nu)$, for an $S^1$-fiber in the $\infty$-framed solid torus, which has a single generator, $t$ and $\cA_\infty$-operations
\[
m_{2+n}(t, \rho_{3}, \rho_{23}, \hdots, \rho_{23}, \rho_2) = U^n t, \hspace{1cm} m_{2+n}(t, \rho_{1}, \rho_{01}, \hdots, \rho_{01}, \rho_0) = V^n t.
\]
This diagram is shown in Figure \ref{fig:T_inf_hd}. Our convention is that the basepoint $w$ is recorded by the variable $U$ and basepoint $z$ is recorded by the variable $V$.

\begin{figure}[h]
    \centering
    \includegraphics[width = 0.3 \textwidth]{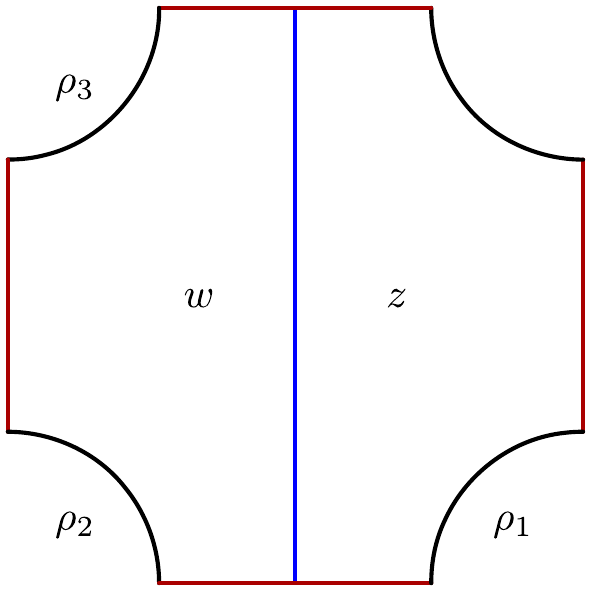}
    \caption{The doubly pointed bordered Heegaard diagram for the longitudinal knot in the solid torus.}
    \label{fig:T_inf_hd}
\end{figure}

The passage from $\CFK_\cR(K)$ to $\CFDh(\KC)$ is described in \cite[Theorem 11.26]{LOT_bordered_HF}. 

We also note that it will actually suffice to work with a slightly smaller algebra than the full extended algebra. The $UV = 0$ truncation of the knot Floer complex corresponds to counting holomorphic curves which can cover either of the basepoints, but not both. On the bordered side, this means we need to allow curves which are asymptotic to the chords $\rho_{01}$ and $\rho_{23}$, but not chords containing $\rho_{30}$. Therefore, in order to recover $\CFK_\cR(S^3,K)$, it suffices to work with the  quotient algebra 
\[
\tilde{\cA}_{0,3} := \tilde{\cA}/\langle \rho_{30}\rangle,
\]
and bordered modules over this algebra. In particular, we will often make use of the following object.

\begin{defn}\label{def:half extended torus module}
    Define $\CFAnuhalf$ to be the type A structure over $\cAt_{0,3}$ generated by a single element $x$ and $\cR$-equivariant $\cA_\infty$ operations 
    \[
    m_{2+n}(t, \rho_{3}, \rho_{23}, \hdots, \rho_{23}, \rho_2) = U^n t, \hspace{1cm} m_{2+n}(t, \rho_{1}, \rho_{01}, \hdots, \rho_{01}, \rho_0) = V^n t.
    \]
    Equivalently, $\CFAnuhalf$ is the type A structure with $\cA_\infty$ operations defined by counting holomorphic curves which may cover either of the two basepoints determining $\nu$, but not both. Curves covering $w$ are weighted by powers of $U$ and those covering $z$ are weighted by powers of $V$.
\end{defn}
\begin{rem}
    We note that $\CFAnuhalf$ has the same $\cA_\infty$ operations as $\CFAt(T_\infty, \nu)$, but we emphasize that the two modules have different underlying algebras. 
\end{rem}

\subsection{Free basepoints}

Working with these extended modules allow us to pass back and forth between $\CFK_\cR$ and $\CFDh$. But, there are some significant draw-backs. We would like to be able to make use of some of the features of bordered Floer theory (such as bimodules and the morphism pairing theorem) which are incompatible with extensions.

Therefore, as in \cite{kang_bordered_involutive_HFK}, we will at times try to extract information about $\CFK_\cR(K)$ from the unextended $\CFDh(\KC)$ by working with bordered modules with additional basepoints; in particular, we will make use of the triply pointed bordered Heegaard diagram $\bX$ in Figure \ref{fig:x_hd}. This diagram represents an $S^1$-fiber, $\nu := S^1 \times \pt$ in the infinity framed solid torus $S^1\times D^2$ with a free basepoint, $p$, on its boundary (by free, we mean that $p$ does not intersect $\nu$.) 

Let $\CFAx$ be the type A structure over $\cA(T^2)$ freely generated over $\cR$ by tuples of intersection points in $\bX$ with operations 
\[
    m_{n+1}(x, a(\rho_{I_1}), \hdots, a(\rho_{I_n})) = \sum_{y \in\mathfrak{G}(\bX)} \sum_{\substack{B\in \pi_2(x, y),\\\ind(B, \Vec{\rho})}} \# \left(\cM^B(x, y, \rho_{I_1}, \hdots, \rho_{I_n})\right) U^{n_w(B)}V^{n_x(B)} y.
\]
Note that these curves are allowed to cover the interior basepoints defining $\nu$, but not the boundary basepoint. By the pairing theorem,
\[
\CFAx\boxtimes \CFDh(\KC) \simeq \CFK_\cR(S^3,K, p),
\]
where $p$ is a free basepoint. 

\begin{figure}[h]
    \centering
    \includegraphics[width = 0.3 \textwidth]{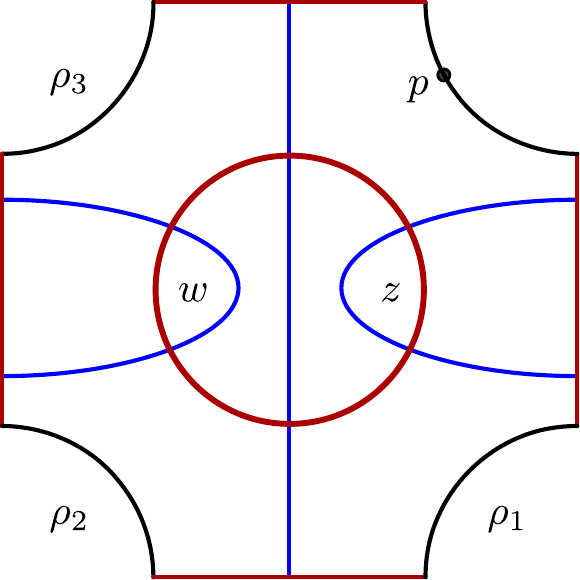}
    \caption{The triply pointed bordered Heegaard diagram for the longitudinal knot in the solid torus with a free basepoint.}
    \label{fig:x_hd}
\end{figure}

Working with the triply pointed diagram $\bX$, as opposed to the usual doubly pointed diagram for $(T_\infty,\nu)$, offers some technical advantages. As noted above, $\CFAh_\cR(\bX)\boxtimes \CFDh(\KC) \simeq \CFK_\cR(S^3,K,p)$, which according to \cite[Lemma 7.1]{zemke_linkcob}, is homotopy equivalent to $\CFK_\cR(S^3,K)^{\oplus 2}$. Hence, by tensoring with $\CFAh_\cR(\bX)$, we can study $\CFK_\cR(S^3,K)$ without having to work over the extended (or half extended) torus algebra. The extra basepoint is a small price to pay. 

Secondly, the diagram $\bX$ behaves better than $(T_\infty, \nu)$ when glued to the Auroux-Zarev piece. Any singly pointed alpha-bordered Heegaard diagram $\cH$ has the property that 
\[
    \overline{\cH}^\beta \cup \overline{\AZ} \sim \cH^\alpha.
\]
This fact is a key ingredient in the proof of the morphism version of the pairing theorem as well as in the construction of the bordered involution. This property does not hold for doubly pointed bordered Heegaard diagrams, and in particular, the standard diagram for $(T_\infty, \nu)$ is not related to the diagram obtained by gluing the Auroux-Zarev piece to its conjugate (the bordered modules are not even homotopy equivalent). However, by \cite[Lemma 3.1]{kang_bordered_involutive_HFK}, we do have that 
\[
\overline{\bX} \cup \overline{\AZ} \sim \bX,
\]
up to swapping the basepoints, which makes working with $\bX$ more flexible. Crucially, this fact will allow us to utilize \cite{cohen2023composition}, and realize the composition map 
\begin{align*}
    \End^\cA(\CFDh(\KC)) \otimes \Mor^\cA(\CFDh(\KC), \CFDx) \ra \Mor^\cA(\CFDh(\KC), \CFDx)
\end{align*}
as a cobordism map. 

In order to bridge the two perspectives (extensions vs. extra basepoints), we introduce a bimodule version of $\bX$. A Heegaard diagram for the complement of $\nu$ in $S^1\times D^2$ can be obtained from $\bX$ in the standard way: we stabilize $\bX$ with a punctured tube whose feet are attached near $w$ and $z$, add an extra $\beta$ circle and a new pair of $\alpha$ arcs to obtain a diagram which represents the complement of $\nu$ in $S^1 \times D^2$. Label the Reeb chords on the new boundary by $\{\sigma_i\}_{0\le i \le 3}$ and add a basepoint to the region whose boundary contains $\sigma_0$. We will denote this diagram (see Figure \ref{fig:y_arc}) by $\bY$. 

\begin{figure}
    \centering
    \includegraphics[width = 0.65 \textwidth]{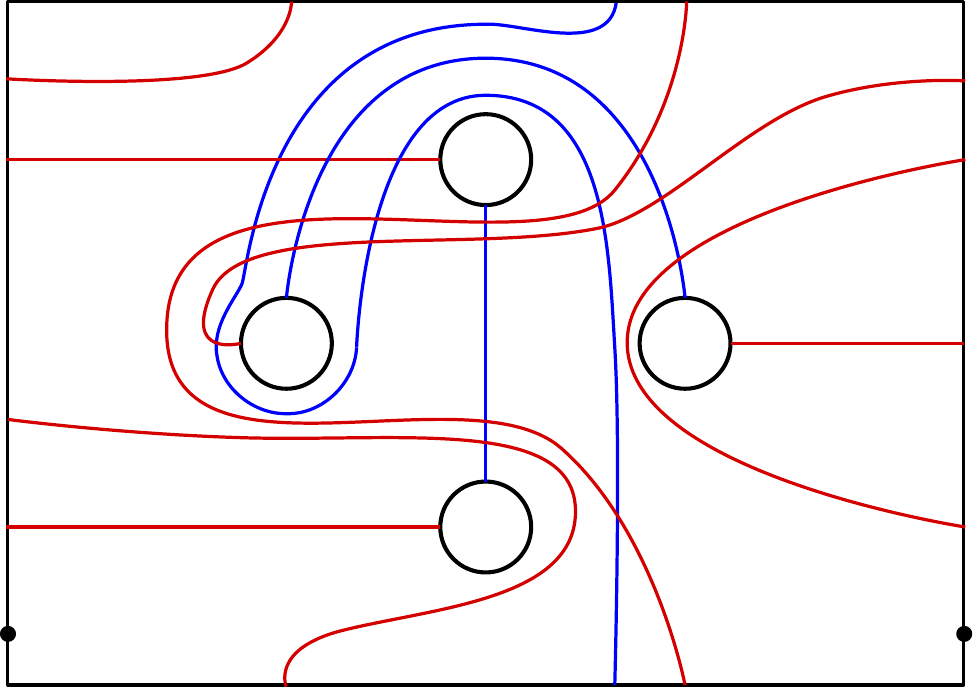}
    \caption{The Heegaard diagram $\bY$.}
    \label{fig:y_arc}
\end{figure}

By construction, $\bX = \bY \cup (T^\infty, \nu)$. We would like associate to $\bY$ a type $DA$ bimodule which will pair with $\CFAnuhalf$ to recover $\CFAh_\cR(\bX).$ Naively, one might define a bimodule $\CFDAh(\bY)$ just as in \cite{LOT_bimodules}, as the type DA bimodule generated by tuples of intersection points with type DA structure maps 
\[
\delta^1_{n+1}(x, a(\sigma_1), \hdots, a(\sigma_n)) 
:=\sum_{y \in \mathfrak{G}(\bY)} \sum_{\substack{B \in \pi_2(x, y)\\
                              \ind(B;\Vec{\rho};\sigma_1,\hdots \sigma_n) = 1}}
          \# (\cM^B(x, y; \Vec{\rho}; \sigma_1, \hdots, \sigma_n))a(-\rho_1), \hdots, a(-\rho_m)\otimes  y,
\]
given by counting rigid holomorphic curves with prescribed asymptotics. However, it is not immediate that these maps give rise to well-defined bimodules, since the diagram $\bY$ is not an arced Heegaard diagram; in fact, there is no path in $\bY$ connecting the basepoints on the two boundaries so we cannot appeal to \cite{LOT_bimodules} to verify that the maps $\delta_n$ satisfy the compatibility conditions for a type DA bimodule. 

Rather than carefully study the moduli spaces appearing in the maps $\delta_n^1$, we employ the following devious trick. The curves defining the structure maps avoid the basepoints on the two boundaries. So, we may consider the diagram $\bY^+$ obtained by attaching a 2-dimensional 1-handle to $\bY$ with feet attached near the two basepoints, in the regions containing $\rho_{0}$ and $\sigma_0$. $\bY^+$ can be promoted to an arced diagram by connecting the two basepoints by an arc that traverses the annulus. In this case, there is a well-defined bimodule $\CFDAh(\bY^+)$ with type DA-module structure maps 
\[
    (\delta^1_{n+1})^+: \CFDAh(\bY^+) \otimes \cA^{\otimes n} \ra \cA \otimes \CFDAh(\bY^+)
\]
given by \cite{LOT_bimodules}. The curves appearing in the structure maps for $\CFDAh(\bY^+)$ avoid the basepoint arc, and are therefore supported in the region of the diagram which is identical to $\bY.$ Therefore, the moduli spaces defining the maps $\delta_n^1$ are identical to those defining the maps $(\delta^1_{n+1})^+$; it therefore follows that the bimodule $\CFDAh(\bY)$ is well defined. Equivalently, we can simply \emph{define} $\CFDAh(\bY)$ to be bimodule with structure maps $(\delta^1_{n+1})^+$. It follows that this bimodule can be paired with other bordered modules to compute the Floer homology after gluing. In summary, we have:

\begin{lem}
    The bimodule $(\CFDAh(\bY), \delta_n^1)$ is well-defined. Moreover, the various pairing theorems for bimodules also hold for $\CFDAh(\bY)$. 
\end{lem}
In particular, it follows that
    \[
        \CFAh_{\F[U, V]/V}(T_\infty, \nu) \boxtimes \CFDAh(\bY) \simeq \CFAh_{\F[U, V]/V}(\bX),
    \]
since the $w$ basepoint is contained in the region containing the chords $\rho_2$ and $\rho_3$. By moving the basepoint from the region intersecting the Reeb chord $\sigma_0$ to the region intersecting $\sigma_{3}$ and working over the algebra $\cA'(T^2) := \cAt/\langle\sigma_{3}\rangle$, we have that   
    \[
        \CFAh_{\F[U, V]/U}(T_\infty, \nu) \boxtimes \CFDAh(\bY) \simeq \CFAh_{\F[U, V]/U}(\bX),
    \]
since the $z$ basepoints contained in the region containing the chords $\rho_0$ and $\rho_1$. By the same argument as above, attaching a tube with one foot connecting the two basepoints allows us to define type DA structure maps $(\delta^1_n)'$ which produce a type DA bimodule over $(\cA'(T^2), \cA(T^2)).$ By combining the contributions of these two structure maps, we may define a \emph{half-extended} bimodule associated to $\bY$, i.e. a $(\cA(T^2)_{0,3},\widetilde{\cA}(T^2))$-bimodule. Working over this quotient allows us to encode information about holomorphic curves which cover either the $\sigma_0$ region or the $\sigma_3$ region and not both; hence, when paired with $\CFAnuhalf$, we can record those curves which cover at most one of the two basepoints.

\begin{rem}
    We note that it is unclear how to define a bimodule over the fully extended torus algebra. For this reason, we content ourselves with working over the smaller algebra $\cA(T^2)_{0,3}$. As noted above, this is enough to recover $\CFK_\cR(S^3, K)$.
\end{rem}

\begin{lem}\label{lem:half-extended pairing}
    The diagram $\bY$ gives rise to a well-defined type-DA $(\cA(T^2),\widetilde{\cA}(T^2)_{0,3})$-bimodule, which we denote $\Yd$. In particular, by the pairing theorem 
    \[
    \CFAnuhalf \boxtimes \CFDAyh \simeq \CFAh_\cR(\bX).
    \]
    where $\CFAnuhalf$ is the half-extended module from Definition \ref{def:half extended torus module}.
\end{lem}

The proof of \Cref{lem:half-extended pairing} is the same as the ordinary pairing formula for a type D module and a type DA bimodule \cite[Theorem 11]{LOT_bimodules}, and thus is omitted.

\begin{rem}
    Since it will be necessary to work with this bimodule over the usual torus algebra and also over the half-extended torus algebra we emphasize our notation: $\CFDAy$ will be used to denote the $(\cA, \cA)$-bimodule and $\CFDAyh$ to denote the half-extended $(\widetilde{\cA}_{0,3},\cA)$-bimodule. Since $\CFDAyh$ is constructed to capture information regarding curves which may cover either of the two basepoints in $(T_\infty, \nu)$, we think of $\CFDAyh$ as having two basepoints, or a knot, in its left hand boundary. 
\end{rem}

The following sections will necessitate a bit of a juggling act with the various extensions of the usual bordered algebras and modules we have defined in this section. Aiming to help improve the reader's notation-eye-coordination, the various algebraic objects are collected in Table \ref{tab:modules tables}. 

\begin{table}[h]
    \centering
    \begin{tabular}{|l|l|l|} \hline 
        (Bi)Module & Ground Ring & Basepoints \\ \hline & &\\[-1em]
         $\CFAh(T_\infty, \nu)$ & $\cA(T^2)$ & Two boundary basepoints  \\ [0.5ex]
         $\CFAt(T_\infty, \nu)$ & $\widetilde{\cA}(T^2)\otimes \cR$ & Two boundary basepoints  \\ [0.5ex]
         $\CFAh_\cR(\bX)$ & $\cA(T^2)\otimes \cR$ & One boundary basepoint, two internal  \\ [0.5ex]
         $\CFAnuhalf$ & $\widetilde{\cA}(T^2)_{0,3}\otimes \cR$ & Two boundary basepoints  \\ [0.5ex]
         $\CFDAy$ & $(\cA(T^2),\cA(T^2))$ & One basepoint on each boundary. \\ [0.5ex]
         $\CFDAyh$ & $(\widetilde{\cA}(T^2)_{0,3},\cA(T^2))$ & Two basepoints on type A boundary, one on the type D boundary \\ [0.5ex] \hline
    \end{tabular}
    \caption{A summary of some of bordered modules and bimodules that will be used in this paper.}
    \label{tab:modules tables}
\end{table}

\subsection{Splittings}\label{subsection:splittings}

Before proceeding, we make a few remarks about how the \LOT correspondence behaves with respects to splittings.

\begin{defn}\label{def:extendable summand}
    Let $M$ be an extendable type D structure. We say that a summand $N$ of $M$ is an \emph{extendable direct summand} if $N$ is itself extendable, and there is another extendable type D structure $N'$ such that  $M\simeq N\oplus N^\prime$.
\end{defn}

Let us introduce some notation. If $C$ is a summand of $\CFK_\cR(K,p)$, we will write $\cM_C$ for the corresponding summand of $\CFDh(\KC,p)$ obtained by applying the \LOT construction. Conversely, if $M$ is an extendable summand of $\CFDh(\KC,p)$ we will write $\cC_M$ for the summand of $\CFK_\cR(K,p)$ given by choosing an extension $M_{0,3}$ of $M$ over $\cA_{0,3}$ and tensoring with $\CFAnuhalf.$

In comparing $\CFK_\cR$ and $\CFDt$, we will make repeated use of the classification of these two classes of objects due to \cite{hanselman2023bordered,popovic2023link}. The following is a direct consequence of their work, though we collect the relevant details for completeness. In particular, this correspondence is carefully discussed in \cite[Section 6]{popovic2023link}.

\begin{lem}[\cite{hanselman2023bordered,popovic2023link}] \label{lem:type D vs R cx summands}
    Under the \LOT correspondence, indecomposable summands of $\CFDt(\KC)$ correspond precisely to indecomposable summands of $\CFK_\cR(S^3, K)$.
\end{lem}
\begin{proof}
    Indecomposable summands of $\cR$-complexes are either local systems or snake complexes. Applying the \LOT correspondence such a summand $C$ yields a type D structure $\cM_C$. According to \cite[Section 3]{hanselman2023bordered}, $\cM_C$ can be represented by immersed train track. If $C$ is a snake complex, the corresponding immersed train track is already an immersed curve with a single connected component, and therefore specifies an indecomposable, extendable summand $\widetilde{\cM_C}$. See \Cref{fig:popovic_correspondence2}. 
    
    If $C$ is a local system, it is specified up to isomorphism by a triple $(s, w, [A])$, where $s$ specifies the \emph{shape} of of $C$, $w$ is a natural number, and $[A]$ is a conjugacy class of $A \in \mathrm{GL}_w(\F)$. $C$ determines an immersed curve with local system as follow: the shape of $C$ specifies an immersed curve in the solid torus and $[A]$ determines the local system. In \cite{hanselman2023bordered}, they note that local systems can be interpreted as ``crossover arrows'', between parallel strands. Hence, the associated type $D$ structure is extendable and consists of a single component so is indecomposable. See \Cref{fig:popovic_correspondence} for an example.

    The reverse direction is even simpler. Given an extendable indecomposable summand $M$, we can represented it by an immersed multi-curve with local system. Replacing the puncture with a pair of basepoints and counting bigons weighted by powers of $U$ and $V$ recovers either a snake complex or a local system. The resulting complex is exactly $\cC_M.$ 
\end{proof}
\begin{rem} \label{rem: LOT are extendable}
    In particular, any summand $M$ of $\CFDh(\KC)$ which came from a summand of $\CFK_\cR(S^3,K)$ is extendable. This can be generalized even further: it follows from \cite[Theorem 4.2]{popovic2023link} that $\mathcal{M}_C$ is extendable for any finitely generated free chain complex $C$ over $\cR$. Furthermore, the non-equivariant version of \Cref{thm: induced splittings} is clear.
\end{rem}

\begin{figure}
\def\svgwidth{\linewidth}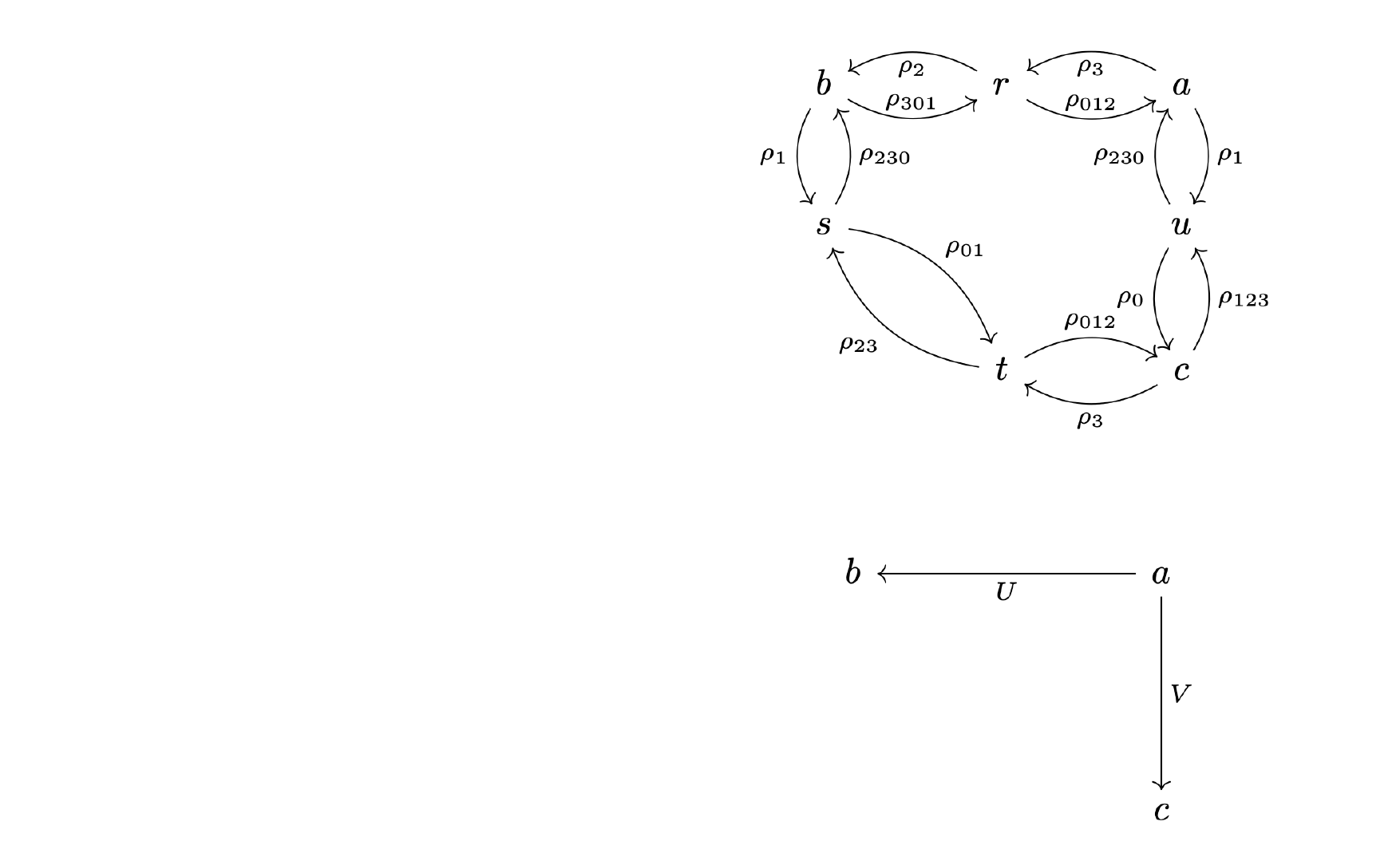
\caption{An irreducible snake complex, represented as a decorated immersed curve, an extended type D structure, and an $\cR$-complex.}\label{fig:popovic_correspondence2}
\end{figure}

\begin{figure}
\def\svgwidth{\linewidth}
\begingroup%
  \makeatletter%
  \providecommand\color[2][]{%
    \errmessage{(Inkscape) Color is used for the text in Inkscape, but the package 'color.sty' is not loaded}%
    \renewcommand\color[2][]{}%
  }%
  \providecommand\transparent[1]{%
    \errmessage{(Inkscape) Transparency is used (non-zero) for the text in Inkscape, but the package 'transparent.sty' is not loaded}%
    \renewcommand\transparent[1]{}%
  }%
  \providecommand\rotatebox[2]{#2}%
  \newcommand*\fsize{\dimexpr\f@size pt\relax}%
  \newcommand*\lineheight[1]{\fontsize{\fsize}{#1\fsize}\selectfont}%
  \ifx\svgwidth\undefined%
    \setlength{\unitlength}{841.88976378bp}%
    \ifx\svgscale\undefined%
      \relax%
    \else%
      \setlength{\unitlength}{\unitlength * \real{\svgscale}}%
    \fi%
  \else%
    \setlength{\unitlength}{\svgwidth}%
  \fi%
  \global\let\svgwidth\undefined%
  \global\let\svgscale\undefined%
  \makeatother%
  \begin{picture}(1,0.70707071)%
    \lineheight{1}%
    \setlength\tabcolsep{0pt}%
    \put(0,0){\includegraphics[width=\unitlength,page=1]{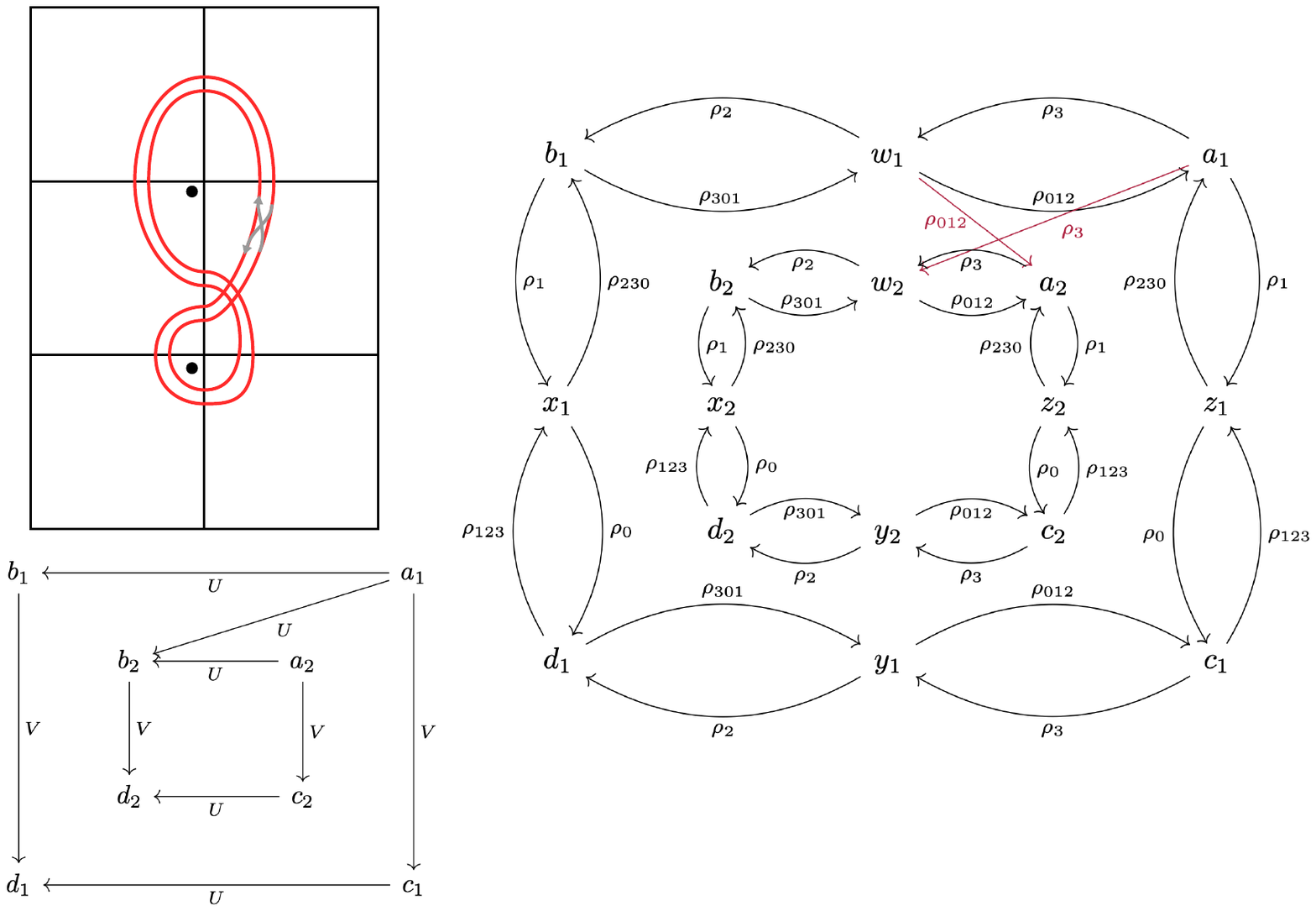}}%
    \put(0.18140654,0.39028697){\color[rgb]{0,0,0}\makebox(0,0)[lt]{\smash{\begin{tabular}[t]{l}{\small$c$}\end{tabular}}}}%
    \put(0.18111765,0.64641778){\color[rgb]{0,0,0}\makebox(0,0)[lt]{\smash{\begin{tabular}[t]{l}{\small$b$}\end{tabular}}}}%
    \put(0.14500408,0.47576799){\color[rgb]{0,0,0}\makebox(0,0)[lt]{\smash{\begin{tabular}[t]{l}{\small$a$}\end{tabular}}}}%
    \put(0.17966966,0.50627555){\color[rgb]{0,0,0}\makebox(0,0)[lt]{\smash{\begin{tabular}[t]{l}{\small$d$}\end{tabular}}}}%
    \put(0.12195333,0.44599218){\color[rgb]{0,0,0}\makebox(0,0)[lt]{\smash{\begin{tabular}[t]{l}{\small$y$}\end{tabular}}}}%
    \put(0.21656591,0.44162622){\color[rgb]{0,0,0}\makebox(0,0)[lt]{\smash{\begin{tabular}[t]{l}{\small$z$}\end{tabular}}}}%
    \put(0.23115604,0.56892683){\color[rgb]{0,0,0}\makebox(0,0)[lt]{\smash{\begin{tabular}[t]{l}{\small$w$}\end{tabular}}}}%
    \put(0.110416,0.56922743){\color[rgb]{0,0,0}\makebox(0,0)[lt]{\smash{\begin{tabular}[t]{l}{\small$x$}\end{tabular}}}}%
  \end{picture}%
\endgroup%

\caption{An irreducible local system, represented as a decorated immersed curve, an extended type D structure, and an $\cR$-complex.}\label{fig:popovic_correspondence}
\end{figure}

The free basepoint plays no essential role. We can certainly apply the algebraic construction of \cite[Theorem 11.26]{LOT_bordered_HF} to $\CFK_\cR(K, p)$ to produce a corresponding type $D$ structure. Moreover, the proof \cite[Theorem 11.26]{LOT_bordered_HF} still holds to imply that the resulting type D structure is a model for $\CFDh(\KC, p)$. Though, for simplicity, we deduce the desired result algebraically as follows.

\begin{lem}\label{thm: LOT}
    Let $\widetilde{M}$ be a summand of the extended type $D$ structure $\CFDt(\KC, p)$ for a knot complement with a free interior basepoint, $p$. Denote by $\widetilde{M}_{0,3}$ truncated module over the algebra $\widetilde{\cA}_{0,3}$ and by $M$ the module over the algebra over $\cA$ coefficients. Let 
    \[
    \cC_M := CFA_\cR(T_\infty,\nu)\boxtimes \widetilde{M}_{0,3}.
    \]
    Then $M$ is homotopy equivalent to the type D structure given by the \LOT correspondence.
\end{lem}
\begin{proof}
    Given a $\cR$-complex $C$, denote by $\cM_C$ the type D structure associated to $C$ under the \LOT correspondence. In this notation, the claim is that $M \simeq \cM_{C_M}$, i.e. that the \LOT construction is a left inverse to the assignment 
    \[
    M \mapsto \CFA_\cR(T_\infty, \nu)\boxtimes \widetilde{M}_{0,3}.
    \]
    
    By assumption, the type D structure $M$ is an extendable direct summand of $\widehat{CFD}(S^3 \smallsetminus K,p)$. By the work of Zemke \cite{zemke_linkcob}, adding an interior basepoint amounts to performing a free stabilization. It follows that 
    \[
    \widehat{CFD}(S^3 \smallsetminus K,p)\simeq \widehat{CFD}(S^3 \smallsetminus K)^{\oplus 2}.
    \]
    Furthermore, we may assume that 
    \[
    M \simeq \mathcal{M}_{C_1 \oplus C_2} \simeq \mathcal{M}_{C_1}\oplus \mathcal{M}_{C_2}
    \]
    for some (possibly trivial) direct summands $C_1$ and $C_2$ of $CFK_{\mathcal{R}}(S^3,K)$. This follows from, say, considering the immersed curve formulation of bordered Floer homology \cite{hanselman2023bordered}: direct sums of extendable summands correspond to unions of connected components in the immersed curve formulation of extendable type D structures \cite[Theorem 1.5]{hanselman2023bordered} the immersed curve representative of $\widehat{CFA}(S^3 \smallsetminus K)^{\oplus 2}$ is simply given by two copies of the immersed curve representative of $\widehat{CFA}(S^3 \smallsetminus K)$.
    
    According to \Cref{lem:type D vs R cx summands} (and \Cref{rem: LOT are extendable}), for any direct summand $C$ of a knot Floer chain complex $CFK_\cR(S^3,K)$ of any knot $K$, $\cM_C$ admits an extension $\widetilde{\cM}_C$, whose truncation $\widetilde{\cM}^{0,3}_C$ satisfies
    \[
    CFA_\cR(T_\infty,\nu)\boxtimes \widetilde{\cM}^{0,3}_C \simeq C.
    \]
    Then it follows from the unique extension property \cite[Proposition 4.16]{hanselman2023bordered} of extendable type D structures that
    \[
    \widetilde{M}\simeq \widetilde{\mathcal{M}}_{C_1}\oplus \widetilde{\mathcal{M}}_{C_2},
    \]
    so tensoring with $CFA_\mathcal{R}(T_\infty,\nu)$ yields 
    \[
    CFA_\mathcal{R}(T_\infty,\nu)\boxtimes \widetilde{M}_{0,3}\simeq C_1 \oplus C_2,
    \]
    implying that $\cC_M \simeq C_1\oplus C_2$. Therefore,
    \[
    M\simeq \mathcal{M}_{C_1 \oplus C_2}\simeq \mathcal{M}_{\cC_M},
    \]
    as desired.
\end{proof}

In summary, splittings of $\CFK_\cR(S^3,K)$ induce splittings of $\CFDh(\KC)$ and splittings of $\CFDt(\KC)$ induce splittings of $\CFK_\cR(S^3,K).$ However, it is not clear a priori that a splitting of $\CFDh(\KC)$ induces a splitting of $\CFK_\cR(S^3,K)$. We will address this issue in \Cref{lem:summand-extendable} by showing that splittings of $\CFDh(\KC)$ induce splittings of $\CFDt(\KC)$.

\subsection{Algebraic Lemmas}\label{subsection: algebraic lemmas}

Here, we include some basic algebraic lemmas which will be useful to us later.

\begin{defn}
    We say that a morphism $p$, between objects of the same type, i.e. chain complexes, type D or type A structures, etc. is a \emph{projection} if $p^2 = p$ and a \emph{homotopy projection} if $p^2$ is homotopic to $p$.
\end{defn}

The homotopy type of a projection determines the homotopy equivalence type of the corresponding summand. 

\begin{lem}\label{lem: homotopic projections have equivalent images}
    Let $C$ be a finitely generated chain complex over $\mathcal{R}$ and let $p, q \in \End_\cR(C)$ be projection maps. If $p$ and $q$ are homotopic, then their images $p(C)$ and $q(C)$ are homotopy equivalent. The same holds for finitely generated reduced type D (or type A) structures.
\end{lem}
\begin{proof}
    If $p$ and $q$ are homotopic, then it is well known that $\Cone(p)\simeq \Cone(q)$. But, as $p$ and $q$ are projections, $\Cone(p) \simeq \ker(p)$ and $\Cone(q) \simeq \ker(q)$. Hence, the kernels of homotopic projections are homotopy equivalent. But, if $p$ is a projection, so is $1 + p$ (similarly for $q$). But, $\ker(1 + p) \simeq p(C)$ and $\ker(1 + q) \simeq q(C)$. Hence, $p(C) \simeq q(C)$.
\end{proof}

We will make repeated use of the fact that homotopy projections are homotopic to honest projections. 

\begin{lem} \label{lem: homotopy projection to projection}
    Given a finitely generated chain complex $C$ over $\mathcal{R}$, any homotopy projection of $C$ is homotopic to a projection. This statement also holds for finitely generated reduced type D (or type A) structures over the torus algebra $\mathcal{A}$.
\end{lem}
\begin{proof}
    We will only prove the lemma for the case of chain complexes since the proof for type D (or type A) structures case is nearly identical. Let $p:C\rightarrow C$ be a homotopy projection. Since $C$ is finitely generated over $\mathcal{R}$, we see that it has only finitely many elements in each bidegree. Thus, by invoking the finite generation again, we see that there are only finitely many bidegree-preserving chain endomorphisms of $C$. Hence the maps $p,p^2,p^3,\cdots$ cannot be pairwise distinct, i.e. there exists some integers $n,k>0$ such that $p^n = p^{n+k}$. Choose any integer $N$, divisible by $k$, such that $N>n$. Then $p^N = p^{2N} = (p^N)^2$ and thus $p^N$ is a projection. Also, since $p\sim p^2$, any positive power of $p$ is homotopic to $p$, and thus $p\sim p^N$. Therefore $p^N$ is a projection of $C$ that is homotopic to $p$, as desired.
\end{proof}

Furthermore, this process can be done for whole collections.

\begin{lem} \label{lem: commuting projections}
    Given a finitely generated chain complex $C$ over $\mathcal{R}$, let $p_i:C\rightarrow C$, $i=1,\cdots,n$ be chain endomorphisms such that the following conditions are satisfied:
    \begin{itemize}
        \item For each $i$, $p_i^2 \sim p_i$.
        \item For each $i,j$, $p_i p_j \sim p_j p_i$.
    \end{itemize}
    Then, each $p_i$ can be homotoped so that they form a family of pairwise commuting projections. The same statement holds for finitely generated type D (or type A) modules if $p_i$ are replaced by type D (or type A) endomorphisms.
\end{lem}
\begin{proof}
    We only prove the lemma in the chain complex setting; the type D (or type A) setting can be proved in the same way. First, homotope $p_i$ so that they form a pairwise commuting family of endomorphisms. To do this, inductively define chain endomorphisms $p^\prime_i$ as follows.
    \begin{itemize}
        \item Invoke \Cref{lem: homotopy projection to projection} to homotope $p_1$ to an honest projection $p^\prime_1$;
        \item Define $p^\ast_{i+1} = \displaystyle\sum_{\lambda:\{1,\cdots,i\}\rightarrow \{0,1\}} p_{\lambda,i} p_{i+1} p_{\lambda,i}$, where  $p_{\lambda,i} = \prod_{j=1}^i (\lambda(j) + p_j)$.
        \item Again, invoke \Cref{lem: homotopy projection to projection} to homotope $p^\ast_{i+1}$ to an honest projection $p^\prime_{i+1}$.
    \end{itemize}
    Note that $p_{\lambda,i}p_{\lambda^\prime,i}\sim 0$ whenever $\lambda\neq \lambda^\prime$.
    \begin{align*}
        p_{\lambda,i}p_{\lambda^\prime,i} & = \prod_{j=1}^i (\lambda(j) + p_j)(\lambda'(j) + p_j)\\
        & = \prod_{j=1}^i(\lambda(j)\lambda'(j) + p_j\lambda'(j) +\lambda(j)p_j + p_jp_j) \\
        & \sim \prod_{j=1}^i (\lambda(j)\lambda'(j) + p_j\lambda'(j) +\lambda(j)p_j + p_j.)
    \end{align*}
    Since $\lambda\neq \lambda^\prime$, there is some $j$ so $\lambda(j)\lambda'(j) = 0$ and $\lambda(j)\neq \lambda'(j)$. Hence, 
    \begin{align*}
        \lambda(j)\lambda'(j) + p_j\lambda'(j) +\lambda(j)p_j + p_j & = 0 + p_j + p_j\\
        & = 0
    \end{align*}
    as claimed. Therefore, we have that
    \[
    \left( p^\ast_{i+1} \right)^2 = \sum_\lambda (p_{\lambda,i}p_{i+1}p_{\lambda,i})^2 \sim  \sum_\lambda p^2_{\lambda,i}p^2_{i+1}p^2_{\lambda,i} = \sum_\lambda p_{\lambda,i}p_{i+1}p_{\lambda,i} = p^\ast_{i+1},
    \]
    so the map $p^\ast_{i+1}$ is a homotopy projection. Since the maps $p_1,\cdots,p_n$ homotopy-commute, it is straightforward to see that each $p^\ast_{i+1}$ is homotopic to $p_{i+1}$. Furthermore, since $p_1,\cdots,p_i$ are inductively assumed to be pairwise commuting projections, we can also see that $p^\ast_{i+1}$ commutes with $p_1,\cdots,p_i$. This ensures that the map $p^\ast_{i+1}$ splits along the splitting of $C$, induced by $p_1,\cdots,p_i$, into $2^i$ direct summands. 
    
    Now, since the proof of \Cref{lem: homotopy projection to projection} modifies a homotopy projection only on the image of its sufficiently big power by making it act by identity on it, we see that the projection $p^\prime_{i+1}$ should also split along the same splitting of $C$. Therefore $p^\prime_{i+1}$ is a projection which commutes with $p_1,\cdots,p_i$; the lemma follows.
\end{proof}

\section{Induced Splittings from Cobordism Maps}\label{sec: cobs}

As outlined in the introduction, our first step in establishing the invariant splitting principles is to construct a map
\(
\Lambda: \End(\CFDh(\KC)) \ra \End(\CFK_\cR(S^3,K)),
\)
prove that it descends to a ring map on homology, and show that it sets up a correspondence between splittings of $\CFDh(\KC)$ and $\CFK_\cR(S^3, K)$, without needing to pass through $\CFDt(\KC)$. Essentially by construction, this correspondence will respect the actions of $\iota_K$ and $\iota_{\KC}$. The following is the first key result of this section. 

\begin{prop} \label{prop:doubling-CFA-to-CFK}
    Let summand $M$ be a summand of $\CFDh(\KC)$. Then, there is a corresponding summand $\Lambda(C)$ of $\CFK_\cR(S^3,K)$ determined up to homotopy equivalence by the property that  
    \[
        \Lambda(C) \oplus \Lambda(C) \simeq \CFAx\boxtimes M.
    \]
    Moreover, if $\CFDh(\KC) \simeq M_1 \oplus \hdots \oplus M_n$ is an $\iota_{\KC}$-equivariant splitting, then $\CFK_\cR(S^3,K) \simeq \Lambda(M_1) \oplus \hdots \oplus\Lambda(M_n)$ is an $\iota_K$-equivariant splitting.
\end{prop}

We emphasize that this splitting of $\CFK_\cR(S^3,K)$ does not depend on a choice of splitting of $\CFDt(\KC)$ (and therefore bypasses the issue of extending splittings of $\CFDh(\KC)$). If $\CFDh(\KC) \cong M_1 \oplus \hdots \oplus M_n$ was obtained from a splitting $\CFDt(\KC) \cong \widetilde{M}_1 \oplus \hdots \oplus \widetilde{M}_n$, we then have two (a priori different) splittings of $\CFK_\cR(S^3,K)$ into summands of the form $\Lambda(M_i)$ and $\CFAt(T^\infty, \nu) \boxtimes \widetilde{M_i}$. Given the relation $\Lambda(C_i) \oplus \Lambda(C_i) \simeq \CFAx\boxtimes M_i$, it seems reasonable to hope that $\Lambda(C_i) \simeq \CFAt(T^\infty, \nu)\boxtimes M_i$. We return to this in \Cref{sec: doubling}. 

The second key result of this section is to show that the association $\pi \mapsto \Lambda(\pi)$ actually establishes a correspondence between projection maps of $\CFDh(\KC)$ and $\CFK_\cR(S^3,K)$.

\begin{prop}\label{prop: bijective on projection}
    The map $\Lambda$ induces a bijection between homotopy classes of projections on $\CFDh(\KC)$ and $\CFK_\cR(S^3,K)$.
\end{prop}

\begin{rem}
    We note that it may initially appear odd that the function $\Lambda$ takes ``hat'' endomorphisms to $\cR$-complex endomorphisms. Though, this is not too unusual in Heegaard Floer theory. In some sense, this is because the function $\Lambda$ is closely related to the large surgery isomorphism, which we recall (for large $N$) computes $\CFh(S^3_N(K))$ in terms of $\CFK_\cR(S^3, K)$ rather than $\CFKh(S^3, K)$. 
    
    More precisely, by the morphism pairing theorem, we may identify
    \begin{align*}
        \End^\cA(\CFDh(\KC)) \simeq \CFh(S^3_0(K\#\overline{K})).
    \end{align*}
    For our purposes, it will be more convenient to replace the basepoint of $S^3_0(K\#\overline{K})$ with a local unknot $\lambda$. The map $\Lambda$ will be defined in terms of a certain three-ended cobordism, which gives rise to a map:
    \begin{align*}
        \CFK_\cR(S^3_0(K\#\overline{K}),\lambda) \otimes_\cR \CFK_\cR(S^3, K) \ra \CFK_\cR(S^3, K).
    \end{align*}
    Since there is a clear identification $\CFh(S^3_0(K\#\overline{K}))\otimes_\F \cR \simeq \CFK_\cR(S^3_0(K\#\overline{K}),\lambda)$, we think of this as a map 
    \begin{align*}
        \CFKh(S^3_0(K\#\overline{K})) \otimes_\F\CFK_\cR(S^3, K) \ra \CFK_\cR(S^3, K),
    \end{align*}
    which explains the mixing of the $\CFh$ and $\CFK_\cR$.
\end{rem}

\subsection{Pair of pants cobordisms}

We recall the ``pair of pants'' construction from \cite{LOT_HF_as_morphism,cohen2023composition}. Given a triple of bordered 3-manifolds, $Y_1$, $Y_2$, and $Y_3$ with common boundary $F$, the associated pair of pants cobordism 
\[
P: (-Y_1 \cup Y_2) \amalg (-Y_2 \cup Y_3) \ra (-Y_1 \cup Y_3)
\]
is defined to be
\[
P = (\Delta\times F)\cup_{e_1 \times F} (e_1 \times Y_1)\cup_{e_2 \times F} (e_2 \times Y_2)\cup_{e_3 \times F} (e_3 \times Y_3),
\]
where $\Delta$ is a triangle with edges $e_1$, $e_2$, and $e_3$. We will also consider the case where $Y_1$ contains a knot $K$ in its interior. The triple $((Y_1, K)$, $Y_2$, $Y_3)$ then determines a link cobordism 
\[
(P, C): (-Y_1 \cup_F Y_2, K) \amalg (-Y_2 \cup Y_3, \emptyset) \ra (-Y_1 \cup Y_3, K)
\]
by the same construction. 

\begin{figure}
\def\svgwidth{.5\linewidth}
\begingroup%
  \makeatletter%
  \providecommand\color[2][]{%
    \errmessage{(Inkscape) Color is used for the text in Inkscape, but the package 'color.sty' is not loaded}%
    \renewcommand\color[2][]{}%
  }%
  \providecommand\transparent[1]{%
    \errmessage{(Inkscape) Transparency is used (non-zero) for the text in Inkscape, but the package 'transparent.sty' is not loaded}%
    \renewcommand\transparent[1]{}%
  }%
  \providecommand\rotatebox[2]{#2}%
  \newcommand*\fsize{\dimexpr\f@size pt\relax}%
  \newcommand*\lineheight[1]{\fontsize{\fsize}{#1\fsize}\selectfont}%
  \ifx\svgwidth\undefined%
    \setlength{\unitlength}{248.65168294bp}%
    \ifx\svgscale\undefined%
      \relax%
    \else%
      \setlength{\unitlength}{\unitlength * \real{\svgscale}}%
    \fi%
  \else%
    \setlength{\unitlength}{\svgwidth}%
  \fi%
  \global\let\svgwidth\undefined%
  \global\let\svgscale\undefined%
  \makeatother%
  \begin{picture}(1,0.44909848)%
    \lineheight{1}%
    \setlength\tabcolsep{0pt}%
    \put(0,0){\includegraphics[width=\unitlength,page=1]{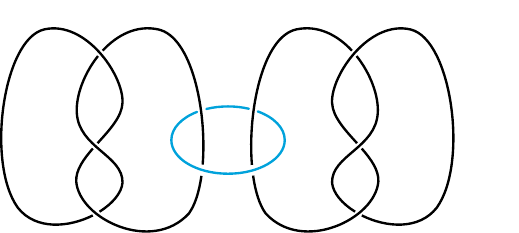}}%
    \put(0.41008329,0.27170317){\color[rgb]{0,0.63921569,0.85490196}\makebox(0,0)[lt]{\smash{\begin{tabular}[t]{l}{\small$0$}\end{tabular}}}}%
    \put(0.03610767,0.41248049){\color[rgb]{0,0,0}\makebox(0,0)[lt]{\smash{\begin{tabular}[t]{l}{\small$0$}\end{tabular}}}}%
    \put(0.7559783,0.41248049){\color[rgb]{0,0,0}\makebox(0,0)[lt]{\smash{\begin{tabular}[t]{l}{\small$0$}\end{tabular}}}}%
  \end{picture}%
\endgroup%

\caption{The splice of the 0-framed trefoil and its mirror}\label{fig:splice}
\end{figure}

We will be interested in the pair of pants cobordism $(P,C)$ formed from the triple $(T_\infty, \nu), (\KC)_0, (\KC)_n$. It will be helpful to have an explicit description of manifolds of this kind. First, we note that in this situation, all three 3-manifolds are knot complements, hence the glued up 3-manifolds are \emph{splices}. If $(\KC)_n$ and $(S^3 \smallsetminus J)_m$ are framed knot complements, we will write $\mathcal{S}((K,m), (J,n))$ for their splice. We will make use of the fact that the splice of two knot complements is diffeomorphic to Dehn surgery on their connected sum, i.e. 
\[
(\KC)_n \cup (S^3 \smallsetminus J)_m \cong S^3_{n+m}(K\#J).
\]
See, for example, \cite{Gordon_dehn_surgery_satellites}. $(P, C)$ can be given a handle decomposition as follows. The process is carried out in Figure \ref{fig:lambda_cob1}.

\begin{figure}[h]
\def\svgwidth{\linewidth}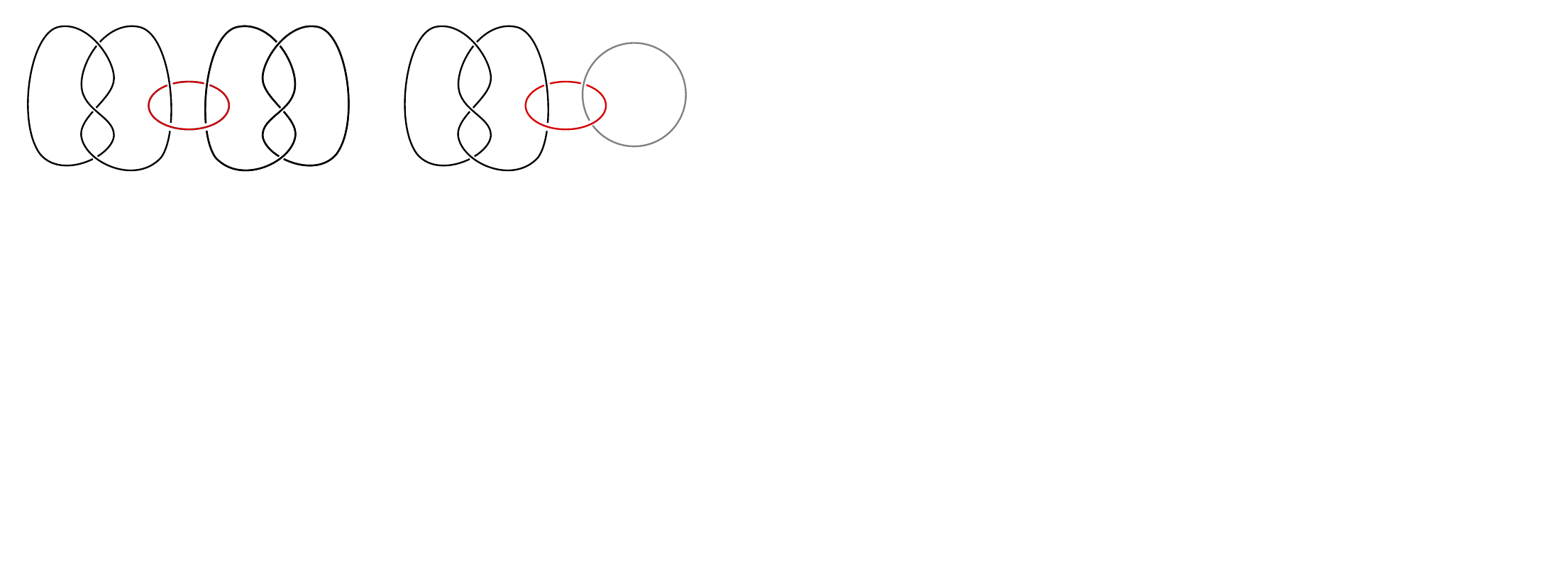
\caption{Building the cobordism $(\Lambda,C)$}\label{fig:lambda_cob1}
\end{figure}

\begin{enumerate}
    \item The base of the cobordism has two components given by thickening up $\mathcal{S}((K,n), (\overline{K},0))$ and $\mathcal{S}((K,0), (U,\infty))$. At this step, the surface $C$ appears as (the product of) a meridian of the unknotted component of the surgery presentation of $\mathcal{S}((K,0), (U,\infty))$. See Figure \ref{fig:lambda_cob1} (a).
    
    \item  Next, attach $(S^3\smallsetminus K) \times I$ to $-(S^3\smallsetminus K)_0 \amalg (S^3\smallsetminus K)_0$ along $(S^3\smallsetminus K) \times \partial I$. This can done in two steps. Let $B$ be a small ball intersecting $K$ in a trivial arc. Then, the complement of $B \times I$ in $(S^3\smallsetminus K) \times I$ is diffeomorphic to the complement of the canonical ribbon disk for $\overline{K}\# K$, which we denote $\Delta_{\overline{K}\#K}$. 
    
    Topologically, $I \times (B\smallsetminus K)$ is just an interval cross a solid torus. If we build $B\smallsetminus K$ with a single 3-dimensional 0-handle and 1-handle, the product $I \times (B\smallsetminus K)$ decomposes as a 4-dimensional 1-handle and 2-handle. This 4-dimensional 1-handle is attached to $\partial_- P$, giving a cobordism 
    \begin{align*}
        -(S^3\smallsetminus K)_0 \amalg (S^3\smallsetminus K)_0 \ra -(S^3\smallsetminus K)_0 \# (S^3\smallsetminus K)_0.
    \end{align*}
    Next, the 2-handle is attached (with framing zero) along the connected sum of the meridian of $S^3 \smallsetminus \overline{K}$ and $S^3 \smallsetminus K$. See Figure \ref{fig:lambda_cob1} (b). The 1-handle connecting the two components of $\partial_-P$ is not drawn. 
    
    \item We now just need to attach $B^4 \smallsetminus \Delta_{\overline{K}\#K}$. The boundary of $B^4 \smallsetminus \Delta_{\overline{K}\#K}$ is $S^3_0(\overline{K} \# K)$, or equivalently, the splice of $-(\KC)_0$ and $(\KC)_0$. This 3-manifold is visible in \ref{fig:lambda_cob1} (b), as surgery on the 3-component link consisting of $\overline{K}$, $K$, and an unknotted component clasping them. We glue $B^4 \smallsetminus \Delta_{K\#\overline{K}}$ along this submanifold. 
\end{enumerate}

\begin{figure}
\def\svgwidth{\linewidth}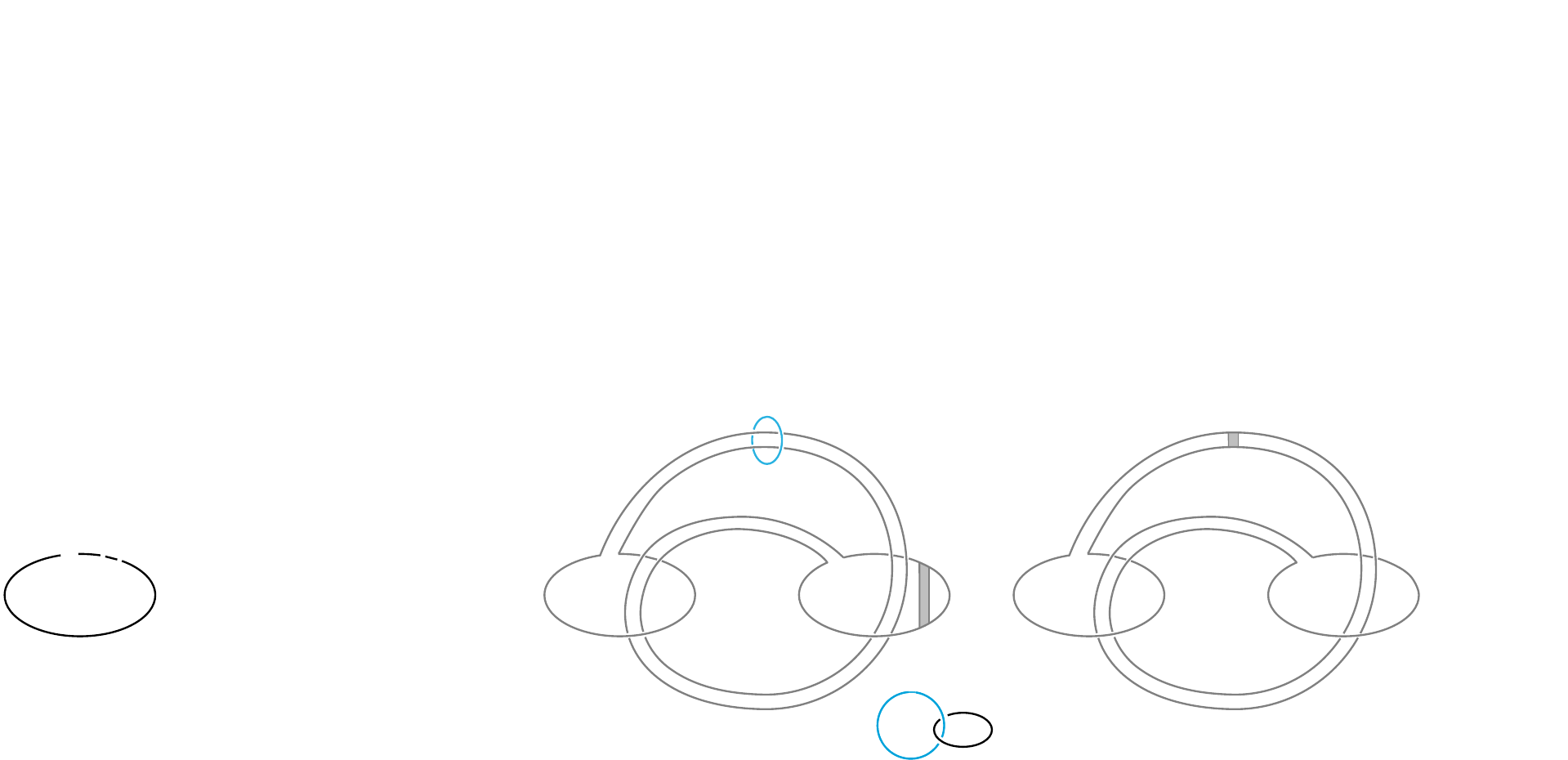
\caption{Trading 2- and 4- dimensional handles. (a) The $D$ in a nonstandard $B^4$ with a canceling band-minima pair. (b) The band-minima pair from (a) have been slid under the right-hand 1-handle and then slid under the left-hand 1-handle by a band that follows the 2-handle. (c) The four-dimensional 1- and 2-handles have been pulled away from $D$ and canceled by a combination of handle-slides and band-handle swim moves. (d) The dual disk $D'$ appears in a nonstandard $B^4$ with a canceling band-maxima pair. (e) The canceling surgery curves are pulled away from the diagram and canceled. (f) The band is slid over the 0-framed 2-handle after which the 2-handle is swum through the band and cancels the 3-handle.}\label{fig:trading}
\end{figure}

There is a conceptually simpler description of this 4-manifold. Given a ribbon disk $D$ with boundary $K$, a handle decomposition of $D$ induces a handle decomposition for $B^4 \smallsetminus D$ \cite[Chapter 6]{GS_4mflds}. Conversely, given a ribbon handle decomposition for $B^4 \smallsetminus D$ we may produce a handle decomposition for $D$ via the following procedure: first, attach a 0-framed 2-handle, $h$, to a meridian of one of the 1-handles in $B^4 \smallsetminus D$; note that $h$ fills the cavity formed by carving out $D$, so what remains is a non-standard handle decomposition for $B^4$. A meridian, $\mu$, of the attaching circle for $h$ can therefore be identified with $K$; moreover, the obvious slice disk for $\mu$ can be identified with $D$. The 1- and 2-handles of $(B^4 \smallsetminus D)\cup h$ can all be canceled to give a diffeomorphism to the standard $B^4$; the disk $D$ can be traced through this diffeomorphism by adding additional minima disks and saddles. We refer to this process as \emph{trading handles}. Of course, this procedure can be dualized, where handles for a for a ribbon disk consisting of saddles and maxima can be traded for 2- and 3-handles of the exterior. Examples of both cases are shown in Figure \ref{fig:trading}.\\

We can apply this trick in the in the last step of the cobordism $P$. Note, that we attached $B^4 \smallsetminus \Delta_{\overline{K}\#K}$ along the splice of $-(\KC)_0$ and $(\KC)_0$; the $\langle 0\rangle$-framed curve $K$ has a earring hanging off, which is precisely the configuration in Frame (d) of \Cref{fig:trading}. Hence, we can trade the handles of $B^4 \smallsetminus \Delta_{\overline{K}\#K}$ for additional handles in the concordance $C$. Concretely, rather than attaching $B^4 \smallsetminus \Delta_{\overline{K}\#K}$, we perform surgery on all the curves shown in Frame (b) of \Cref{fig:lambda_cob1}; the resulting 3-manifold is $S^3$ and through this diffeomorphism, $\nu$ is taken to the knot $K\# \overline{K}\# K$; we attach $\Delta_{\overline{K}\# K}$ to the second two components of this knot, which completes the link cobordism to $(S^3, K)$. For an explicit example of this procedure, see \Cref{fig:lambda_cob2}. 

\Cref{fig:lambda_simplify} shows a simplification of the steps (1) and (2) of this cobordism. We complete the cobordism $(P, C)$ by attaching $\Delta_{\overline{K}\#K}$.

\begin{figure}[h]
\def\svgwidth{\linewidth}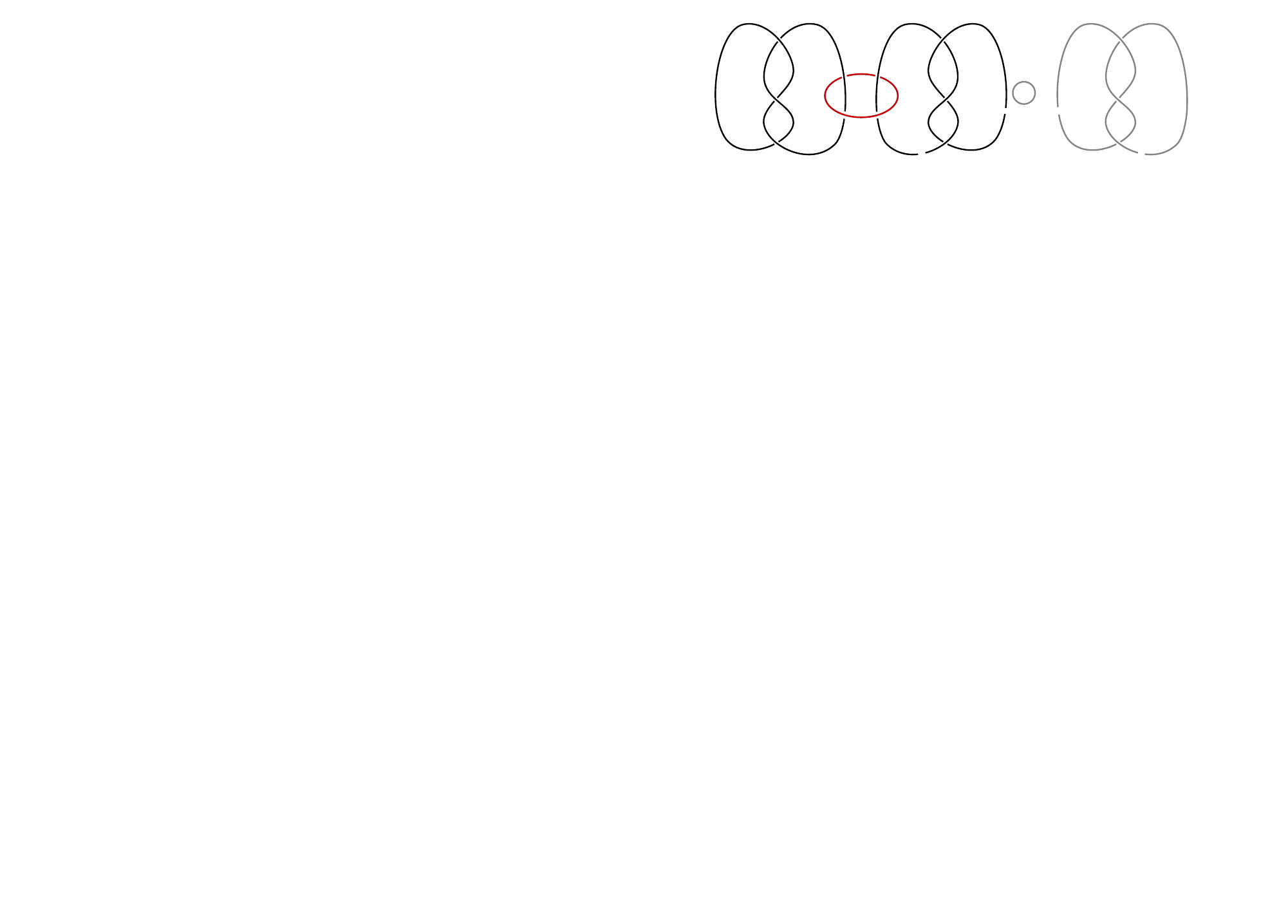
\caption{Canceling handles to simplify $(P,C)$. For simplicity, 3- and 4-handles are dropped from the notation.}\label{fig:lambda_cob2}
\end{figure}

Hence, we have the following efficient description of $(P,C)$. The final frame of \Cref{fig:lambda_simplify} can be built as the composition:
\begin{enumerate}
    \item A birth cobordism
    \begin{align*}
        S^3_n(K\#\overline{K}) \amalg (S^3, K) \ra  (S^3_n(K\#\overline{K}), \lambda) \amalg (S^3, K).
    \end{align*}
    \item A 2-handle cobordism 
    \begin{align*}
        (S^3_n(K\#\overline{K}), \lambda) \amalg (S^3, K) \ra (S^3, K\#\overline{K}) \amalg (S^3, K)
    \end{align*}
    \item A connected sum cobordism 
    \begin{align*}
        (S^3, K\#\overline{K}) \amalg (S^3, K) \ra  (S^3, K\#\overline{K} \# K).
    \end{align*}
\end{enumerate}
This leaves one last step:
\begin{enumerate}
    \item[(4)] Finally, we attach the canonical slice disk $\Delta_{\overline{K}\# K}$ to $K\# \overline{K} \# K$, giving a cobordism 
    \begin{align*}
        (S^3, K\#\overline{K} \# K) \ra (S^3, K).
    \end{align*} 
\end{enumerate}

\begin{figure}[h]
\def\svgwidth{.6\linewidth}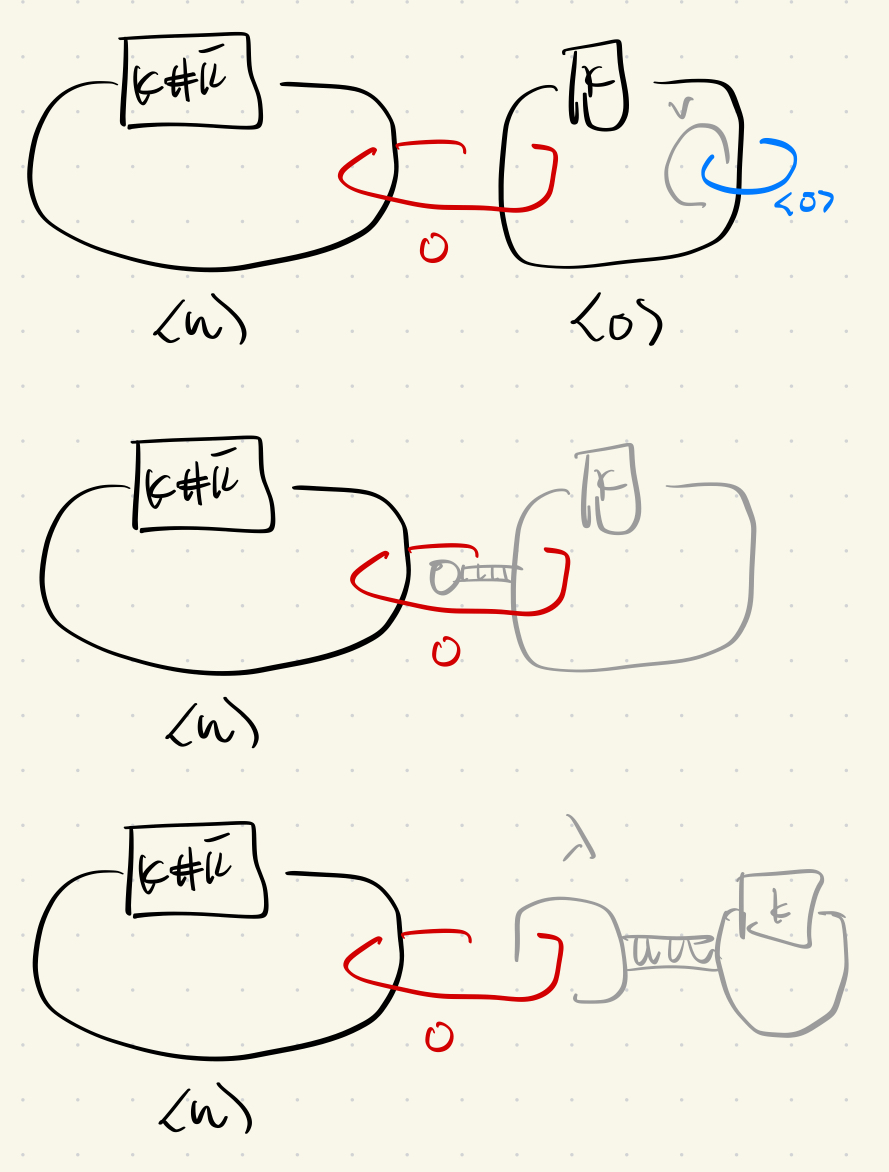
\caption{A simplified handle decomposition for the link cobordism $(P,C)$.}\label{fig:lambda_simplify}
\end{figure}

We would like to associate to $(P, C)$ a cobordism map 
\begin{align*}
    \CFh(S^3_n(K\#\overline{K})) \otimes_{\F} \CFK_\cR(S^3, K) \ra \CFK_\cR(S^3, K),
\end{align*}
using the decomposition of $(P, C)$ above. There is one minor subtlety concerning basepoints. In \cite{zemke_linkcob}, empty links are not considered, so there is no map associated to the birth cobordism 
\begin{align*}
    S^3_n(K\#\overline{K}) \ra (S^3_n(K\#\overline{K}, \lambda)). 
\end{align*}
However, this issue can be avoided by replacing the basepoint in $S^3_n(K\#\overline{K})$ with a local unknot, as $\CFh(S^3_n(K\#\overline{K}))\otimes_\F \cR \simeq \CFK_\cR(S^3_n(K\#\overline{K}, \lambda)$. We shall ignore this component of the cobordism, and think of $(P, C)$ as beginning with the 2-handle link cobordism 
\begin{align*}
    \CFK_\cR(S^3_n(K\#\overline{K}, \lambda)) \ra \CFK_\cR(S^3,K\#\overline{K}).
\end{align*}
Crucially, for large $n$, this cobordism induces the large surgery isomorphism; see \cite[Section 3]{hendricks2022involutive}.

According to \cite[Proposition 5.1]{zemke_connectedsums}, there is a connected sum cobordism map
\[
G_{K_1,K_2}: \CFL(S^3\amalg S^3, K_1 \amalg K_2) \ra \CFL(S^3, K_1 \# K_2)
\]
given by the natural cobordism consisting of a 1-handle/fusion band. This cobordism is second component of $(P, C)$. This map is a homotopy equivalence, with inverse 
\[
E_{K_1\#K_2}: \CFL(S^3, K_1 \# K_2) \ra \CFL(S^3\amalg S^3, K_1 \amalg K_2)
\]
corresponding to attaching a 3-handle/fission band. Moreover, by \cite[Theorem 1.2]{juhasz_zemke_stabilization_bounds}, the map induced by the canonical slice disk, $\Delta_{\overline{K}\# K}$ for $\overline{K}\# K$, viewed as a map
\[
F_D: \CFL(S^3, \overline{K}\# K) \ra \CFL(U) \cong \cR
\]
satisfies 
\[
F_{\Delta_{\overline{K}\# K}} \sim \Tr\circ E_{\overline{K}\# K}
\]
where $\Tr$ is the canonical trace map. In particular, we have the following. 

\begin{lem}\label{lem:compute (Lambda, C)} The map induced by $(P, C)$ is homotopic to the composition
    \[(\id_{\CFK_\cR(K)} \otimes (\Tr\circ E_{K\#(\overline{K}\#K)}))\circ G_{K\#\overline{K},K}\circ (\Gamma_{n,[0]}\circ F_{W_{n}, \fraks_0} \otimes \id_{\CFK(S^3,K,p)}),\]
    where $\Gamma_{n,[0]}$ is the large surgery cobordism map. 
\end{lem}

\begin{proof}
    The composition above corresponds exactly to the simplified handle decomposition of $(P, C)$ discussed above.
\end{proof}

\subsection{Definition and formal properties of $\Lambda$}

Without further ado, we finally formally define our candidate cobordism map. Let $W_n$ be the cobordism $S^3_0(K\#\overline{K})\#L(n,1) \ra S^3_n(K\#\overline{K})$ obtained by attaching a 0-framed two-handle which clasps $K$ and $U$. Here, we think of $L(n,1)$ as the boundary of the disk bundle over $S^2$ with Euler number $n$. The lens space has a distinguished $\Spinc$-structure $\ell$ which bounds a $\Spinc$ structure $\fraks$ on the disk bundle which satisfies $\langle c_1(\fraks), [S^2]\rangle = n.$ We fix an identification $\Spinc(L(n,1)) \cong H^2(L(n, 1);\Z) \cong \Z/n$ by fixing $\ell = [0]$. Let $\Theta$ be the generator of $\HFh(L(n, 1),[0])$. Hence, there is a map 
\begin{align*}
    F_{W_n}(-, \Theta): \CFh(S^3_0(K\#\overline{K})) \ra \CFh(S^3_n(K\#\overline{K}))
\end{align*}

Now, as above, we consider the triple of bordered 3-manifolds: $(T_\infty, \nu), (\KC)_0, (\KC)_n$. The pair of pants construction above defines a cobordism 
\[
(P, C):(-(\KC)_0\cup (\KC)_n)\amalg(-(T_\infty, \nu)\cup (\KC)_0)\ra (-(T_\infty, \nu)\cup (\KC)_n). 
\]
Or, more compactly, $(P, C)$ is a cobordism 
\[
S^3_n(K\#\overline{K}) \amalg (S^3,K) \ra (S^3,K).
\]
We define $(\bm{\curlywedge}, C)$ to be the composition 
\[
(\bm{\curlywedge}, C):= (P, C) \circ (W_n \amalg (S^3, K) \times I).
\]

It will be helpful to understand $\Spinc$-structures on $\bm{\curlywedge}$. Since $P$ and $(W_n \amalg (S^3, K) \times I)$ are glued along rational homology spheres, a $\Spinc$-structure on $\bm{\curlywedge}$ is uniquely specified by its restriction to $P$ and $W_n$. Consider again the handle decomposition for $\bm{\curlywedge}$ shown in \Cref{fig:P_handle_decomp}. Note that the 1-handle connecting the two components of $\partial_- \bm{\curlywedge}$ has already been attached. Furthermore, we have only drawn 4-dimensional handles -- the handles of the surface $C$ belonging to the slice disk $\Delta_{\overline{K}\# K}$ are not included, as they have no effect on the  $\Spinc$-structures. The diagram in \Cref{fig:P_handle_decomp} is obtained from our description of $(P,C)$ above by precomposing with the cobordism $W_n$, and performing a single handle slide to simplify the following computations.

$C_1(\bm{\curlywedge};\Z)$ is generated by the meridians of bracketed surgery curves (and in particular is rank 4) and $C_2(\bm{\curlywedge};\Z)$ is generated by the various 2-cells which are attached during surgery or during the cobordism (it is rank 6 generated by the cells labeled $a,\hdots,f$.) We will write $A, \hdots, F$ for the dual generators of $C^2$. A quick computation shows that $H^2(\bm{\curlywedge};\Z) \cong \Z\langle A, C \rangle$. We can see a handle decomposition for $W_n$ as well: simply erase the curves labeled $d$, $e$, and $f$. A similar computation shows that 
\begin{align*}
    H^2(W_n) & \cong \Z/n \langle A \rangle \oplus \Z \langle C \rangle \\ 
    H^2(\partial_- W_n) & \cong \Z/n \langle A \rangle \oplus \Z \langle C \rangle \\ 
    H^2(\partial_+ W_n) & \cong \Z/n \langle A = C \rangle 
\end{align*}
and the restriction maps $H^2(W_n) \ra H^2(\partial_- W_n)$ and $H^2(W_n) \ra H^2(\partial_+ W_n)$ are given by $([A], C) \mapsto ([A],C)$ and $([A], C) \mapsto [A+C]$ respectively. The restriction $H^2(\bm{\curlywedge}) \ra H^2(W_n)$ is given by $(A, C) \mapsto ([A],C).$

Now, consider the cobordism given by attaching the 2-handle labeled $d$ to $\partial_+ W_n$; this was the cobordism $\Gamma_n$. A sequence of handle slides shows that the ingoing end of $\Gamma_n$ is $S^3$. In this case, we have
\begin{align*}
    H^2(\Gamma_n;\Z) & \cong \Z\langle A = C \rangle,
\end{align*}
and the restriction to $H^2(\partial_-\Gamma_n;\Z)$ is given by $A \mapsto [A]$ and the other restriction is of course trivial. The restriction $H^2(\bm{\curlywedge}) \ra H^2(\Gamma_n)$ is given by $A, C \mapsto A + C$.

\begin{figure}[h]
\def\svgwidth{.6\linewidth}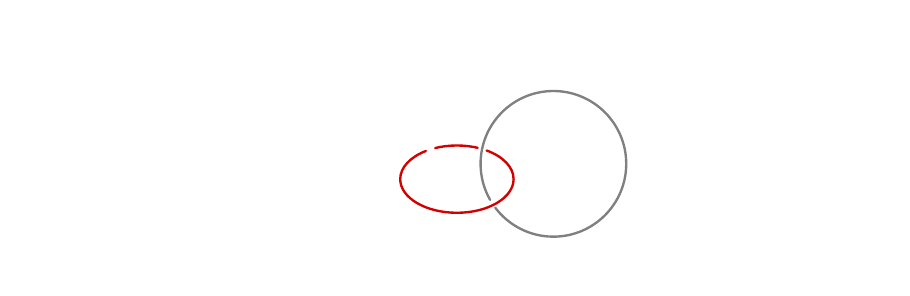
\caption{A handle decomposition for $\bm{\curlywedge}$; the generators for the second homology are denoted $a, b, \hdots, f$.}
\label{fig:P_handle_decomp}
\end{figure}

Note that $\bm{\curlywedge} \smallsetminus (W_n \cup \Gamma_n)$ has a single $\Spinc$-structure (this portion of the cobordism is diffeomorphic to $S^3 \times I$), so it is enough to choose a $\Spinc$-structure on $W_n \cup \Gamma_n.$ $\Spinc$-structures on $W_n$ are determined by their restrictions to $L(n,1) \# S^3_0(K\# \overline{K}) \subset \partial_- W_n$; hence, we choose the $\Spinc$-structure $\frak{u}_0$ on $W_n$ which extends the $\Spinc$-structure $[0] \oplus [0]$. On $\Gamma_n$, we must then choose a $\Spinc$-structure which restricts to $[0]$ on $S^3_n(K\# \overline{K})$. We choose the $\Spinc$-structure $\frak{y}_0$ which satisfies 
\begin{align*}
    \langle c_1(\frak{y}), [F] \rangle + n = 0,
\end{align*}
where $F$ is a surface in $\Gamma_n$ obtained by capping off the cocore of the 2-handle attached along $d$ with a Seifert surface for the image of $\nu$ in $S^3$.

Therefore, we define $\fraks_0 = \frak{u}_0 \# \frak{y}_0$ to be the unique $\Spinc$-structure on $\bm{\curlywedge}$ which extends $\frak{u}_0$ and $\frak{y}_0$. 

\begin{defn}
    Let $\Lambda$ be the map
    \begin{align*}
        \Lambda:= F_{\bm{\curlywedge}, C, \fraks_0}(\Theta, -): \CFh(S^3_0(K\#\overline{K}), [0]) \otimes_\F \CFK_\cR(S^3, K) \ra \CFK_\cR(S^3,K),
    \end{align*}
    where $(\bm{\curlywedge}, C) = (P\cup W_n, C)$ and $n$ is large. Under the morphism paring theorem, $\Lambda$ therefore determines a map
    \[
        \Lambda: \End^\cA_0(\CFDh(\KC)) \ra \End_\cR(\CFK_\mathcal{R}(S^3,K)),
    \]
    where $\End^\cA_0(\CFDh(\KC)) \subset \End^\cA(\CFDh(\KC))$ is the summand corresponding to $\CFh(S^3_0(K\#\overline{K}), [0])$.
\end{defn}

\begin{rem}
    We again emphasize that we are abusing notation here. Technically, the cobordism defines a map 
    \begin{align*}
        \CFK_\cR(S^3_0(K\#\overline{K}), \lambda) \otimes_\cR \CFK_\cR(S^3, K) \ra \CFK_\cR(S^3, K),
    \end{align*}
    but, under the identification $\CFh(S^3_0(K\#\overline{K}), [0])\otimes_\F \cR \simeq \CFK_\cR(S^3_0(K\#\overline{K}), \lambda, [0])$, $\Lambda$ takes the above form.
\end{rem}

This map has a straightforward interpretation: a morphism $f \in \End^\cA_0(\CFDh(\KC)) \simeq \CFh(S^3_0(K\#\overline{K}),[0])$ is mapped to $\CFh(S^3_n(K\#\overline{K}))$ by the cobordism $W_n$; provided $n$ is sufficiently large, $f$ will land in $\CFh(S^3_n(K\#\overline{K}), [0])$. This group is identified via the large surgery isomorphism with $\hat{\mathbb{A}}_0(K\#\overline{K}) \subset \CFK_\cR(K\#\overline{K})$. Under the identification of $\CFK_\cR(K\#\overline{K}) \simeq \End_\cR(\CFK_\cR(S^3, K))$, the connected sum cobordism followed by the disk map for $\Delta_{\overline{K}\#K}$ correspond to evaluation 
\begin{align*}
    \End_\cR(\CFK_\cR(S^3, K)) \otimes \CFK_\cR(S^3, K) \ra \CFK_\cR(S^3, K). 
\end{align*}

Some of the formal properties of $\Lambda$ are immediate from the definition. 
\begin{lem}\label{lem:composition}
    Viewed as a map $\Lambda: \End_0^\cA(\CFDh(\KC)) \ra \End_\cR(\CFK_\mathcal{R}(S^3,K,p))$, $\Lambda$ satisfies: 
    \[
        \Lambda(g) \circ \Lambda(f) \sim \Lambda(g \circ f),
    \]
    for any $f, g \in \End_0^\cA(\CFDh(\KC))$.
\end{lem}
\begin{proof}
    It suffices to prove the following diagram is homotopy commutative.   
    \begin{center}
        \begin{tikzcd}[column sep = huge]
            \End^\cA_0(\CFDh(\KC))^{\otimes 2} \otimes \CFK_\mathcal{R}(S^3,K) \ar[r,"\id\otimes \Lambda"]\ar[d, "\circ^{op} \otimes \id"] &
            \End^\cA_0(\CFDh(\KC)) \otimes \CFK_\mathcal{R}(S^3,K) \ar[d,"\Lambda"] \\
            \End^\cA_0(\CFDh(\KC)) \otimes \CFK_\mathcal{R}(S^3,K)
            \ar[r,"\id\otimes \Lambda"] &
            \CFK_\mathcal{R}(S^3,K)
        \end{tikzcd}
    \end{center}
    This is purely topological. The clockwise composition is induced by the cobordism 
    \[
    (\bm{\curlywedge}, C) \cup_{(S^3, K)} (\bm{\curlywedge}, C): L(n,1)\#S^3_0(K \# \overline{K}) \amalg L(n,1)\#S^3_0(K \# \overline{K}) \amalg (S^3, K) \ra (S^3, K). 
    \]
    Let $\Omega$ be the pair of pants cobordism obtained by gluing together $L(n,1)\#(\KC)_0$, $L(n,1)\#(\KC)_0$, and $(\KC)_0$. The counter-clockwise composition is induced by the cobordism 
    \[
    (\Omega \amalg (S^3, K) \times I) \cup_{S^3_0(K\# \overline{K}) \amalg(S^3, K,p)} (\bm{\curlywedge}, C): S^3_0(K \# \overline{K}) \amalg S^3_0(K \# \overline{K}) \amalg (S^3, K) \ra (S^3, K). 
    \]
    From the handle decomposition of $(\bm{\curlywedge}, C)$, it is clear that these two cobordisms are diffeomorphic rel boundary (the two cobordisms only differ by the order the 2-handles are attached). Hence, they induce homotopic maps. 
\end{proof}

As $\Lambda$ is induced by a cobordism map, it follows from the conjugation invariance of the link cobordism maps that $\Lambda$ is compatible with involutions.
\begin{lem}\label{lem: lambda commutes with involutions}
    The map $\Lambda$ makes the following diagram homotopy-commutative.
    \[
    \xymatrix{
    \mathrm{End}^\mathcal{A}_0(\widehat{CFD}(S^3 \smallsetminus K)) \ar[rr]^{\Lambda} \ar[d]_{f\mapsto \iota_{S^3 \smallsetminus K}\circ (\id\boxtimes f) \circ \iota^{-1}_{S^3 \smallsetminus K}} && \mathrm{End}_\cR(CFK_\mathcal{R}(S^3,K)) \ar[d]^{f\mapsto \iota_K \circ f \circ \iota^{-1}_K} \\
     \mathrm{End}^\mathcal{A}_0(\widehat{CFD}(S^3 \smallsetminus K)) \ar[rr]^{\Lambda} && \mathrm{End}_\cR(CFK_\mathcal{R}(S^3,K))
    }
    \]
\end{lem}
\begin{proof}
    It is shown in the discussions preceding \cite[Section 3]{kang_bordered_involutive_HFK} that the conjugation action 
    \[
    f\mapsto \iota_{S^3 \smallsetminus K}\circ (\id\boxtimes f) \circ \iota^{-1}_{S^3 \smallsetminus K}
    \]
    is identified with the $\iota$-action on $\widehat{CF}(S^3 _0 (K\# \overline{K}))$ under the morphism pairing theorem. Hence the lemma follows from \cite[Theorem 1.3]{zemke_connectedsums} and the fact that $C_0$ is a cylinder.
\end{proof}

\begin{cor}\label{cor: invariant splittings CFA to CFK}
    Let $p\in \mathrm{End}_\mathcal{A}(\CFDh(\KC))$ be a projection map which produces an $\iota_{\KC}$-equivariant splitting; note that this is equivalent to the statement that 
    \[
    p \sim \iota_{\KC}\circ (\bI \boxtimes p)\circ \iota^{-1}_{\KC}.
    \]
    Then the map $\Lambda(p)$ is homotopic to an $\iota_K$-equivariant projection map.
\end{cor}
\begin{proof}
    It follows from \Cref{lem: lambda commutes with involutions} that $\Lambda(p)$ is homotopy $\iota_K$-equivariant and from \Cref{lem:composition} that $\Lambda(p)$ is a homotopy projection, as 
    \begin{align*}
        \Lambda(p)^2 \sim \Lambda(p^2) = \Lambda(p). 
    \end{align*}
    According to \Cref{lem: homotopy projection to projection}, $\Lambda(p)$ can be homotoped to an honest projection. To upgrade this to honest equivariant projection, we simply have to replace $\iota_K$ in its homotopy class by
    \[
    \iota^\prime_K = \iota_K + [\Lambda(p),\iota_K],
    \]
    which commutes with $\Lambda(p)$,
    \[
    [\Lambda(p),\iota^\prime_K] = [\Lambda(p),\iota_K] + [\Lambda(p),[\Lambda(p),\iota_K]] = 0
    \]
    since $\Lambda(p)^2 = \Lambda(p)$, as desired.
\end{proof}

The reverse direction follows from the previous lemmas and \Cref{prop: bijective on projection} (which we will prove in the next subsection).

\begin{cor}\label{cor: invariant splittings CFK to CFA}
    Let $p\in \mathrm{End}_\mathcal{A}(\CFDh(\KC))$ be an endomorphism such that $\Lambda(p)$ is an $\iota_K$-equivariant projection. Then $p$ is an $\iota_{\KC}$-equivariant projection.
\end{cor}
\begin{proof}
    It follows from \Cref{prop: bijective on projection} that $p$ is a projection. Consider the conjugate projection
    \[
    p^\prime = \iota_{\KC} \circ (\bI\boxtimes p)\circ \iota^{-1}_{\KC}.
    \]
    By \Cref{lem: lambda commutes with involutions}, this implies that $\Lambda(p^\prime)\sim\Lambda(p)$. Since $\Lambda(p^\prime)$ is also a projection, applying \Cref{prop: bijective on projection} again tells us that $p\sim p^\prime$, i.e. $p$ is homotopy $\iota_{\KC}$-equivariant. As in the proof of \Cref{cor: invariant splittings CFA to CFK}, $\iota_\KC$ can be changed within its homotopy class such that $p$ becomes honestly $\iota_{\KC}$-equivariant.
\end{proof}

\subsection{The Bordered Perspective}

In order to relate splittings of $\CFK_\cR(S^3, K)$ induced by $\Lambda$ to splittings coming from the \LOT correspondence, we need a bordered interpretation of $\Lambda$. 

This is straightforward for the cobordism $W_n$. We can express $W_n = X_n \cup (\KC \times I)$, where $X_n$ is obtained by attaching the 2-handle of $W_n$ to $L(n,1)\#(\KC)$. The cobordism map induced by $W_n$ can be computed as follows.

\begin{lem}\label{lem: X_n commutes}
Let $X_n: (\KC)_0\#L(n,1) \ra (\KC)_n$ be the cobordism with corners obtained from $W_n$ by drilling out the thickened annulus $\nu(K) \times I$. Then, there is a map $F_{X_n}: \CFDh((\KC)_0\#L(n,1))\ra \CFDh((\KC)_n)$ defined by counting holomorphic triangles which makes the following diagram
    \begin{center}
        \begin{tikzcd}
            \CFAh(-(\KC)_0)\boxtimes \CFDh(((\KC)_0\#L(n,1)) \ar[r] \ar[d,"\bI \boxtimes F_{X_n}"] 
            &  \CFh(S^3_0(K\#\overline{K})\#L(n,1)) \ar[d,"F_{W_n}"] \\
            \CFAh(-(\KC)_0)\boxtimes \CFDh((\KC)_n)  
            \ar[r]
            &  \CFh(S^3_n(K\#\overline{K}))
        \end{tikzcd}
    \end{center}
    commute up to homotopy. The horizontal arrows are given by the \LOT pairing theorem. Moreover, if $n$ is sufficiently large, the map $F_{W_n,\frak{u}_0}(-,\Theta)$ can be computed from the map $\bI \boxtimes F_{X_n}$.
\end{lem}

Throughout this section, we establish the convention:
$$\Mor^\cA(Y_1, Y_1):= \Mor^{\cA}(\CFDh(Y_1), \CFDh(Y_2)).$$ 
We also note that reformulating this lemma in terms of the morphism pairing theorem shows that the cobordism map $F_{W_n}$ is homotopic to the map 
\[
\Mor^\cA((\KC)_0, (\KC)_0) \ra \Mor^\cA((\KC)_0, (\KC)_N)
\]
given by post-composition with $F_{X_n}$. See \Cref{lem:composition_to_boxtensor}.

\begin{proof}
    Fix a doubly pointed Heegaard diagram $\cH_K = (\Sigma, \alpha, \beta, w, z)$ for $K$ in $S^3$. Attach a two-dimensional 1-handle, which we identify with $S^1 \times [-1, 1]$, to $\cH_K$ with feet near $w$ and $z$. Add two additional alpha circles: $\alpha_0$ which we identify with $S^1 \times \{0\}$ and $\alpha_{1}$ which is a longitude of $K$ in $\Sigma$ and intersects $\alpha_0$ in single point, $P$; next we attach an additional beta circle, $\beta_0$ which is geometrically dual to $\alpha_0$. Let $\cH_0$ be the bordered Heegaard diagram $(\Sigma \smallsetminus \nu(P), \alpha_0^a,\alpha_1^a, \alpha,\beta_0\cup \beta)$, where $\alpha_0^a$ and $\alpha_1^a$ are the arcs $\alpha_0 \smallsetminus \nu(P)$ and $\alpha_1 \smallsetminus \nu(P)$ respectively. We also ensure that $\beta_0$ wraps around the new 1-handle sufficiently many times to ensure the boundary is equipped with the $0$-framing.\\ 

    Let $\gamma$ be a collection of curves in $\Sigma$ obtained by Hamiltonian isotopies of the original $\beta$ curves. Let $\gamma_n$ be obtained from $\beta_0$ by winding $n$ additional times around the 1-handle. Hence, the bordered Heegaard diagram $\cH_n:=(\Sigma \smallsetminus \nu(P), \alpha_0^a,\alpha_1^a, \alpha, \gamma_n\cup \gamma)$ represents the $n$-framed complement of $K$. The curves $\beta_0$ and $\gamma_n$ intersect in $n$ points, $x_0, \hdots, x_{n-1}$; in fact, $\CFh(\Sigma, \beta_0\cup \beta, \gamma_n\cup \gamma)\simeq \CFh(\#^{g(\Sigma)-1}S^1\times S^2 \# L(n,1))$. Let $\Theta_0^{\beta, \gamma}$ be the top graded intersection point in the zeroth Spin$^c$ structure. Counting holomorphic triangles yields a map 
    \[
    \CFDh(\cH_0) \xrightarrow{m_2({-}, \Theta^{\beta, \gamma}_0)} \CFDh(\cH_n).
    \]
    Finally, fix a diagram $\cH = (\Sigma', \alpha_0^{a'}, \alpha_1^{a'}, \alpha',\beta')$ for $-(\KC)_0$. Let $\gamma'$ be a push off of the $\beta'$ curves and let $\Theta^{\beta', \gamma'}$ be a top graded generator of $\CFh(\beta', \gamma')$. We may now apply the pairing theorem for triangles to obtain the homotopy commutative square
    \begin{center}
        \begin{tikzcd}[column sep = huge]
        \CFAh(\cH)\boxtimes \CFDh(\cH_0) \ar[d, "\simeq"]  \ar[rr, "m_2({-}{,}\Theta^{\beta', \gamma'}) \boxtimes m_2({-} {,} \Theta^{\beta, \gamma})"]& &
         \CFAh(\cH)\boxtimes \CFDh(\cH_N)   \ar[d,"\simeq"] \\
        \CFh(\cH\cup \cH_0)   \ar[rr, "m_2({-}{,}\Theta^{\beta, \gamma} \cup \Theta^{\beta', \gamma'})"] & &
           \CFh(\cH\cup \cH_N).
        \end{tikzcd}
    \end{center}
    Under the identification of $(\Sigma', \alpha_0^{a'}, \alpha_1^{a'}, \alpha',\beta')$ with $(\Sigma', \alpha_0^{a'}, \alpha_1^{a'}, \alpha',\gamma')$, the map $m_2({-},\Theta^{\alpha', \beta'})$ is the identity. Moreover, the Heegaard triple obtained by gluing together the two bordered Heegaard triples above is, by construction, subordinate to the framed link defining the cobordism from $S^3_0(K\#\overline{K})\# L(n,1) \ra S^3_n(K\#\overline{K})$ (see, for instance \cite[Theorem 3.1]{os_HFK_integer_surgeries}). Therefore, $m_2({-}, \Theta^{\alpha, \beta} \cup \Theta^{\alpha', \beta'})$ is precisely this cobordism map. Let $\Theta_B := \Theta^{\alpha, \beta} \cup \Theta^{\alpha', \beta'}$.

    Finally, the map $\bI\boxtimes F_{X_n}({-},\Theta_B)$ computes the map $\sum_{\fraks \in \Spinc(W_n)} F_{W_n,\fraks}({-},\Theta)$. As $\Spinc$-structures on $F_{W_n}$ are determined by their restrictions to $\partial_-W_n \cong S^3_0(K\#\overline{K})\# L(n,1)$, it is clear that the restriction of $\bI\boxtimes F_{X_n}({-},\Theta_B)$ to $\CFh(S^3_0(K\#\overline{K}),[0])$ is homotopic to $F_{W_n,\frak{u}_0}({-},\Theta)$. Since this map is supported on finitely many $\Spinc$-structures, as long as $n$ is large, the image will be contained in $\CFh(S^3_n(K\#\overline{K}),[0])$.
\end{proof}

\begin{rem}
    Going forward, we will abuse notation, and write $F_{X_n}$ for the map $F_{X_n}({-},\Theta_B)$. 
\end{rem}

The cobordism $(P, C)$ is slightly more problematic. By the work of Cohen \cite{cohen2023composition}, the pair of pants cobordism associated to a triple $(T_\infty, \nu), (\KC)_0, (\KC)_n$: 
\[
(P, C):(-(\KC)_0\cup (\KC)_n)\amalg(-(T_\infty, \nu)\cup (\KC)_0)\ra (-(T_\infty, \nu)\cup (\KC)_n). 
\]
realizes the map 
\begin{align*}
    \Mor^\cA((\KC)_0, (\KC)_n) \otimes \CFA^-(T^\infty, \nu)\boxtimes\CFDh((\KC)_0) \ra \CFA^-(T^\infty, \nu)\boxtimes\CFDh((\KC)_n)
\end{align*}
given by $f\otimes (a \boxtimes x)\mapsto \bI\boxtimes f(a \boxtimes x)$. Here, $\CFA^-(T^\infty, \nu)\boxtimes\CFDh((\KC)_0) \simeq \CFK^-(S^3, K)$. Here, $\CFK^-(S^3, K)$ is the complex whose differential is defined by counting holomorphic curves are not allowed to cover one of the two basepoints. As the $z$-basepoint of $\nu$ is contained in the region of $(T^\infty, \nu)$ containing the Reeb chord $\rho_0$, this construction does not give rise to a map to $\CFK_\cR(S^3, K)$. 

To circumvent this issue, we replace the bordered Heegaard diagram $(T^\infty, \nu)$ with the diagram $\bX$. The associated pair of pants cobordism is topologically identical to $(P, C)$, but the presence of the free basepoint $p$ changes the configuration of basepoints. Above, we replaced the basepoint $q$ in $S^3_n(K \# \overline{K})$ with a local unknot $\lambda$ which is ultimately merged with $K$, and whose basepoints are quasi-destabilized in the connected sum cobordism. But, in the bordered setting, the basepoints are contained in the boundaries of the various bordered 3-manifolds. When the pair of pants cobordism is built by gluing together $(T_\infty, \nu)\times I, (\KC)_0\times I,$ and $(\KC)_n\times I$, the basepoints form a trivalent graph $\frak{G}$, which is disjoint from $C$. By replacing the basepoints $p$ and $q$ with local unknots $U_p$ and $U_q$, and the graph $\frak{G}$ with the link cobordism $\partial(\frak{G}\times D^2)$, we obtained a well-defined cobordism map 
\begin{align*}
    F_{P, C, \frak{G}}: \CFK_\cR(S^3_0(K\#\overline{K}), U_q)\otimes_\cR \CFK_\cR(S^3, K\cup U_p) \ra \CFK_\cR(S^3, K\cup U_p),
\end{align*}
associated to the triple $(P, C, \frak{G})$. We will at times abbreviate $F_{P, C, \frak{G}}$ by $F_{\frak{G}}$. We will continue to implicitly identify $\CFK_\cR(S^3_0(K\#\overline{K}), U_q) \simeq \CFh(S^3_0(K\#\overline{K}))\otimes_\F \cR$.

\begin{lem}{\cite[Theorem 1.1]{cohen2023composition}}\label{lem:cohen}
    Let $K$ be a knot in $S^3$ and let $\bX$ be the triply pointed Heegaard diagram for a fiber in the solid torus from Section \ref{sec: background}. Then, the pair of pants link cobordism $(P, C,\frak{G})$ induced by the triple $\bX$, $\KC$, and $\KC$ fits into a homotopy commutative square 
    \begin{center}
        \begin{tikzcd}
            \Mor^\cA((\KC)_0, (\KC)_n) \otimes_\F \CFAx\boxtimes\CFDh((\KC)_0)   
            \ar[r]\ar[d,"f\otimes a\boxtimes x \mapsto \bI \boxtimes f(a\boxtimes x)"]
            &  \CFh(S^3_n(K\#\overline{K})) \otimes_\F \CFK_\cR(S^3,K,p)\ar[d,"F_{\frak{G}}"]\\
            \CFAx\boxtimes\CFDh((\KC)_n)  \ar[r] & 
                \CFK_\cR(S^3,K,p)
        \end{tikzcd}
    \end{center}
    where the horizontal arrows are given by the pairing theorem. 
\end{lem}

\begin{proof}
    Since the $\cA_\infty$-operations on $\CFA_{\cR}(\bX)$ are defined in terms of rigid holomorphic curves which do not cover the boundary basepoint, the same argument as in the proof of \cite[Theorem 1.1]{cohen2023composition} applies.
\end{proof}

Using this cobordism, we obtain a map which closely resembles $\Lambda$. 

\begin{defn}
    Let $\Lambda_p$ be the map
    \begin{align*}
        \Lambda_p:=F_{P, C, \frak{G}, \frak{y}_0} \circ F_{W_n, \frak{u}_0}(\Theta, -): \CFh(S^3_0(K\#\overline{K}), [0]) \otimes_\F \CFK_\cR(S^3, K, p) \ra \CFK_\cR(S^3,K,p),
    \end{align*}
    where $(\bm{\curlywedge}, C) = (P\cup W_n, C)$ and $n$ is large. Under the morphism paring theorem, $\Lambda_p$ therefore determines a map
    \[
        \Lambda_p: \End^\cA_0(\CFDh(\KC)) \ra \End_\cR(\CFK_\mathcal{R}(S^3,K, p)),
    \]
    where $\End^\cA_0(\CFDh(\KC)) \subset \End^\cA(\CFDh(\KC))$ is the summand corresponding to $\CFh(S^3_0(K\#\overline{K}), [0])$.
\end{defn}

To relate $\Lambda$ and $\Lambda_p$, we introduce a third map $\Lambda_\bullet$. Let $(\bm{\curlywedge},C, \gamma)$ be the cobordism obtained from $(\bm{\curlywedge},C)$ by choosing a path from a point in $S^3_0(K\#\overline{K}) \subset \partial_-(\bm{\curlywedge},C)$ to a point in $(S^3,K) \subset \partial_+(\bm{\curlywedge},C)$ away from $K$. We define $\Lambda_\bullet$ to be the associated map,
\begin{align*}
    \CFK_\cR(S^3_0(K\#\overline{K}), U_q) \otimes_\cR \CFK_\cR(S^3,K, p) \ra \CFK_\cR(S^3,K,p).
\end{align*}

For any link cobordism, $(W, \cF)$, adding a free basepoint arc has a very predictable effect on the induced map. 

\begin{figure}[h]
\def\svgwidth{.5\linewidth}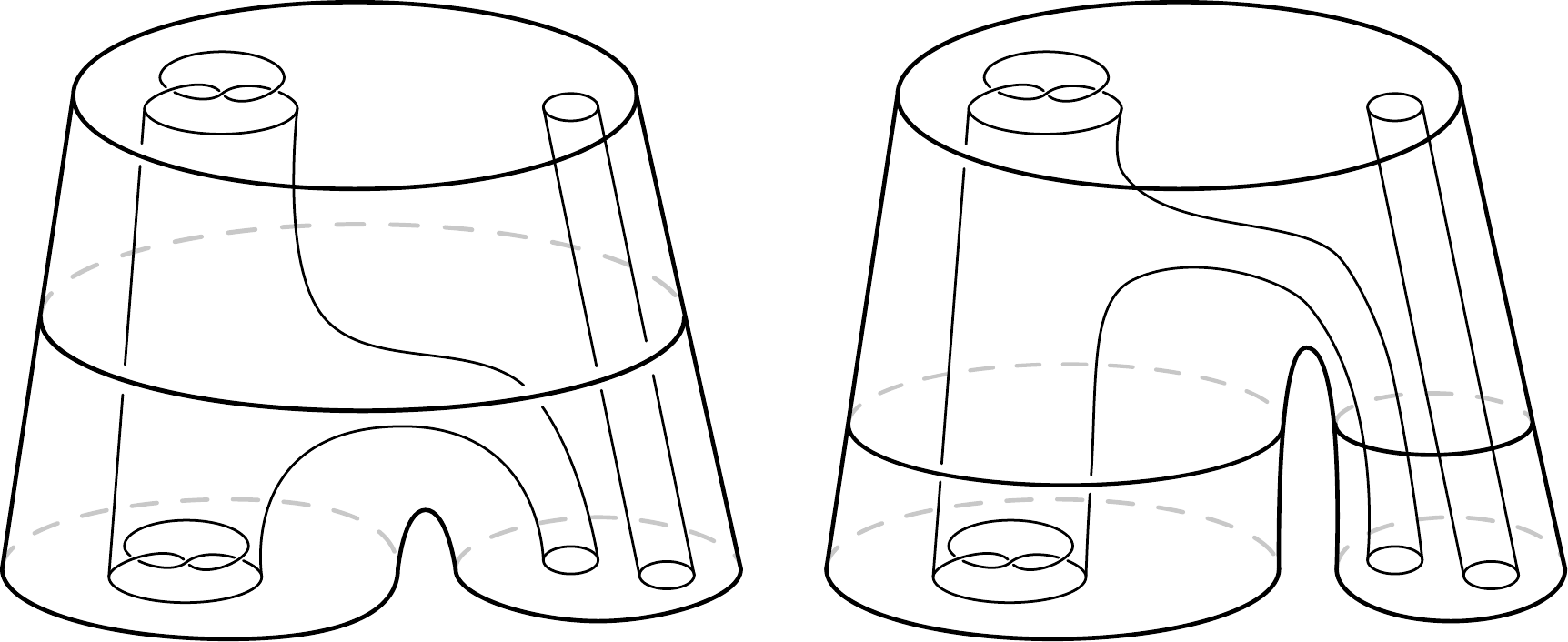
\caption{The cobordisms $F_{W, \cF\cup C_\gamma}\circ G$ and $G\circ (F_{W,\cF}\otimes F_{S^3\times I, (U\cup U) \times I})$}\label{fig:basepoint_cob}
\end{figure}

\begin{lem} \label{lem:free basepoint doubles everything}
    Let $(W, \cF): (Y, K) \ra (Y', K')$ be a decorated link cobordism in the sense of \cite{zemke_linkcob}. Let $p$ and $p'$ be free basepoints in $Y$ and $Y'$ respectively and let $\gamma$ be a path in $W$ connecting $p$ and $p'$. Replace $p$ and $p'$ by local unknots $U_p$ and $U_{p'}$ and $\gamma$ by a cylinder $C_\gamma$ decorated by a pair of parallel arcs. Then, 
    \[
    F_{W, \cF \cup C_\gamma} \circ G_{K,U} \sim G_{K',U} \circ (F_{W, \cF} \oplus  F_{W, \cF}).
    \]
\end{lem}

\begin{proof}
    This follows from the fact that the two compositions correspond to diffeomorphic cobordisms; see Figure \ref{fig:basepoint_cob}). The result follows.
\end{proof}

In particular, after identifying $CFK_\cR(S^3,K)\oplus CFK_\cR(S^3,K)$ and $CFK_\cR(S^3,K,p)$ via the connected sum map, we have that $\Lambda_\bullet \simeq \Lambda\oplus \Lambda$. 

Next, we relate $\Lambda_\bullet$ to $\Lambda_p$. The idea is very similar to the proof above. We will precompose with a simple connected sum cobordism, and relate the resulting 4-manifolds. Again, we will replace free basepoints with local unknots.

\begin{figure}[h]
\def\svgwidth{.8\linewidth}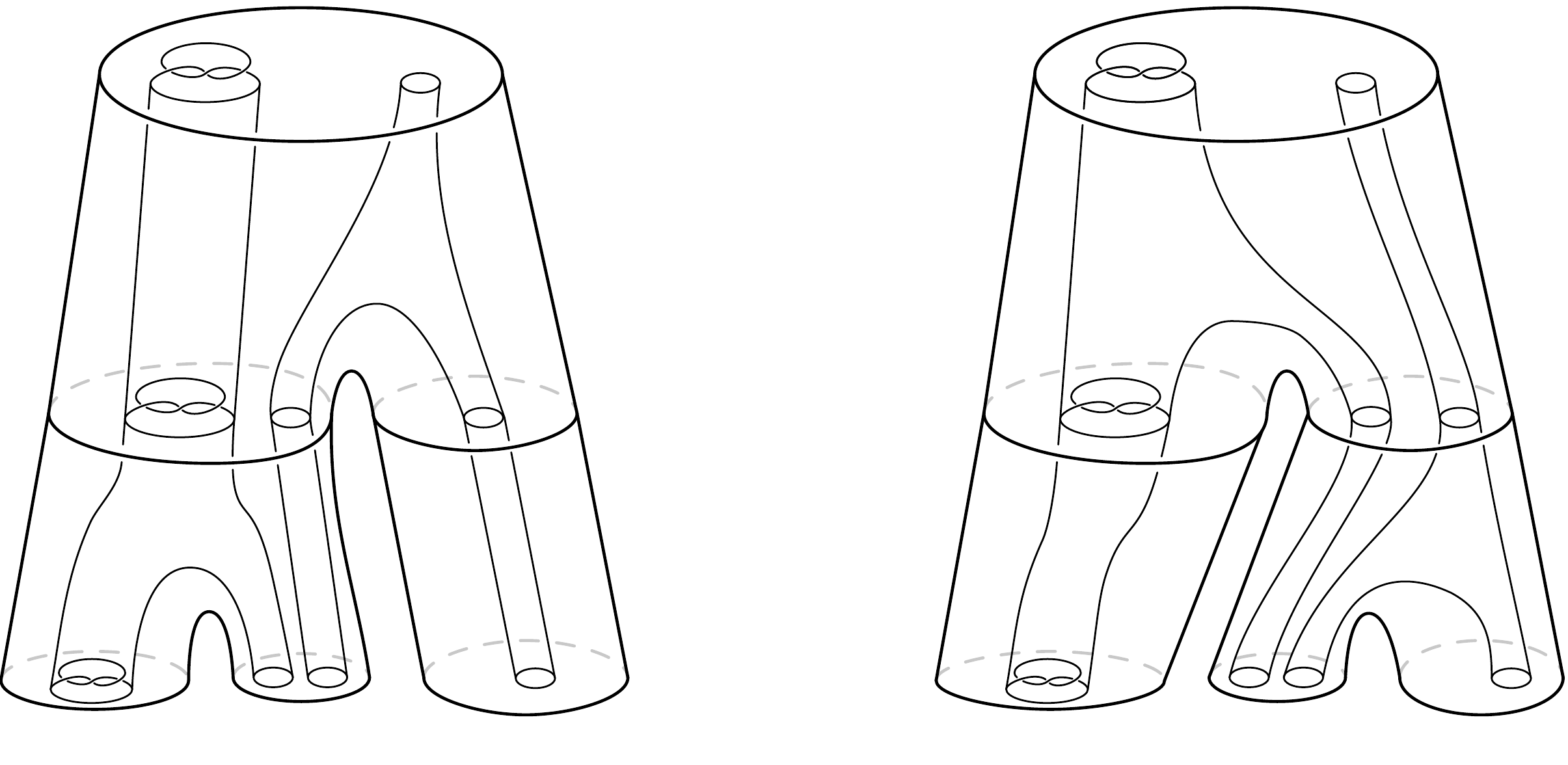
\caption{A diffeomorphism giving a relation between the cobordisms $\Lambda_\bullet$ and $\Lambda_p$.}
\label{fig:trivalent_graph_cob}
\end{figure}

\begin{lem}\label{lem: lambda bullet is lambda p}
    There is a homotopy commutative diagram: 
    \begin{align*}
    \begin{tikzcd}[ampersand replacement = \&, column sep = large]
        \CFK_\cR(S^3_0(K\#\overline{K}), U_r) \otimes_\cR \CFK_\cR(S^3, U_p\cup U_q) \otimes_\cR \CFK_\cR(S^3,K) \ar[r,"\id \otimes G_{U_p,K}"]\ar[d,"G_{U_r,U_q}\otimes \id"] \& \CFKh(S^3_0(K\#\overline{K}), \lambda\cup U_p) \otimes_\cR \CFK_\cR(S^3,K) \ar[d,"\Lambda_\bullet"]\\
        \CFK_{\cR}(S^3_0(K\#\overline{K}), \lambda) \otimes_\cR \CFK_\cR(S^3, K\cup U_p) \ar[r,"\Lambda_p"] \& \CFK_\cR(S^3,K\cup U_p),
    \end{tikzcd}
    \end{align*}    
    where $\Lambda_\bullet$ is the map induced by the cobordism obtained from $(\bm{\curlywedge}, C)$ by adding a free basepoint arc. In particular, the homotopy equivalences $G_{U_p,U_q}$ and $G_{K,U}$ identify $\Lambda_p$ and $\Lambda_\bullet.$
\end{lem}

\begin{proof}
    This is again topological. \Cref{fig:trivalent_graph_cob} shows two decompositions of the same cobordism; the first cobordism corresponds to the counter-clockwise composition, and the second corresponds to the clockwise composition. 

    Recall that in our description of the cobordism $(P,C)$, a local unknot $\lambda$ is birthed in $S^3_n(K\#\overline{K})$ which merges with the knot $K$. In the first frame of \Cref{fig:trivalent_graph_cob}, the birth cobordism $\cD: \emptyset \ra \lambda$ is absorbed into $C$ and so $\lambda$ does not appear in the upper boundary of the connected sum cobordism. In the second cobordism however, the birth disk $\cD$ is merged with the cylinder emanating from the unknot $U_q$; therefore, there is a knot which appears in the boundary, which we identify with $\lambda$.
\end{proof}

In particular, up to the identifications given by taking connected sums, we may identify $\Lambda_p$ with $\Lambda_\bullet$, which in turn can be identified with $\Lambda \oplus \Lambda$.

\subsection{The Projection Correspondence}

We now turn to the proofs of \Cref{prop:doubling-CFA-to-CFK} and \Cref{prop: bijective on projection}. 

The action of $\Lambda_p$ on projection maps is straightforward: if $\pi \in \End^\cA(\CFDh(\KC))$ is a projection map, it follows from the \Cref{lem: X_n commutes} and \Cref{lem:cohen} that $\Lambda_p(\pi) \sim \bI_\bX \boxtimes (F_{X_n}\circ \pi)$. Moreover, we now prove that the summand of $\CFK_\cR(S^3, K, p)$ determined (up to homotopy) by this (homotopy) projection, is homotopy equivalent to the summand determined by $\bI_\bX \boxtimes \pi$. 

\begin{figure}
\def\svgwidth{.8\linewidth}
\begingroup%
  \makeatletter%
  \providecommand\color[2][]{%
    \errmessage{(Inkscape) Color is used for the text in Inkscape, but the package 'color.sty' is not loaded}%
    \renewcommand\color[2][]{}%
  }%
  \providecommand\transparent[1]{%
    \errmessage{(Inkscape) Transparency is used (non-zero) for the text in Inkscape, but the package 'transparent.sty' is not loaded}%
    \renewcommand\transparent[1]{}%
  }%
  \providecommand\rotatebox[2]{#2}%
  \newcommand*\fsize{\dimexpr\f@size pt\relax}%
  \newcommand*\lineheight[1]{\fontsize{\fsize}{#1\fsize}\selectfont}%
  \ifx\svgwidth\undefined%
    \setlength{\unitlength}{635.61820924bp}%
    \ifx\svgscale\undefined%
      \relax%
    \else%
      \setlength{\unitlength}{\unitlength * \real{\svgscale}}%
    \fi%
  \else%
    \setlength{\unitlength}{\svgwidth}%
  \fi%
  \global\let\svgwidth\undefined%
  \global\let\svgscale\undefined%
  \makeatother%
  \begin{picture}(1,0.40359988)%
    \lineheight{1}%
    \setlength\tabcolsep{0pt}%
    \put(0,0){\includegraphics[width=\unitlength,page=1]{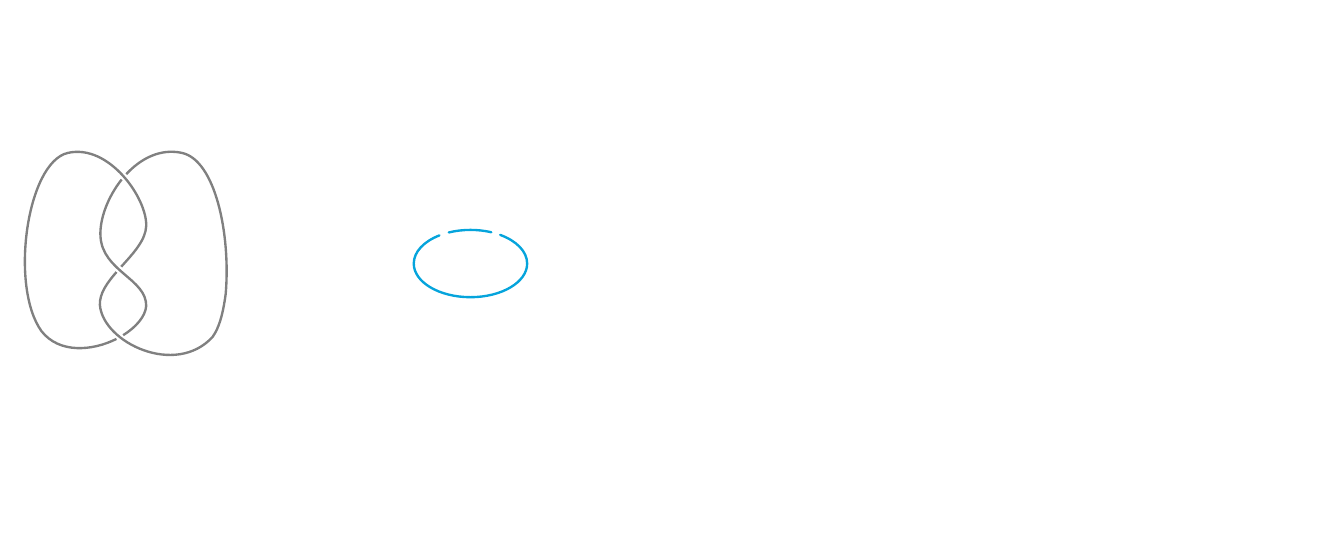}}%
    \put(0.34865005,0.15274776){\color[rgb]{0,0.63529412,0.85882353}\makebox(0,0)[lt]{\smash{\begin{tabular}[t]{l}{\small$0$}\end{tabular}}}}%
    \put(0,0){\includegraphics[width=\unitlength,page=2]{lens_cob.pdf}}%
    \put(0.46931148,0.25771244){\color[rgb]{0.83137255,0,0}\makebox(0,0)[lt]{\smash{\begin{tabular}[t]{l}{\small$\langle n \rangle $}\end{tabular}}}}%
    \put(-0.0012226,0.27618506){\color[rgb]{0.50196078,0.50196078,0.50196078}\transparent{0.917647}\makebox(0,0)[lt]{\smash{\begin{tabular}[t]{l}{\small$K$}\end{tabular}}}}%
    \put(0,0){\includegraphics[width=\unitlength,page=3]{lens_cob.pdf}}%
  \end{picture}%
\endgroup%

\caption{Left: The factorization of the cobordism appearing in Lemma \ref{lem: change framing cob}. Right: A model computation.}\label{fig:lens_cob}
\end{figure}

\begin{lem}\label{lem: change framing cob}
    Let $M$ be a summand of $\CFDh(-(\KC)_0)$. Then $\CFAx \boxtimes M$ is homotopy equivalent to $\CFAx \boxtimes F_{X_n}(M)$. 
\end{lem}

\begin{proof}
    Consider the handle decomposition for $(\bm{\curlywedge}, C)$ again in \Cref{fig:P_handle_decomp}. The 2-handle attached along the curve labeled $b$ can be attached at any point during the cobordism; so far, we have viewed this 2-handle as being attached immediately. But, we could instead attach it last (i.e., we can commute it past the pair of pants cobordisms). We should be slightly careful with $\Spinc$-structures. When we commute $W_n$ through $(P,C)$, we decompose $\bm{\curlywedge}$ as $W_n' \circ P'$, where $W_n': L(n,1) \# (S^3, K) \ra S^3, K)$ is given by attaching the 2-handle and $P'$ is the pair of pants cobordism associated to the triple $(T_\infty, \nu, p), (\KC)_0, (\KC)_0$. We need to determine how $\fraks_0$ restricts to $W_n'$ and $P'$. Utilizing \Cref{fig:P_handle_decomp} once again, we see that under the identifications $\Spinc(W_n) \cong \Z\langle a \rangle$ and $\Spinc(P) \cong \Z\langle c \rangle$, the restriction maps are given by the projections $(a, c)\mapsto a$ and $(a, c)\mapsto c$. 
    
    Therefore, by the invariance of the Heegaard Floer cobordism maps, there is a homotopy commutative diagram 
    \begin{center}
        \begin{tikzcd}
            \Mor^\cA((\KC)_0, (\KC)_0)\otimes (\CFAx\boxtimes\CFDh((\KC)_0)) \ar[d,"f \otimes (a \boxtimes v) \mapsto (\bI\boxtimes (F_{X_n} \circ f))(a \boxtimes v)"]\ar[r,"ev"] &
            (\CFAx\boxtimes\CFDh((\KC)_0) )\ar[d," \bI \boxtimes F_{X_n}"]\\
            \Mor^\cA((\KC)_0, (\KC)_n)\otimes (\CFAx\boxtimes\CFDh((\KC)_0))\ar[r,"ev"] &
            (\CFAx\boxtimes\CFDh((\KC)_n)).
        \end{tikzcd}
    \end{center}
    When $f = \pi$ is the projection to $M$, then the counter clockwise composition maps to $\CFAx\boxtimes F_{X_n}(M)$; the other composition maps to $(\bI \boxtimes F_{X_n})(\CFAx \boxtimes M)$. We claim that the map $\bI \boxtimes F_{X_n}$ is a homotopy equivalence. \\

    First, note that $\CFAx \boxtimes \CFDh((\KC)_0)\simeq \CFK_\cR(S^3, K,p) \simeq \CFAx  \boxtimes \CFDh((\KC)_n)$. As noted above, the map $\bI \boxtimes F_{X_n}$ is induced by the cobordism $S^3_n(U) \# (S^3, K,p) \ra (S^3, K,p)$ obtained by attaching a 0-framed 2-handle along the connected sum of the meridians of $U$ and $K$, and evaluating on the unique generator of $\widehat{HF}(L(n,1),[0])$.
    
    When $K$ is the unknot, this map can be computed by a model computation. The map is induced by the pair of pants cobordism, which by \cite{zemke_duality_mapping_tori}, is given by counting triangles in a genus one Heegaard diagram; see Figure \ref{fig:lens_cob}. The total cobordism map $\CFK_\cR(S^3, U) \ra \CFK_\cR(S^3, U)$ takes $1 \mapsto 1 + U^n + V^n$. One such triangle is shown in \ref{fig:lens_cob}. This is not a homogeneous element; the only term which is contributed by the component of the map in the $\Spinc$-structure $[0]$ on $W_n'$ is the first. Hence, in this $\Spinc$-structure, the map is indeed a homotopy equivalence. 
    
    The general case now follows from the fact that the connected sum cobordism $(S^3,K) \amalg (S^3, U) \ra (S^3,K\#U)$ induces a homotopy equivalence and commutes with the cobordism described above. Hence, the given map is actually homotopic to identity. The lemma follows.
\end{proof}

\Cref{prop:doubling-CFA-to-CFK} is now a simple consequence. 
\begin{proof}[Proof of \Cref{prop:doubling-CFA-to-CFK}]
    Let $\pi$ be the projection to the summand $M$ of $\CFDh(\KC)$. By \Cref{lem:free basepoint doubles everything} and \Cref{lem: lambda bullet is lambda p}, $$\Lambda_p(\pi) \sim \Lambda(\pi) \oplus \Lambda(\pi)$$ for $\Lambda(\pi) \in \End_\cR(\CFK_\cR(S^3, K))$. By \Cref{lem:composition}, we have that $\Lambda(\pi)$ is a homotopic to a projection, as:
    \begin{align*}
        \Lambda(\pi) = \Lambda(\pi^2) \sim \Lambda(p)^2.
    \end{align*}
    Hence, by \Cref{lem: homotopy projection to projection}, we can take $\Lambda(\pi)$ to be an honest projection. By Lemmas \ref{lem: X_n commutes} and \ref{lem:cohen}, $\Lambda_p$ is homotopic to the composition $\bI_\bX \boxtimes (F_{X_n}\circ \pi)$. By \Cref{lem: change framing cob}, the image of this projection map is homotopy equivalent to $\CFAx \boxtimes M$. This establishes the first claim of \Cref{prop:doubling-CFA-to-CFK}. 

    The second claim, that $\iota_{\KC}$-equivariant splittings are taken to $\iota_K$ invariant splittings, follows from \Cref{cor: invariant splittings CFA to CFK}.
\end{proof}

We now turn to the proof of \Cref{prop: bijective on projection}. This follows from our simplified description of $(\bm{\curlywedge}, C)$ from above. 

\begin{proof}[Proof of \Cref{prop: bijective on projection}]
    Let $[p] \in H_*\End_\cA(\CFDh(\KC)) \cong \HFh(S^3_0(K\#\overline{K}))$. Under the morphism pairing theorem, $p$ corresponds to an element in $\HFh_{\frac{1}{2}}(S^3_0(K\#\overline{K}),[0])$. We claim that the map $\Lambda$ takes $\HFh_{\frac{1}{2}}(S^3_0(K\#\overline{K}),[0])$ isomorphically to the summand of $\HFK_\cR(K\#\overline{K})$ in grading $(0,0)$. 
    
    We need to analyze two cobordism maps: first the cobordism 
    \begin{align*}
        W_n: L(n,1) \# S^3_0(K \# \overline{K})  \ra S^3_n(K \# \overline{K}).
    \end{align*}
    and second the large surgery cobordism 
    \begin{align*}
        \Gamma_{n}: S^3_n(K \# \overline{K}) \ra (S^3, K \# \overline{K}).
    \end{align*}
    Of course, $\Gamma_n$ induces an isomorphism, so, it suffices to show that $W_n$ induces an isomorphism $\HFh_{1/2}(S^3_0(K \# \overline{K},[0]) \ra \HFh_{d(L(n,1),[0])}(S^3_n(K \# \overline{K},[0])$. 
    
    To prove this carefully, we recall from \cite{os_HFK_integer_surgeries} that there is an exact triangle
    \begin{center}
        \begin{tikzcd}
            \HFh(S^3_0(K\#\overline{K})) \ar[rr,"\phi"]& & \HFh(S^3_n(K\#\overline{K})) \ar[dl] \\
            & \bigoplus_n \HFh(S^3). \ar[lu]
        \end{tikzcd}
    \end{center}
    The horizontal arrow is induced by the cobordism $W_n: L(n, 1)\# S^3_0(K) \ra S^3_n(K)$. We take $\Theta$ to be the generator of $\HFh(L(n,1), \ell)$; the horizontal map in the triangle is $F_{W_n}(\Theta, -)$.  Recall that the natural restriction map $\Spinc(W_n) \ra \Spinc(L(n,1) \# S^3_0(K \# \overline{K}))$ is an isomorphism; in particular, there is a single $\Spinc$-structure, $\frak{u}_0$, on $W_n$ which extends $\ell \oplus [0] \in \Spinc(L(n,1) \# S^3_0(K \# \overline{K})$. Furthermore, by \cite[Proposition 4.6]{os_HFK_integer_surgeries}, provided $n$ is sufficiently large, the map 
    \begin{align*}
        \HFh(S^3_0(K\# \overline{K})) \xra{F_{W_n}(\Theta,-)} \HFh(S^3_n(K\# \overline{K})) \ra \HFh(S^3_n(K\# \overline{K}),[0]),
    \end{align*}
    preserves the relative $\Z$-gradings of these groups. In particular, $F_{W_n, \frak{u}_0}(\Theta,-)$ takes $\HFh_{1/2}(S^3_0(K\# \overline{K}),[0])$ to $\HFh_{d(L(n,1),[0])}(S^3_n(K\# \overline{K}),[0])$.

    To see this map is an isomorphism, we make the following observations: 
    \begin{enumerate}
        \item Since $K\#\overline{K}$ is slice (ribbon, in fact), there are maps $f_{\bullet}: S^3_\bullet(U)\ra S^3_\bullet(K\#\overline{K})$ for $\bullet \in \{0, n, \infty\}$, where $f_{\bullet} = F_{X_\bullet(C)}$ is the cobordism map induced by the manifold obtained by performing surgery on the canonical concordance $C: U \ra K \# \overline{K}$. We also consider the maps $g_{\bullet}:S^3_\bullet(K\#\overline{K}) \ra S^3_\bullet(U) $ for $\bullet \in \{0, n, \infty\}$, where $g_\bullet$ is induced by the cobordism $\overline{X_\bullet(C)}$ obtained by turning $X_\bullet(C)$ around. Since the $C$ is ribbon, it follows from \cite[Theorem 4.10]{DLVW_ribbon} that $g_\bullet \circ f_\bullet$ is homotopic to the identity.  These maps fit into a commutative diagram:
    \begin{center}
        \begin{tikzcd}
            \ker g_0 \ar[rr,"\simeq"]\ar[dd]& & \ker g_n \ar[dl]\ar[dd] \\
            & \ker g_\infty =0  \ar[lu] & \\
            \HFh(S^3_0(K\#\overline{K})) \ar[dd,"g_0" left, bend right] \ar[rr,"\phi", near start]& & \HFh(S^3_n(K\#\overline{K})) \ar[dl]\ar[dd,"g_n"left, bend right] \\
            & \bigoplus_n \HFh(S^3) \ar[lu]\ar[from  = uu, crossing over] & \\
            \HFh(S^3_0(U))\ar[rr] \ar[uu,"f_0" right, bend right] & &\HFh(S^3_n(U)) \ar[dl] \ar[uu,"f_n" right, bend right] \\
            & \bigoplus_n \HFh(S^3)\ar[lu]\ar[from  = uu,"g_\infty" near start, bend right, crossing over] \ar[uu,"f_\infty" near start, bend right, crossing over] & \\
        \end{tikzcd}
    \end{center}
    Commutativity of the various squares is again topological: the 2-handle attachment cobordisms which induce the maps in the exact triangles commute with the cobordisms obtained by performing surgery on $C$. 
    \item The maps in the bottom triangle can be computed explicitly. Following \cite{os_HFK_integer_surgeries}, we write the generators of $\oplus_n \HFh(S^3)$ as $T^i$, $i = 0, \hdots, (n-1)$. The generator $\theta^+$ of $\HFh(S^3_0(U))$ in grading $+1/2$ is taken to the generator $x_0$ of $\HFh(S^3_n(U), [0])$; the generators $x_s$ of $\HFh(S^3_n(U), [s])$ for $s \ge 1$ are taken to the generators $T^s$ of $\oplus_n \HFh(S^3)$, and $1 \in \oplus_n \HFh(S^3)$ is taken to the generator $\theta^-$ of $\HFh(S^3_0(U))$ in grading $-1/2$.
    \item Since $g_\bullet \circ f_\bullet \sim \id$, we have splittings
    \begin{align*}
        \HFh(S^3_\bullet(K \# \overline{K}),[s]) \cong \ker g_{\bullet, [s]} \oplus \HFh(S^3_\bullet(U),[s]).
    \end{align*}
    We note that since $X_\bullet(C)$ and $\overline{X_\bullet(C)}$ are homology cobordisms, elements of $\Spinc(X_\bullet(C))$ are determined by their restriction to either boundary component. Moreover, the commutative diagram above implies that the maps in the exact triangle respect this splitting.
    \end{enumerate}    
    It follows from (2) and (3) that $\phi$ takes $\ker g_0$ isomorphically to $\ker g_n$. Moreover, the image of $\phi$ is identified with $\ker g_n \oplus \HFh(S^3_n(U),[0])$ and that the kernel of $\phi$ is generated by $\theta^- \in \HFh_{-1/2}(S^3_0(U)) \subset \HFh(S^3_0(K\#\overline{K}))$. In particular, as we have assumed that $n$ is very large, $\HFh_{1/2}(S^3_0(K\#\overline{K}),[0])$ is taken isomorphically onto $\HFh_{d(L(n,1),[0])}(S^3_n(K\#\overline{K}),[0])$. 
    
    It then follows that the composition 
    \[
    \HFh_{\frac{1}{2}}(S^3_0(K\#\overline{K}),\fraks_0) \xrightarrow{F_{W_n,\fraks_0}} \HFh_{d(L(n,1),[0])}(S^3_n(K\#\overline{K}),[0]) \xrightarrow{\Gamma_{n,[0]}} H_0(\widehat{\mathbb{A}}_{0}(K\#\overline{K}))
    \]
    is an isomorphism. Note that this map indeed does land in 
    the summand of $\HFK_\cR(K\#\overline{K})$ of Alexander grading zero as well as Maslov grading zero (since $\Gamma_{n, [0]}$, the the large surgery isomorphism restricted to the Spin$^c$ structure $[0]$, has degree $-d(L(n,1),[0])$ \cite[Corollary 4.2]{os_knotinvts}). 

    We now use the claim to prove that $\Lambda$ induces a bijection between homotopy classes of projections. Given a projection $p\in \mathrm{End}^\mathcal{A}(\widehat{CFA}(S^3 \smallsetminus K))$, we consider its image $\Lambda(p)\in \mathrm{End}(CFK_\mathcal{R}(S^3,K))$; by \Cref{lem:composition}, we have
    \[
    \Lambda(p)^2 \sim \Lambda(p^2) = \Lambda(p),
    \]
    i.e. $\Lambda(p)$ is a homotopy projection. Since $CFK_\mathcal{R}(S^3,K)$ is finitely generated, it follows from \Cref{lem: homotopy projection to projection} that $\Lambda(p)$ is homotopic to a projection. Conversely, if $p$ is an endomorphism of $\widehat{CFD}(S^3 \smallsetminus K)$ such that $\Lambda(p)$ is a projection, we have
    \[
    \Lambda(p+p^2) \sim \Lambda(p)+\Lambda(p^2) \sim 0,
    \]
    so that $p+p^2$ is contained in the kernel of $\Lambda$; then it follows from the claim above that $p+p^2$ is nullhomotopic, i.e. $p$ is a homotopy projection. Then it follows from \Cref{lem: homotopy projection to projection} that $p$ is homotopic to a projection.
\end{proof}

\begin{lem}
     $\Lambda$ is a homotopy ring homomorphism.
\end{lem}
\begin{proof}
    We have already established that $\Lambda(f\circ g) \sim \Lambda(f) \circ \Lambda(g)$; it thus suffices to show that $\Lambda(\id) \sim \id$. To prove this, note that we have
\begin{align*}
    \Lambda(\pi) = \Lambda(\pi \circ \id) \sim \Lambda(\pi)\circ \Lambda(\id).
\end{align*}
for any projection $\pi \in \End^\cA(\CFDh(\KC))$. By \Cref{prop: bijective on projection}, every projection contained in $\End_\cR(\CFK_\cR(S^3, K))$ is homotopic to $\Lambda(\pi)$ for some projection $\pi$ of $\End^\cA(\CFDh(\KC))$. It follows that composition by $\Lambda(\id)$ fixes homotopy classes of all projections of $\End_\cR(\CFK_\cR(S^3, K))$. Since the identity map is also a projection, we get $\Lambda(\id) = \Lambda(\id)\circ \id \sim \id$, as desired.
\end{proof}

\section{The Doubling Property}\label{sec: doubling}

In the previous section, we showed how projection maps on $\CFDh(\KC)$ correspond to projection maps on $\CFK_\cR(S^3,K)$, and that this correspondence is equivariant with respect to the $\mathrm{Spin}^c$ conjugation action. In particular, we showed that if $M$ is a summand of $\CFDh(\KC)$, there is a corresponding summand $\Lambda(M)$ of $\CFK_\cR(S^3, K)$, which satisfies
\begin{align*}
    \Lambda(M) \oplus \Lambda(M) \simeq \CFAx \boxtimes M,
\end{align*}
(see \Cref{prop:doubling-CFA-to-CFK}). We will therefore think of $\CFAx \boxtimes (-)$ as assigning to summands of $\CFDh$ ``doubled'' summands of $\CFK_\cR$. Our primary goal in this section is to relate this process to the \LOT correspondence. To do so, we will explore some of the properties of the diagrams $\bX$ and $\bY$ and derive some basic facts regarding the (bi-)modules $\CFAh_\cR(\bX)$ and $\CFDAy$. In particular, we show that the diagram $\bX$ can be used to pass back and forth between summands of $\CFK_\cR(S^3,K)$ and extendable summands of $\CFDh(\KC)$. 


\begin{prop}\label{prop: doubling complexes}
    Let $\mathfrak{C}$ be the set of homotopy equivalence classes of reduced $\cR$-complexes which appear as summands of $\CFK$-complexes and let $\mathfrak{D}$ denote the set of homotopy equivalence classes of extendable direct summands of $\widehat{CFD}(S^3\smallsetminus K)$, where $K$ runs over all knots. Define a map 
    \[
    \cM: \fr{C} \ra \fr{D},\;\; C \mapsto \cM_C
    \]
    which takes a complex $C$ to the extendable type-D structure $\cM_C$ given by applying the \LOT complex-to-type-D structure procedure (which is extendable, as discussed in \Cref{rem: LOT are extendable}). Let 
    \[
    \cC: \fr{D} \ra \fr{C}, \;\; M \mapsto \cC_M
    \]
    be the map which takes a type D structure $M$ to the $\cR$-complex $\cC_M:= \Lambda(M)$, which we recall is determined (up to homotopy equivalence) by the property that $\CFAh_\cR(\bX)\boxtimes M \simeq \cC_M \oplus \cC_M$. The maps $\cC$ and $\cM$ are inverses.
\end{prop}

\subsection{Properties of the Bimodule $\CFDAy$}

Recall from Section \ref{sec: background}, the bordered diagram $\bY$ with two boundary components. The diagram $\bY$ enjoys a number of useful properties. 

The first useful property is that tensoring with $\CFDAy$ always produces an extendable type D structure. This will be especially useful when trying to lift a splitting of $\CFAh(\KC)$ to a splitting of $\CFAt(\KC)$.

\begin{lem}\label{lem: Y tensor extends}
    Let $K$ be a knot and $M\subset \widehat{CFD}(S^3 \smallsetminus K)$ be any type D direct summand (which may not necessarily extendable). Then $\CFDAh(\bY) \boxtimes M$ is an extendable type D structure.
\end{lem}

\begin{proof}
    Let $N\oplus N'$ be a splitting of $\CFDh(\KC)$ and let $p_N\in \End_\cA(\CFDh(\KC))$ be the projection which which acts by the identity on $N$ and by zero on $N'$. By the pairing theorem, 
    \[
    \CFDAy \boxtimes \CFDh(\KC) \simeq \CFDh(\bY \cup (\KC)).
    \]
    It therefore suffices to find a summand of $\CFDt(\bY \cup (\KC))$ whose restriction to $\cA$ is homotopy equivalent to $\CFDAy \boxtimes N$. Our strategy will be to show that the projection map $\bI_{\CFDAy} \boxtimes p_N$ admits a lift to a (homotopy) projection $\CFDt(\bY \cup (\KC)) \ra \CFDt(\bY \cup (\KC))$. We will apply our usual strategy, and realize $\bI_{\CFDAy}\boxtimes p_N$ as a cobordism map. \\

    Let $P: (\KC) \amalg S^3_0(K\# \overline{K}) \ra \KC$ be the pair of pants cobordism associated to the triple $(\bY,\KC, \KC)$. Under the identification of $\HFh(S^3_0(K\#\overline{K}))$ with $\End^\cA(\CFDh(\KC))$, it follows from \cite[Theorem 1.1]{cohen2023composition} that the map $\bI_{\CFDAy} \boxtimes p_N$ is homotopic to the map given by precomposing $F_P$ with the coevaluation map 
    \[
    \mathrm{coev}:\CFDAh(\bY)\boxtimes\CFDh(\KC) \ra \CFDAh(\bY)\boxtimes\CFDh(\KC)\otimes \End^\cA(\KC),
    \]
    \[
    y \otimes x \mapsto y \otimes x \otimes p_N.
    \]
    We note that Cohen's argument applies to the bimodule case, though we could instead close our 3-manifolds up by gluing in $\CFDh(\KC)$ and using the techniques in the proof of \Cref{thm:bordered naturality one boundary}. 

    The cobordism $P$ is the composition of the connected sum cobordism 
    \[
    (\KC) \amalg S^3_0(K\#\overline{K}) \ra (\KC) \# S^3_0(K\#\overline{K})
    \]
    with a two-handle cobordism 
    \[
    (\KC) \# S^3_0(K\#\overline{K}) \ra (\KC).
    \]
    A handle decomposition for this manifold can be obtained from Frame (c) of \Cref{fig:P_handle_decomp} by removing a neighborhood of $\nu$ and taking $n = 0$. Since $\bY$ has a free, interior basepoint, there are well-defined connected sum maps
    \[
    \CFDh(\bY \cup (\KC)) \otimes \CFh(S^3_0(K\# \overline{K})) \xrightarrow{\simeq} \CFDh((\bY \cup \KC)\#(S^3_0(K\# \overline{K}))
    \]
    \[
    \CFDt(\bY \cup (\KC)) \otimes \CFh(S^3_0(K\# \overline{K})) \xrightarrow{\simeq} \CFDt((\bY \cup \KC)\#(S^3_0(K\# \overline{K}))
    \]
    since we may attach the feet of the connected sum tube in the region containing the free basepoint of $\bY \cup (\KC)$. By the pairing theorem for triangles, $\CFDAy\boxtimes N$ is characterized up to homotopy as the image of the 2-handle cobordism 
    \[
        \CFDh((\bY \cup (\KC))\#(S^3_0(K\# \overline{K})) \ra \CFDh(\bY \cup (\KC))
    \]
    precomposed with the coevaluation map. The 2-handle map is defined by counting holomorphic triangles in a bordered Heegaard diagram for $\bY \cup (\KC)$ and therefore lifts to a map 
    \[
        \CFDt((\bY \cup (\KC))\#(S^3_0(K\# \overline{K})) \ra \CFDt(\bY \cup (\KC)).
    \]
    Precomposing this map with (the obvious lift of) the coevaluation map produces the desired lift of $F_P \circ \mathrm{coev} \sim \bI_{\CFDAy}\boxtimes p_N$.\\

    We note that the extension of $\CFDAy \boxtimes N$ obtained in this way can be taken to be a direct summand of $\CFDt(\bY \cup (\KC))$; since $\bI_{\CFDAy}\boxtimes p_N$ is a projection, $(F_P \circ \mathrm{coev})^2 \simeq (F_P \circ \mathrm{coev})$. Therefore, for large $k$, $F := (F_P \circ \mathrm{coev})^k$ is an honest projection to some summand $S$. Let $\widetilde{F}$ be a lift of this projection with respect to $\pi: \CFDt(\bY \cup (\KC)) \ra \CFDh(\bY \cup (\KC))$. Since $\pi \circ \widetilde{F}^\ell = F^\ell\circ \pi = F\circ \pi$, the image of $\widetilde{F}^\ell$ is an extension of $S$ for all $\ell$. For sufficiently large $\ell$, the map $\widetilde{F}^\ell$ is a projection, and therefore its image is a summand. 
\end{proof}

The second useful property enjoyed by $\CFDAy$ is that it acts surjectively on the set of indecomposable summands of $\CFDh(\KC).$ We define indecomposable type D summands as follows.

\begin{defn}
    We say than an extendable type-D structure $M$ is \emph{indecomposable} if for any splitting $M \simeq N \oplus N'$, with $N$ and $N^\prime$ both extendable, either $N$ or $N'$ is acyclic. 
\end{defn}

It follows from \cite[Subsection 4.2]{hanselman2023bordered} that any extendable type D structure admits a unique decomposition into indecomposable summands, up to homotopy equivalence of each summand. It follows that any indecomposable summand of $\CFDt(\KC)$ will appear in any choice of a decomposition of $\CFDt(\KC)$ into indecomposable summands. This fact will be used crucially in the proof of \Cref{lem: Y surjects}.

\begin{lem}\label{lem: Y surjects}
    Let $M$ be an indecomposable direct summand of $\CFDh(\KC)$. Then, there is an extendable direct summand $N$ such that 
    \[
        \CFDAh(\bY) \boxtimes N \simeq M \oplus M.
    \]
\end{lem}
\begin{proof}
    Suppose that such a summand $N$ does not exist. Choose a decomposition
    \[
    \widehat{CFD}(S^3 \smallsetminus K)\simeq N_1\oplus\cdots\oplus N_k
    \]
    into indecomposables. By \Cref{prop:doubling-CFA-to-CFK}, we know that for each $i$, there exists a direct summand $C_i$ of $CFK_\mathcal{R}(S^3,K)$ such that 
    \[
    \CFAx \boxtimes N_i  \simeq C_i \oplus C_i.
    \]
    Since $\CFAx\simeq \CFA_\cR(T_\infty,\nu)\boxtimes\CFDAyh$, it follows that the half-extended type D structure $\widetilde{N}_i^{0,3} :=\CFDAyh \boxtimes N_i $ satisfies $\CFA_\mathcal{R}(T_\infty,\nu) \boxtimes \widetilde{N}_i\simeq C_i\oplus C_i$ and $\CFDAh(\bY)\boxtimes N_i$ is its truncation to $\mathcal{A}$ coefficients. Hence, by \Cref{lem: Y tensor extends} and \Cref{thm: LOT}, we see that
    \[
    \CFDAh(\bY)\boxtimes N_i \simeq \cM_{C_i} \oplus \cM_{C_i},
    \]
    where $\cM_{C_i}$ is a type D summand of $\widehat{CFD}(S^3 \smallsetminus K)$ which corresponds to $C_i$ under the \LOT correspondence. As noted in \Cref{rem: LOT are extendable}, the summand $\cM_{C_i}$ is extendable.

    We claim that $\cM_{C_i}$ is not acyclic; for contradiction, suppose that it were. If $\cM_{C_i}$ is acyclic, then both $C_i$ and $\CFAh_\mathcal{R}(\mathbb{X})\boxtimes N_i$ are acyclic as well. Following the discussions in \cite[Section 5]{kang_bordered_involutive_HFK}, we can define a type A morphism
    \[
    G:\widehat{CFA}(T_\infty,\nu)\rightarrow \widehat{CFA}(\mathbb{X})
    \]
    such that the map given by tensoring with the identity
    \[
    G \boxtimes \bI_{\widehat{CFD}(S^3 \smallsetminus K)}:\widehat{CFK}(S^3,K)\rightarrow \widehat{CFK}(S^3,K,p) \simeq \widehat{CFK}(S^3,K\amalg U)
    \]
    is the cobordism map induced by the split cobordism from $K$ to $K\amalg U$, where $U$ is the unknot. In particular, the map $G\boxtimes \mathbb{I}_{\widehat{CFD}(S^3 \smallsetminus K)} $ is injective on homology. Now consider the following homotopy-commutative diagram:
    \begin{center}
        \begin{tikzcd}
             \widehat{CFA}(T_\infty,\nu) \boxtimes N_i\ar[r,hook]\ar[d,"G \boxtimes \bI"] & 
                \widehat{CFK}(S^3,K) \ar[d,"G \boxtimes \bI"]\\
            \CFAx \boxtimes  \widehat{CFD}(T_\infty,\nu) \ar[r,hook] &
                \widehat{CFK}(S^3,K,p)
        \end{tikzcd}
    \end{center}

    Since the map on the right induces an injective map on homology and the lower left object is acyclic, it follows that $\widehat{CFA}(T_\infty,\nu) \boxtimes N_i$ is acyclic. However, $\widehat{CFA}(T_\infty,\nu) \boxtimes N_i$ is acyclic only if $N_i$ is acyclic (this can be seen easily in the immersed curve picture: it follows from \cite[Example 6.8]{hanselman2023bordered} that the dimension of the homology of $\widehat{CFA}(T_\infty,\nu) \boxtimes N_i$ is equal to the number of intersection points between the immersed curve invariant of $N_i$ (pulled tight) and the line of slope infinity; if the homology vanishes, the immersed curve invariant $N_i$ must be empty and therefore $N_i$ is acyclic). This contradicts our assumption that $N_i$ was indecomposable. Hence, $\cM_{C_i}$ cannot be acyclic.

    For each $i$, consider the number, $n_i$, of summands in $\cM_{C_i}$ when represented up to homotopy equivalence as a direct sum of indecomposable summands. Since $\cM_{C_i}$ is not acyclic, we know that $n_i \ge 1$. Since we have
    \[
    \CFDAy \boxtimes \widehat{CFD}(S^3 \smallsetminus K)\simeq \widehat{CFD}(S^3 \smallsetminus K,p)\simeq \widehat{CFD}(S^3 \smallsetminus K)\oplus \widehat{CFD}(S^3 \smallsetminus K),
    \]
    it follows that $\CFDAy \boxtimes \widehat{CFD}(S^3 \smallsetminus K)$ is homotopy equivalent to a direct sum of $2(n_1+\cdots+n_k)$ indecomposable summands of $\widehat{CFD}(S^3 \smallsetminus K)$. Thus it follows from the uniqueness of decompositions of extendable type A structures in terms of indecomposable direct summands that
    \[
    2(n_1+\cdots+n_k) = 2k,
    \]
    and thus we have that $n_1=\cdots=n_k=1$, i.e. each $\cM_{C_i}$ is indecomposable. Now, since we have assumed that no direct summand $N$ of $\widehat{CFD}(S^3 \smallsetminus K)$ satisfies
    \[
    \CFDAh(\bY)\boxtimes N\simeq M\oplus M
    \]
    and $M$ is indecomposable, it follows again from the uniqueness of decompositions into indecomposable summands that 
    \[
    \cM_{C_i} \oplus \cM_{C_i} \not\simeq M\oplus M
    \]
    for each $i$, and thus none of the indecomposable summands $\cM_{C_i}$ are homotopy equivalent to $M$. Hence $M$ cannot be a direct summand of $\widehat{CFD}(S^3 \smallsetminus K)$; a contradiction. The lemma follows.
\end{proof}

The final useful property of $\CFDAy$ is its behavior under doubling. By construction, $\bY$ represents $T^2 \times [0,1]$ with a basepoint on each boundary; the diagram $\bY\cup \bY$ therefore also represents $T^2 \times [0,1]$, though now with an additional basepoint, $q$, in the interior. In fact, the diagram $\bY \cup \bY$ is equivalent to one which is obtained as a quasi-stabilization of $\bY$. This is shown explicitly in Figure \ref{fig:yy_hmoves}.

\begin{figure}
\def\svgwidth{\linewidth}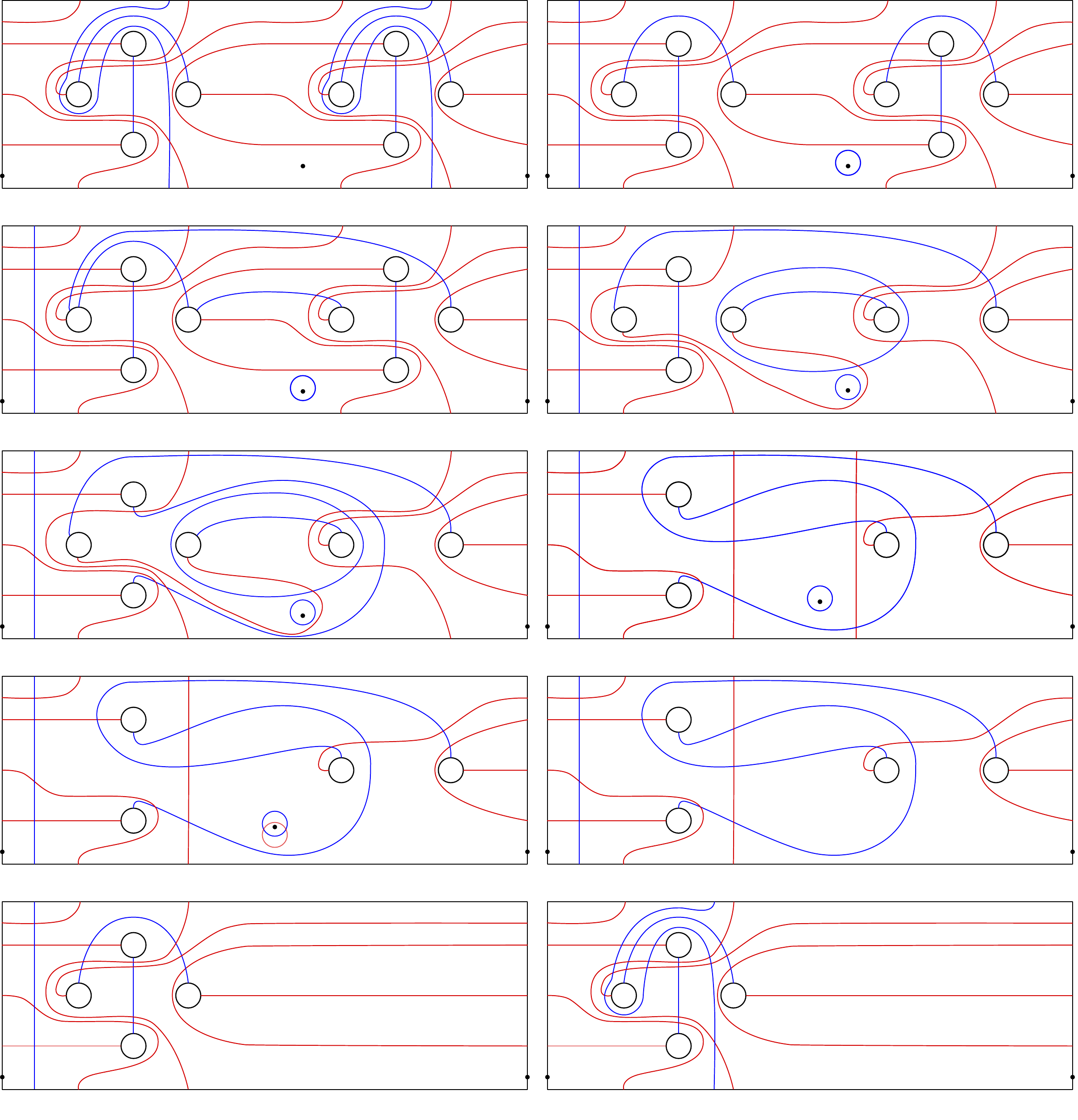
\caption{Heegaard moves relating $\bY \cup \bY$ to $\bY$: $\beta$-handle slides; $\beta$-handle slide; $\alpha$-handle slide; destabilization; $\beta$-handle slide; destabilization; $\beta$-handle slide; quasi-destabilization; diffeomorphism; $\beta$-handle slide.}\label{fig:yy_hmoves}
\end{figure}

Therefore, up to homotopy, the bimodule $\CFDAh(\bY \cup \bY)$ is obtained from $\CFDAh(\bY)$ by a quasi-stabilization (see \cite{zemke_linkcob} for more on quasi-stabilizations in the context of link cobordism maps and \cite{kang_bordered_involutive_HFK} for more in the context of bordered Floer homology). Consequently, 
\begin{equation}\label{eq: doubling Y}
    \CFDAh(\bY) \boxtimes \CFDAh(\bY) \simeq \CFDAh(\bY, q)
        \simeq \CFDAh(\bY) \oplus \CFDAh(\bY).
\end{equation}

We are now equipped to prove Proposition \ref{prop: doubling complexes}.

\begin{proof}[Proof of Proposition \ref{prop: doubling complexes}:]
    Let $C$ be an element of $\fr{C}$ and let $\cM_C$ be the associated element of $\fr{D}$. It suffices to show that
    \[
    \CFAh_\cR(\bX) \boxtimes \cM_C \simeq C \oplus C.
    \]
    Since $\bX = \bY \cup (T_\infty, \nu)$, it follows from the pairing theorem that 
    \[
    \CFAh_\cR(\bX) \boxtimes \cM_C \simeq \CFA_\cR(T_\infty, \nu)  \boxtimes \CFDAyh \boxtimes \cM_C.
    \] 
    Then $\CFDAyh\boxtimes \cM_C$ is a (half extended) type D structure over the algebra $\widetilde{\cA}_{0,3}$.  \\

    \textbf{Claim:} The half-extended type D structure $\CFDAyh\boxtimes \cM_C$ is an extension of the type D structure $\cM_C \oplus \cM_C$.

    \begin{proof}[Proof of claim:]
        Since every summand of $\widehat{CFD}(S^3 \smallsetminus K)$ is the direct sum of indecomposable summands, we may reduce the claim to the case that $\mathcal{M}_C$ is indecomposable. By the classification of extendable type D structures of \cite{hanselman2023bordered} and the classification of $\cR$-complexes of \cite{popovic2023link}, it is clear that indecomposable summands of $\CFK_\cR(S^3,K)$ correspond to indecomposable summands of $\CFDh(\KC)$. Therefore, if $\cM_C$ is indecomposable, it must be that $C$ itself was indecomposable. Then, by Lemma \ref{lem: Y surjects}, we may choose some type $D$ structure $N$ which has the property that $$\CFDAy \boxtimes N \simeq \cM_C \oplus \cM_C.$$ Consider the effect of tensoring with the bimodule $\CFDAy\boxtimes\CFDAy.$ On the one hand (again, by Lemma \ref{lem: Y surjects}), it follows that 
    \begin{align*}
        \CFDAy \boxtimes (\CFDAy \boxtimes N)& \simeq \CFDAy \boxtimes (\cM_{C} \oplus \cM_{C}) \\
       & \simeq (\CFDAy \boxtimes \cM_{C})^{\oplus 2}.
    \end{align*}
    On the other hand, by Equation \ref{eq: doubling Y},
    \begin{align*}
        (\CFDAy \boxtimes \CFDAy) \boxtimes N
        &\simeq (\CFDAy\boxtimes N)^{\oplus 2}\\
        &\simeq (\cM_C \oplus \cM_C)^{\oplus 2}
    \end{align*}
    Therefore, we have that 
    \[
    \CFDAy \boxtimes \cM_C \simeq \cM_C\oplus \cM_C,
    \]
    as desired.
    \end{proof}

    Observe that, by applying the arguments used in the proof of \Cref{lem: Y tensor extends}, we can deduce that this half-extended type D structure lifts to an extended type D structure, which is realized as a direct summand of $\CFDt(\KC,p)$. Hence, by the claim and Theorem \ref{thm: LOT}, it follows that 
    \begin{align*}
        \CFAh_\cR(\bX)\boxtimes \cM_C & \simeq \CFA_\cR(T_\infty,\nu) \boxtimes \CFDAyh \boxtimes \cM_C \\
        & \simeq C \oplus C.
    \end{align*}
    proving the first direction.

    In the reverse direction, let $M$ be an element of $\fr{D}$. By definition, $\cC_M$ is characterized by the property that 
    \[
    \CFAh_\cR(\bX)\boxtimes M \simeq \cC_M \oplus \cC_M.
    \]
    Once again, by the pairing theorem, we have that 
    \begin{align*}
        \CFAh_\cR(\bX)\boxtimes M \simeq \CFA_\cR(T_\infty,\nu)\boxtimes \CFDAyh\boxtimes M \simeq \cC_M \oplus \cC_M.
    \end{align*}
    from which it follows that $\CFDAyh\boxtimes M$ is an extension of $\cM_{\cC_M}\oplus \cM_{\cC_M}.$ But, by the same argument as above, $\CFDAyh\boxtimes M$ is an extension of $\CFDAy\boxtimes M \simeq M \oplus M$. Hence, $M \simeq \cM_{\cC_M}$.
    
\end{proof}

\subsection{Induced Splittings}

Theorems \ref{thm:CFK-to-CFA} and \ref{thm:CFA-to-CFK} are now simple consequences. 

\begin{proof}[Proof of \Cref{thm:CFA-to-CFK}:]
    Fix an $\iota_{\KC}$-invariant splitting of $\CFDh(\KC)$, i.e.
    \[
    \CFDh(\KC) = M_1 \oplus \hdots \oplus M_n.
    \]
    Denote the projection map to $M_i$ by $p_i$. Then $p_1,\cdots,p_n$ are pairwise commuting projections of $\CFDh(\KC)$. By \Cref{lem:composition} and \Cref{prop: bijective on projection}, we know that $\Lambda(p_1),\cdots,\Lambda(p_n)$ are pairwise homotopy commuting projections. Applying \Cref{lem: commuting projections} then makes them pairwise (strictly) commuting projections, which induces a splitting
    \[
    CFK_\mathcal{R}(S^3,K)\simeq N_1 \oplus \hdots \oplus N_n,
    \]
    so that the projection map to $N_i$ is given by $\Lambda(p_i)$. By \Cref{prop:doubling-CFA-to-CFK} and \Cref{prop: doubling complexes}, we see that $N_i$ corresponds to $M_i$ under the \LOT correspondence. Furthermore, by \Cref{cor: invariant splittings CFA to CFK}, it follows that $\Lambda(p_i)$ commutes with $\iota_K$ for all $i$. Therefore our splitting is $\iota_K$-invariant.
\end{proof}

The proof of \Cref{thm:CFK-to-CFA} follows almost directly from \Cref{thm:CFA-to-CFK}, but we need one more lemma.

\begin{lem} \label{lem:summand-extendable}
    Let $K$ be a knot. Then every type-D direct summand of $\widehat{CFD}(S^3 \setminus K)$ is extendable.
\end{lem}
\begin{proof}
    Given a direct summand $M$ of $\widehat{CFD}(S^3 \setminus K)$, denote its complement by $N$, i.e.
    \[
    \widehat{CFD}(S^3 \setminus K) \simeq M\oplus N.
    \]
    Denote the projection endomophism which is identity on $M$ and zero on $N$ by $p$. By \Cref{prop: bijective on projection}, we know that $\Lambda(p)$ is a projection, and thus there exists a splitting map (which is a homotopy equivalence)
    \[
    f:CFK_\cR(S^3,K)\xrightarrow{\simeq} C\oplus D
    \]
    (for some direct summands $C,D$ of $CFK_\cR(S^3,K)$) such that $f\circ \Lambda(p) \circ f^{-1}$ is homotopic to the endomorphism $p_{C,D}$ of $C\oplus D$ defined as identity on $C$ and zero on $D$. 

    Denote the (homotopy types of) type-D modules $\mathcal{M}_C$ and $\mathcal{M}_D$ which correspond to $C$ and $D$, respectively, under the \LOT correspondence. By \cite[Theorem 1.1]{popovic2023link}, we know that they are extendable. Denote the projection endomorphism $p^\prime$ (of $\widehat{CFD}(S^3 \setminus K)$) defined as identity on $\mathcal{M}_C$ and zero on $\mathcal{M}_D$. Applying \Cref{prop: bijective on projection} again to $p^\prime$ and then applying \Cref{thm:CFA-to-CFK}, we see that there exists another splitting map
    \[
    g:CFK_\cR(S^3,K)\xrightarrow{\simeq} C\oplus D
    \]
    such that $g^{-1}\circ p_{C,D}\circ g\sim \Lambda(p^\prime)$.

    Now consider the composition $h_0 = g^{-1}\circ f$. By \Cref{prop: bijective on projection}, we know that $\Lambda^{-1}(h_0)$ is well-defined. Then, by \Cref{lem:composition}, we have
    \[
    p \sim \Lambda^{-1}(h_0)^{-1} \circ p^\prime \circ \Lambda^{-1}(h_0).
    \]
    By taking a reduced model of $\widehat{CFD}(S^3 \setminus K)$, so that it has no nontrivial acyclic (type-D) direct summand, we see that the homotopy autoequivalence $\Lambda^{-1}(h_0)$ of $\widehat{CFD}(S^3 \setminus K)$ can be homotoped to a type-D automorphism. Then the homotopy inverse $\Lambda^{-1}(h_0)^{-1}$ of $\Lambda^{-1}(h_0)$ can be taken to be its strict (i.e. set-theoretic) inverse. This implies that the image of $p$, which is $M$ by definition, is homotopy equivalent to the image of $p^\prime$, which is $\mathcal{M}_C$ by definition. Hence we have a type-D homotopy equivalence
    \[
    M\simeq \mathcal{M}_C;
    \]
    since $\mathcal{M}_C$ is extendable, we conclude that $M$ is indeed extendable.
\end{proof}

\begin{proof}[Proof of \Cref{thm:CFK-to-CFA}:]
    Fix an $\iota_K$-invariant splitting of $CFK_\mathcal{R}(S^3,K)$, i.e.
    \[
    CFK_\mathcal{R}(S^3,K)\simeq N_1 \oplus \cdots \oplus N_n.
    \]
    Denote the projection map to $N_i$ by $p_i$. Then $p_1,\cdots,p_n$ are pairwise commuting projections of $CFK_\mathcal{R}(S^3,K)$. By \Cref{prop: bijective on projection}, there exist projections $q_1,\cdots,q_n$ of $\CFDh(\KC)$ such that $p_i = \Lambda(q_i)$ for all $i$. By \Cref{lem:composition}, we have
    \[
    \Lambda(q_i q_j) \sim \Lambda(q_i)\Lambda(q_j) = p_ip_j = 0
    \]
    for any $i\ne j$. Since $0$ is a projection, it follows from \Cref{prop: bijective on projection} that $q_iq_j \sim 0$ for all $i\ne j$. Hence $q_1,\cdots,q_n$ are pairwise homotopy commuting projections. Then applying \Cref{lem: commuting projections} allows us to assume that they are actually pairwise (strictly) commuting projections, which induces a splitting
    \[
    \CFDh(\KC) = M_1 \oplus \hdots \oplus M_n.
    \]
    By \Cref{lem:summand-extendable}, we know that each $M_i$ is extendable. Hence, by applying \Cref{prop:doubling-CFA-to-CFK} and \Cref{prop: doubling complexes}, we see that $N_i$ corresponds to $M_i$ under the \LOT correspondence. Furthermore, by \Cref{cor: invariant splittings CFK to CFA}, we know that $q_i$ is $\iota_{\KC}$-equivariant for all $i$. Therefore our splitting is $\iota_{\KC}$-invariant.
\end{proof}

To conclude this section, we prove \Cref{thm:relative CFK-to-CFA} and \Cref{thm:relative CFA-to-CFK}. Their proofs are almost identical and similar in spirit to the proofs of Theorems \ref{thm:CFK-to-CFA} and \ref{thm:CFA-to-CFK} (and also use them in a straightforward manner), so we will only present the latter.

\begin{proof}[Proof of \Cref{thm:relative CFA-to-CFK}]
    The assumption that $CFK_\mathcal{R}(S^3,K)$ and $CFK_\mathcal{R}(S^3,K^\prime)$ are locally equivalent (not necessarily involutively) implies that $CFK_\mathcal{R}(S^3,\overline{K}\# K^\prime)$ has a free summand of bidegree $(0,0)$. Hence, even though $\overline{K}\# K^\prime$ is not slice, it is locally equivalent to the unknot, so the same argumentation as in the proof of \Cref{prop: bijective on projection} may be applied to deduce that the composition-preserving chain map
    \[
    \Lambda_{K,K^\prime}:\mathrm{Mor}^A(\CFDh(S^3 \smallsetminus K),\CFDh(S^3 \smallsetminus K^\prime))\rightarrow \mathrm{Mor}(CFK_\mathcal{R}(S^3,K),CFK_\mathcal{R}(S^3,K^\prime))
    \]
    induces a bijection between homotopy classes of degree (or bidegree) preserving morphisms. Furthermore, it is also straightforward, by mimicking the proof of \Cref{lem:composition}, to see that
    \[
    \Lambda_{K,K^\prime}(g\circ f)\sim \Lambda_{K^\prime}(g)\circ \Lambda_{K,K^\prime}(f),\quad \Lambda_{K,K^\prime}(f\circ h)\sim \Lambda_{K,K^\prime}(f)\circ \Lambda_K(h)
    \]
    for degree-preserving morphisms
    \[
    f\in \mathrm{Mor}^\mathcal{A}(\CFDh(S^3 \smallsetminus K),\CFDh(S^3 \smallsetminus K^\prime)),\quad g\in \mathrm{End}_\mathcal{A}(\CFDh(S^3 \smallsetminus K^\prime)),\quad h\in \mathrm{End}_\mathcal{A}(\CFDh(S^3 \smallsetminus K)).
    \]

    Given a degree-preserving type-D morphism
    \[
    f^D:\CFDh(S^3 \smallsetminus K)\rightarrow \CFDh(S^3 \smallsetminus K^\prime),
    \]
    consider the corresponding bidegree-preserving chain map 
    \[
    \Lambda_{K,K^\prime}(f^D):CFK_\mathcal{R}(S^3,K)\rightarrow CFK_\mathcal{R}(S^3,K^\prime).
    \]
    A slight modification of the proof of \Cref{lem: lambda commutes with involutions} shows that $\Lambda_{K,K'}$ commutes with the involution. Hence, we have
    \[
    f^D \circ \iota_{S^3 \smallsetminus K} \sim \iota_{S^3 \smallsetminus K^\prime}\circ (\mathrm{id}\boxtimes f^D).
    \]
    Suppose that $f^D$ splits along the decompositions
    \[
    \CFDh(S^3 \smallsetminus K) \simeq M_1 \oplus \cdots \oplus M_m, \quad
        \CFDh(S^3 \smallsetminus K^\prime) \simeq N_1 \oplus \cdots \oplus N_n.
    \]
    The proof of \Cref{thm:CFA-to-CFK} then produces corresponding splittings
    \[
    CFK_\mathcal{R}(S^3,K) \simeq C_1 \oplus \cdots \oplus C_m, \quad CFK_\mathcal{R}(S^3,K^\prime) \simeq D_1 \oplus \cdots \oplus D_n,
    \]
    such that the following conditions are satisfied.
    \begin{itemize}
        \item Each $C_i,D_j$ corresponds to $M_i,N_j$ under the \LOT correspondence, respectively.
        \item If we denote the projections to $M_i$ and $N_j$ by $p_i$ and $q_j$, respectively, then the projections to $C_i$ and $D_j$ are given by $\Lambda_K(p_i)$ and $\Lambda_{K^\prime}(q_j)$, respectively. 
    \end{itemize}
    
    Now suppose that the $(i,j)$ entry of the block matrix form of $f^D$ is zero. This is equivalent to saying that $q_j \circ f^D \circ p_i \sim 0$, which then implies that
    \[
    \Lambda_{K^\prime}(q_j)\circ \Lambda_{K,K^\prime}(f^A)\circ \Lambda_K(p_i) \sim \Lambda_{K,K^\prime}(q_j \circ f^D \circ p_i) \sim \Lambda_{K,K^\prime}(0)=0.
    \]
    Thus the $(i,j)$ entry of the block matrix form of $\Lambda_{K,K^\prime}(f^D)$ is zero. Conversely, suppose that the $(i,j)$ entry of the block matrix form of $\Lambda_{K,K^\prime}(f^D)$ is zero. Then, by the same logic, we see that
    \[
    \Lambda_{K,K^\prime}(q_j \circ f^D \circ p_i)\sim 0
    \]
    However, since $\Lambda_{K,K^\prime}$ induces a bijection between homotopy classes of degree (or bidegree) preserving morphisms, it follows that $q_j \circ f^D \circ p_i \sim 0$, which means that the $(i,j)$ entry of the block matrix form of $f^D$ is zero. The theorem thus follows by taking $f=\Lambda_{K,K^\prime}(f^D)$.
\end{proof}

\section{Satellites and Involutive Splittings} \label{sec:maintheoremproof}
In this section, we prove Theorem \ref{thm: induced splittings}.

\newcommand{\biP}{\prescript{}{\cA}{P}^\cA}
\newcommand{\cHaa}{\prescript{\alpha}{}{\cH^\alpha}}
\newcommand{\cHbb}{\prescript{\beta}{}{\overline{\cH}^\beta}}

\subsection{Type DA Bimodule Involutions}

\begin{defn}
    An \emph{involutive} type DA bimodule over $(\cA, \cA) := (\cA(F),\cA(F))$ is a pair $(\biP, \iota_P)$ where $\biP$ is a type DA bimodule over $\cA$ and $\iota_P$ is a homotopy equivalence
    \[
    \iota_P: \CFDAh(\AZ)\boxtimes\biP \boxtimes \CFDAh(\overline{\AZ}) \ra \biP.
    \]
\end{defn}
Given bordered 3-manifold $Y$ with $\partial_L Y \cong F \cong \partial_R Y$ represented by an arced diagram $\cH$, it follows from \cite[Lemma 4.6]{LOT_HF_as_morphism} that
\[
    \prescript{\alpha}{}{\AZ^\beta}\cup \cHbb\cup \prescript{\beta}{}{\overline{\AZ}^\alpha} \sim \cH^\alpha,
\]
where $\sim$ means the two diagrams are related by Heegaard moves. These Heegaard moves define a homotopy equivalence 
    \[
    \CFDAh(\AZ)\boxtimes \CFDAh(Y) \boxtimes \CFDAh(\overline{\AZ}) \ra \CFDAh(Y).
    \]
By naturality of the bordered bimodules, Theorem \ref{thm: naturality for multimodules}, up to homotopy, this map does not depend on the choice of Heegaard moves. We define this map to be the bordered involution, $\iota_Y$, for $Y$.

Unsurprisingly, $\iota_{S^3 \smallsetminus P(K)}$ can be recovered from $\iota_{\KC}$ and $\CFDAh((S^1\times D^2)\smallsetminus P)$.

\begin{prop}\label{prop:bimodule_involution}
    Let $K$ be a knot in $S^3$ and let $P$ be a satellite pattern in the solid torus. Let $\iota_{S^3\smallsetminus K}$, $\iota_{S^3\smallsetminus P(K)}$, and $\iota_{P}$ be the bordered involutions for $\CFDh(\KC)$, $\CFDh(S^3 \smallsetminus P(K))$, and $\CFDAh((S^1\times D^2)\smallsetminus P)$ respectively. Then, the following diagram commutes up to homotopy
    \begin{center}
        \begin{tikzcd}[column sep = huge]
            \CFDAh(\AZ)\boxtimes\CFDh(S^3 \smallsetminus P(K)) \ar[rr,"\iota_{S^3\smallsetminus P(K)}"] \ar[d,"\simeq"] && \CFDh(S^3 \smallsetminus P(K))\ar[d,"\simeq"] \\
             \shortstack[c]{$\CFDAh(\AZ)\boxtimes \CFDAh(S^1\times D^2 \smallsetminus P) \boxtimes \CFDAh(\overline{\AZ})$\\
             $\boxtimes$ \\
             $ \CFDAh(\AZ) \boxtimes\CFDh(S^3 \smallsetminus K) $ }
             \ar[rr,"\iota_{\KC}\boxtimes \iota_{P}"] && 
             \shortstack[c]{$ \CFDAh(S^1\times D^2 \smallsetminus P)$\\
             $\boxtimes$ \\ $\CFDh(S^3 \smallsetminus K)$ }.
        \end{tikzcd}
    \end{center}
    The vertical arrows are given by the pairing theorem. 
\end{prop}

\begin{proof}
    Fix an bordered Heegaard diagram $\cH_K^\alpha$ for $\KC$ and an arced Heegaard diagram $\cHaa$ for $T_\infty\smallsetminus P$ (both alpha-bordered). By \cite{HL_Inv_bordered_floer} that map $\iota_{S^3 \smallsetminus P(K)}$ is induced by a sequence of Heegaard moves relating
    \[
    \prescript{\beta}{}{\AZ^\alpha}\cup \cHbb\cup\overline{\cH}_K^\beta \sim  \cHaa\cup \cH^\alpha_K. 
    \]
    By \cite[Lemma 4.6]{LOT_HF_as_morphism}, we have that 
    \[
    \overline{\cH}_K^\beta\cup \prescript{\beta}{}{\overline{\AZ}^\alpha} \sim \cH^\alpha_K \hspace{2cm}
    \prescript{\alpha}{}{\AZ^\beta}\cup \cHbb\cup \prescript{\beta}{}{\overline{\AZ}^\alpha} \sim \cH^\alpha.
    \] 
    We can therefore choose a sequence of Heegaard moves respecting our original decomposition:
    \begin{align*}
        \prescript{\beta}{}{\AZ^\alpha}\cup \cHbb\cup\overline{\cH}_K^\beta & \sim \prescript{\beta}{}{\AZ^\alpha}\cup \cHbb\cup\bI \cup \overline{\cH}_K^\beta \\
        & \sim \prescript{\beta}{}{\AZ^\alpha}\cup \cHbb\cup \prescript{\beta}{}{\overline{\AZ}^\alpha} \cup \prescript{\alpha}{}{\AZ^\beta}\cup \overline{\cH}_K^\beta\\
        & \sim \cHaa\cup \cH^\alpha_K
    \end{align*}
    These sequence of moves induce maps  
    \[
    \CFDAh(\AZ)\boxtimes\CFDh(S^3 \smallsetminus K)\ra \CFDh(S^3 \smallsetminus K)
    \]
    and 
    \[
    \CFDAh(\AZ)\boxtimes \CFDAh(S^1\times D^2 \smallsetminus P) \boxtimes \CFDAh(\overline{\AZ}) \ra \CFDAh(S^1\times D^2 \smallsetminus P).
    \]
    The first is the map $\iota_{\KC}$ from \cite{HL_Inv_bordered_floer} and the other is $\iota_{P}$. By the same argument as in \cite[Theorem 5.1]{HL_Inv_bordered_floer}, these maps tensor together to compute $\iota_{S^3 \smallsetminus P(K)}$. 
\end{proof}

\subsection{Proof of Main Theorem} Theorem \ref{thm: induced splittings} follows quickly.

\begin{proof}[Proof of Theorem \ref{thm: induced splittings}]
    Let $K$ be a knot in $S^3$ and $P$ a pattern knot in the solid torus. Fix an $\iota_K$-equivariant splitting of $\CFK_\cR(S^3,K)$, 
    \[
    \CFK_\cR(S^3,K) \simeq C_1 \oplus \hdots \oplus C_n.
    \]
    By Theorem \ref{thm:CFK-to-CFA}, the \LOT correspondence produces an $\iota_{\KC}$-equivariant splitting of $\CFDh(\KC)$
    \[
    \CFDh(\KC) \simeq \cM_{C_1} \oplus \hdots \oplus  \cM_{C_n}.
    \]
    Denote by $P(M)$ the type $D$ structure $\CFDA(S^1\times D^2 \smallsetminus P) \boxtimes M$. It follows from Proposition \ref{prop:bimodule_involution} that the decomposition 
    \[
    \CFDh(S^3 \smallsetminus P(K)) \simeq P(\cM_{C_1}) \oplus \hdots \oplus  P(\cM_{C_n}).
    \]
    is $\iota_{S^3 \smallsetminus P(K)}$-equivariant up to homotopy. By \Cref{lem:summand-extendable}, each of the summands $P(\cM_{C_i})$ admit extensions $\widetilde{P(\cM_{C_i})}$. By Theorem \ref{thm:CFA-to-CFK}, the induced splitting 
    \[
    \CFK_\cR(S^3,P(K)) \simeq P(C_1) \oplus \hdots \oplus P(C_n),
    \]
    is $\iota_{P(K)}$-equivariant, where $P(C_i) = \CFAt(T_\infty, \nu)\boxtimes \widetilde{P(\cM_{C_i})}$.
\end{proof}

\section{Relative splitting principles and a topological application} \label{sec:topapp}

In this section, we prove \Cref{thm: one stab inf family}. We start by very briefly summarizing the arguments of \cite{kang2022stabilization}.

The exotic 4-manifolds considered there are constructed from a pair of slice disks $D$ and $D'$ for a certain knot $K$. By doing $+1$-surgery on these disks, one obtains two contractible 4-manifolds, $B_{+1}(D)$ and $B_{+1}(D')$. The manifolds $B_{+1}(D)\#S^2 \times S^2$ and $B_{+1}(D')\#S^2 \times S^2$ can be compared by computing their induced involutive cobordism map (recall that in the involutive context, stabilizing by $S^2 \times S^2$ amounts to multiplying by $Q$.) 

In fact, it suffices to understand the maps induced by $D$ and $D'$ on the usual knot Floer complex and the action of $\iota_K$. By \cite[Proposition 4.4]{kang2022stabilization}, if 
\[
(F_D + F_{D'})(1) \not\in \mathrm{im}(\id + \iota_K),
\]
then $B_{+1}(D)\#S^2 \times S^2$ and $B_{+1}(D')\#S^2 \times S^2$ will be distinguished by their maps induced on involutive Heegaard Floer homology. 

A particularly useful class of slice disks are \emph{deform spun disks}, $D_{K,f}$, associated to a pair $(K, f)$, where $K$ is a knot in $S^3$ and $f$ is a diffeomorphism of $(S^3, K)$ fixing a small neighborhood of a point of $K$. The map induced by $D_{K, f}$ is determined by the action of $f$ on $\widehat{\HFK}(S^3, K)$. In particular, there is a diffeomorphism considered in \cite{juhasz_zemke_stabilization_bounds} related to swapping the summands of the knot $K \# K$, which can be computed explicitly, as 
\[
\mathrm{Sw}(\id \otimes \id + \id \otimes(\Phi \Psi) + \Psi \otimes \Phi).
\]
The associated deform spun disks are natural candidates for constructing stably exotic behavior due to their computable nature. The following definition characterizes the algebraic structure the knot $K$ should exhibit to use this strategy to construct stably exotic 4-manifolds.

\begin{defn}\cite[Definition 4.7]{kang2022stabilization}\footnote{We note that there is a typo in the current arXiv version of this article; Definition 4.7 currently says $K$ is \emph{involutively weird}, but it should say \emph{involutively nonsimple}, as in this article. \emph{Involutively weird} is defined in Definition 4.5 in that article.}
    We say that a knot $K$ is \emph{involutively nonsimple} if $\widehat{HFK}(S^3,K)$ admits a splitting
    \[
    \widehat{HFK}(S^3,K)\simeq V_1 \oplus V_2 \oplus W
    \]
    which is invariant under the actions of $\hat{\iota}_K$, $\hat{\Phi},$and $\hat{\Psi}$ and $\dim\,V_1=\dim\,V_2=1$.
    where the following conditions are satisfied.
\end{defn}
We summarize proof of \cite[Theorem 1.1]{kang2022stabilization} in the following proposition.

\begin{prop}[Proof of {\cite[Theorem 1.1]{kang2022stabilization}}] \label{prop: nonsimple implies stable}
    Let $K$ be an involutively nonsimple knot. Then $S^3 _{+1}(4K\# 4\overline{K})$ bounds an exotic pair of smooth contractible 4-manifolds which remains exotic after one stabilization. 
\end{prop}

We now deduce \Cref{thm: one stab inf family} from \Cref{prop: nonsimple implies stable}.

\begin{figure}
\def\svgwidth{.7\linewidth}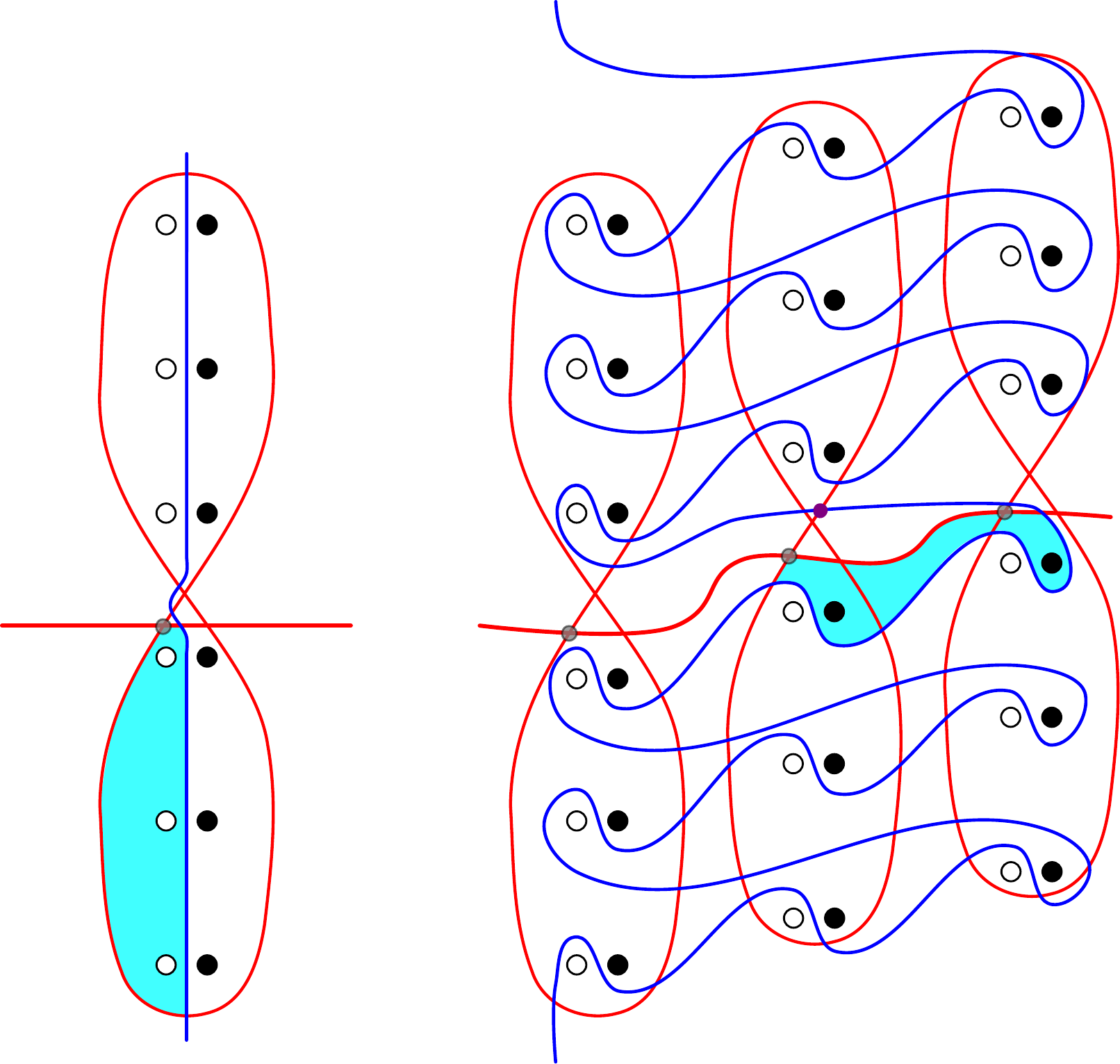
\caption{The cabling procedure via Hanselman-Chen \cite{chen_hanselman2023satellite} for the square complex $S_n$. Here, we have drawn the case $n = m = 3.$ The grey intersection points are turning points encoding the grading arrows of \cite{hanselman_rasmussen_watson_properties_apps}, which determine the relative Maslov grading. For example, in the left figure, the intersection point between the vertical and horizontal lines has grading $(0,0)$; there is a bigon to $x$ from the lowest intersection point, $c$ which crosses the turning point and $n$ of the $z$ basepoints. Hence, $\gr(c) = (-1,2n-1)$. Likewise, after applying the cabling operation, there are bigons connecting $x$ to generators in each of the resulting components which specify the grading of those components relative to the intersection point lying on the horizontal component of the immersed curve. The generator $\zeta_{\frac{m+1}{2}}$ is shown in purple.}\label{fig:cables_grading}
\end{figure}

\begin{proof}[Proof of \Cref{thm: one stab inf family}]
    Given any odd integers $n,m\ge 3$, the knot $K_{n,m}$ is given by $K_{n,m}=(K^\ast_n)_{m,-1}$, where $K^\ast_n = -T_{2n,4n+1}\# 2T_{2n,2n+1}$. By \cite[Proposition 3.5]{hendricks2022quotient} there is an $\iota_K$-invariant splitting
    \[
    \CFK_\mathcal{R}(S^3,K^\ast_n)\simeq C_n \oplus D_n
    \]
    where the complex $C_n$ is the complex
    \[
    \xymatrix{
     && b \ar[dd]_{V^n} && a \ar[ll]_{U^n} \ar[dd]^{V^n} \\
    x & \oplus & &&\\
     && d && c \ar[ll]^{U^n}
    }
    \]
    with gradings $\gr(x) = (0,0)$, $\gr(a) =(0,0)$, $\gr(b) = (2n-1,-1)$, $\gr(c) = (-1,2n-1)$,  $\gr(d)=(2n-2,2n-2).$ The involution, $\iota_K$, is given by $a\mapsto a$, $x\mapsto x+d$, $b\leftrightarrow c$, $d\mapsto d$. Note that, over the full coefficient ring $\F[U,V]$, $\iota_K(a)=a+U^{n-1}V^{n-1}x$. Decompose $C_n$ further into two parts, i.e.
    \[
    C_n \simeq O \oplus S_n,
    \]
    where $O$ is generated by $x$ and $S_n$ is generated by $a,b,c,d$.

    By \Cref{thm: induced splittings}, there is an $\iota_{K}$-invariant splitting 
    \[
    \CFK_\mathcal{R}(S^3,K_{n,m})\simeq (C_n)_{m,-1}\oplus (D_n)_{m,-1}.
    \]
    Since $O_{m,-1}\simeq O$, this complex decomposes further as
    \[
    \CFK_\mathcal{R}(S^3,K_{n,m})\simeq O\oplus (S_n)_{m,-1}\oplus (D_n)_{m,-1},
    \]
    where the action of $\iota_K$ splits as the direct sum of its action on $O\oplus (S_n)_{m,-1}$ and its action on $(D_n)_{m,-1}$. Denote the induced splitting given by setting $U=V=0$ and then taking homology as 
    \[
    \widehat{HFK}(S^3,K_{n,m})\simeq \F\langle x\rangle \oplus S^H_{n,m}\oplus D^H_{n,m}.
    \]
    
    The chain complex $(S_n)_{m,-1}$ can be computed by applying the Hanselman-Watson or Chen-Hanselman cabling formula \cite{hanselman_watson_cabling,chen_hanselman2023satellite}; see \Cref{fig:cables_grading}. Note that $\widehat{\HFK}$ is generated by intersection points between immersed curves and the vertical line which passes through the marked point on the torus, and the Alexander grading of a generator is determined by its height in the universal cover of the marked torus. $S^H_{n,m}$ has $m$ generators of bidegree $(2-2n,2-2n)$
    , denoted $\zeta_1,\cdots,\zeta_m$. These lift to $(S_n)_{m,-1}$ such that the differential takes the following form
    \[
    \partial \zeta_i = U^{a_i}y_i + V^{b_i}z_i,
    \]
    for some generators $y_i,z_i$ of $S^H_{n,m}$. Here, the $\{a_i\}$ are strictly increasing positive integers and the $\{b_i\}$ are strictly decreasing sequences of positive integers, with
    \[
    a_{\frac{m+1}{2}} = b_{\frac{m+1}{2}} = \frac{m+1}{2}.
    \]
    By \Cref{thm: induced splittings}, $\hat{\iota}_K(\zeta_{\frac{m+1}{2}})$ is contained in $S^H_{n,m}$, and thus can be written as a nontrivial $\F$-linear combination of $\zeta_1,\cdots,\zeta_m$. If $\hat{\iota}_K(\zeta_{\frac{m+1}{2}})$ contains $\zeta_i$ for some $i<\frac{m+1}{2}$, then $\partial \iota_K(\zeta_{\frac{m+1}{2}})$ would contain $U^{a_i}y_i$, which is a contradiction since $a_i < \frac{m+1}{2}$ and 
    \[
    \partial \iota_K(\zeta_{\frac{m+1}{2}}) = \iota_K(\partial \zeta_{\frac{m+1}{2}}) = \iota_K(U^{\frac{m+1}{2}}y_{\frac{m+1}{2}} + V^{\frac{m+1}{2}}z_{\frac{m+1}{2}}) = U^{\frac{m+1}{2}} \iota_K(z_{\frac{m+1}{2}}) + V^{\frac{m+1}{2}} \iota_K(y_{\frac{m+1}{2}}).
    \]
    Similarly, $\hat{\iota}_K(\zeta_{\frac{m+1}{2}})$ contains no $\zeta_i$ for any $i>\frac{m+1}{2}$. Furthermore, the converse statement is also true: for any $i\ne \frac{m+1}{2}$, $\hat{\iota}_K(\zeta_i)$ can never contain $\zeta_{\frac{m+1}{2}}$ by the same arguments. Hence, for some summand $W$, there is a decomposition $\widehat{\HFK}(S^3, K_{m,n}) = \F \langle \zeta_{\frac{m+1}{2}}\rangle \oplus W$ which is $\hat{\iota}_K$-equivariant. 
    

    We claim that $x$ is also isolated under the action of $\hat{\iota}_K$. To show the claim, consider the chain map
    \[
    f_1:CFK_\mathcal{R}(S^3,K^\ast_n)\rightarrow CFK_\mathcal{R}(S^3,U)
    \]
    defined as the projection to the free summand $O$, as well as the chain map
    \[
    f_2:CFK_\mathcal{R}(S^3,4_1)\rightarrow CFK_\mathcal{R}(S^3,K^\ast_n)
    \]
    defined as in \cite[Lemma 5.2]{kang2022torsion}; note that $f_2$ maps the generator of the free summand of $CFK_\mathcal{R}(S^3,4_1)$ to the generator of $O$. By \Cref{thm:relative CFK-to-CFA} and \Cref{thm:relative CFA-to-CFK}, and using the fact that the induced summands in those theorems are canonically chosen, i.e. depends only on the given splittings, we see that $f_1$ and $f_2$ induce $\iota_K$-equivariant chain maps
    \[
    \begin{split}
    (f_1)_{m,-1} &: CFK_\mathcal{R}(S^3,K_{n,m})\rightarrow CFK_\mathcal{R}(S^3,U), \\
    (f_2)_{m,-1} &: CFK_\mathcal{R}(S^3,(4_1)_{m,-1})\rightarrow CFK_\mathcal{R}(S^3,K_{n,m}),
    \end{split}
    \]
    such that the following conditions are satisfied.
    \begin{itemize}
        \item $(f_1)_{m,-1}$ maps $x$, i.e. the generator of $O$, to the generator of $CFK_\mathcal{R}(S^3,U)$, and maps $S^H_{n,m}$ and $D^H_{n,m}$ to 0.
        \item $(f_2)_{m,-1}$ maps the generator of the free summand of $CFK_\mathcal{R}(S^3,(4_1)_{m,-1})$, which we denote by $z$, to $x$.
    \end{itemize}
    Since the action of $\hat{\iota}_K$ on $CFK_\mathcal{R}(S^3,U)$ is trivial, we deduce from the first condition that for any generator $w$ of $S^H_{n,m} \oplus D^H_{n,m}$, its image $\hat{\iota}_K(w)$ does not contain $x$. Furthermore, by \cite[Lemma 5.2]{kang2022torsion},  $\hat{\iota}_K(z)=z$, from which it follows that
    \[
    \hat{\iota}_K(x) = \hat{\iota}_K((f_2)_{m,-1}(z)) = (f_2)_{m,-1}(\hat{\iota}_K(z)) = (f_2)_{m,-1}(z) = x.
    \]
    The claim is thus proven.

    In summary, we have produced a splitting
    \[
    \widehat{HFK}(S^3,K_{n,m})\simeq \F \langle x \rangle \oplus \F \langle \zeta_{\frac{m+1}{2}} \rangle  \oplus W'
    \]
    which is $\hat{\iota}_K$-equivariant. However, since all terms arising in $\partial \zeta_{\frac{m+1}{2}}$ and $\partial z$ have coefficients which are of the form $U^i$ or $V^j$ with $i,j\ge 2$, and no other generators have differentials containing either $\zeta_{\frac{m+1}{2}}$ or $z$, it follows that the basepoint actions $\hat{\Phi}$ and $\hat{\Psi}$ (which are partial derivatives of the full knot Floer differential with respect to $U$ and $V$, respectively) also respect the given splitting. Therefore, $K_{n,m}$ is involutively nontrivial; the theorem then follows from \Cref{prop: nonsimple implies stable}.
\end{proof}


\begin{rem}
    It might also be possible that one can use even $n$, with the condition that $n\ge 4$. When $n=2$, it is already known \cite[Remark 4.8]{hendricks2022involutive} that the knot Floer chain complex of $-T_{2n,4n+1}\sharp 2T_{2n,2n+1}$ splits off a summand $C_2$ consisting of a free summand and a 2-by-2 box summand, with the involution $\iota_K$ acting as in the case of $n$ odd with $n\ge 3$. It is suspected that this should also be true for every even $n\ge 4$.
\end{rem}

\appendix

\section{The kernel and image of type D projection maps}

Let $\mathcal{A}=\mathcal{A}(T^2)$ be the torus algebra and $M$ be a finitely generated type D structure over $\mathcal{A}$. We assume $M$ is graded by some nonabelian group, though this plays no substantial role in our arguments. Suppose that a projection $p^1 \in \End^\cA(M)$ is given, i.e. a type D (degree-preserving) endomorphism which satisfies $(p^1)^2 = p^1$. Recall that the composition of type D morphisms is defined graphically as:
\begin{align*}
(g \circ f)^1 = 
    \begin{tikzcd}[ampersand replacement = \&]
       {} \&{} \ar[d, dashed] \\
       {} \& \delta^M \ar[d, dashed]\ar[ldd, Rightarrow, bend right = 20] \\
       {} \& f^1 \ar[d, dashed]\ar[ld] \\
       {} \ar[d, Rightarrow]\& \delta^N \ar[d, dashed] \ar[dl,Rightarrow] \\
       {}\ar[d, Rightarrow] \& g^1 \ar[d, dashed] \ar[ld] \\
       {}\ar[d, Rightarrow] \& \delta^P \ar[dd, dashed]\ar[dl,Rightarrow] \\
       \mu \ar[d] \&{} \\
       {} \&{}  \\
    \end{tikzcd}
\end{align*}
Were $M$ a chain complex and $p^1$ a chain endomorphism, the image and kernel of $p^1$ would be clear. However, since the definition of type D structures over $\mathcal{A}$ is modeled on projective $\mathcal{A}$-modules, it is not immediately clear how the image and kernel of $p^1$ should be defined. For any type D structure $(M, \delta^1)$, there is an associated differential graded module $(\cA \otimes_\mathcal{I} M, \bI_\cA \boxtimes \delta^1)$. A type D projection map $p^1$ gives rise to a projection map $\bI_\cA \boxtimes p^1$ on $\cA \boxtimes M$, and the kernel and image of this map are clear. In this appendix, we formally define type D structures $\ker(p^1)$ and $\mathrm{im}(p^1)$ whose associated differential graded modules are precisely $\ker(\bI_\cA \boxtimes p^1)$ and $\mathrm{im}(\bI_\cA \boxtimes p^1)$ and prove their homotopy invariance. 

Recall that a type D structure $M$ over $\mathcal{A}$ consists of $\iota_0$-generators which $\mathbb{F}_2$-linearly generates $\iota_0 M$, $\iota_1$-generators which $\mathbb{F}_2$-linearly generates $\iota_1 M$, and the differential $\delta^1:M\rightarrow \mathcal{A}\otimes_\mathcal{I}M$, where $\mathcal{I}=\{0,1,\iota_0,\iota_1\}$ denotes the idempotent subring of $\mathcal{A}$. Choose $\mathbb{F}_2$-linear bases $B_0$ and $B_1$ of $\iota_0 M$ and $\iota_1 M$. Then we may write
\[
p^1(b) = \sum_{I\in\{\emptyset,1,2,3,12,23,123\}} \rho_I \otimes p_I (b)
\]
for uniquely determined elements $p_I(b)\in M$. We then define an $\mathbb{F}_2$-linear endomorphism $\hat{p}$ of $M$ as follows:
\[
\hat{p}(b) = p_\emptyset(b)\quad \text{for}\quad b\in B_0\sqcup B_1.
\]
Note that this definition does not depend on the choice of bases $B_0$ and $B_1$, due to the $\mathcal{A}$-linearity of $p^1$. Furthermore, the identity $p^1=(p^1)^2$ immediately implies $\hat{p} = \hat{p}^2$, i.e. $\hat{p}$ is an $\mathbb{F}_2$-linear projection. Hence, we may write
\[
M = \ker(\hat{p}) \oplus \mathrm{Im}(\hat{p}).
\]
Going forward, we will write $\rho x$ rather than $\rho \otimes x$ for elements of $\cA \otimes_{\cI} M$ and $p$ rather than $p^1$. 

We will now try to extend this $\mathbb{F}_2$-linear basis to a suitable $\mathcal{A}$-linear basis of $\mathcal{A}\otimes_\mathcal{I} M$ that is compatible with the action of $p$. Let us introduce the following terminology.

\begin{defn}\label{def: A linear basis}
    We say that a subset $S\subset \mathcal{A}\otimes_\mathcal{I} M$ is an \emph{$\mathcal{A}$-linear basis} of $\mathcal{A}\otimes_\mathcal{I} M$ if the following conditions are satisfied:
\begin{itemize}
    \item For each $s\in S$, either $\iota_0 s=s$ (in which case we define $i_s=0$) or $\iota_1 s=s$ (in which case we define $i_s=1$, and the map
    \[
    \iota_{i_s}\mathcal{A}\rightarrow \mathcal{A}\otimes_\mathcal{I} M,\quad \rho\mapsto \rho \otimes s
    \]
    is injective;
    \item For any distinct elements $s,s'\in S$, the orbits $\mathcal{A}s,\mathcal{A}s'\subset \mathcal{A}\otimes_\mathcal{I} M$ are disjoint and either $\iota_0 s=s$ and the elements $s,\rho_1 s,\rho_3 s, \rho_{23}s, \rho_{123}s$ are pairwise distinct in $\mathcal{A}\otimes_\mathcal{I} M$, or $\iota_1 s=s$ and the elements $s,\rho_2 s,\rho_{12} s$ are pairwise distinct in $\mathcal{A}\otimes_\mathcal{I} M$;
    \item The subset $\bigcup_{s\in S} \mathcal{A} s \subset \mathcal{A}\otimes_\mathcal{I} M$ is an $\mathbb{F}_2$-linear basis of $\mathcal{A}\otimes_\mathcal{I} M$.
\end{itemize}
\end{defn}


\begin{lem} \label{lem: Reeb chord correction terms}
    For any $x\in \mathrm{Im}(\hat{p})$, there exists some $y\in \sum_{|I|>0} \rho_I \otimes \ker(\hat{p})$ such that $p(x+y)=x+y$. Similarly, for any $x\in \ker(\hat{p})$, there exists some $y\in \sum_{|I|>0} \rho_I \otimes \mathrm{Im}(\hat{p})$ such that $p(x+y)=0$.
\end{lem}
\begin{proof}
    Suppose first that $x\in \mathrm{Im}(\hat{p})$. Since $\hat{p}+\hat{p}^2=0$, we may write
    \[
    (\mathrm{id}+p)(x)=\sum_{|I|\ge 1} \rho_I (z_I+w_I)
    \]
    where $z_I \in \mathrm{Im}(\hat{p})$ and $w_I \in \ker(\hat{p})$. Applying $p$ then gives
    \[
    0=(p+p^2)(x) = \sum_{|I|\ge 1} \rho_I (p(z_I)+p(w_I)) = \sum_{|I|=1} \rho_I z_I + \sum_{|I|\ge 2} \rho_I (z'_I+w'_I)
    \]
    for some $z'_I\in \mathrm{Im}(\hat{p})$ and $w'_I\in \ker(\hat{p})$. By comparing coefficients of $\rho_1$, $\rho_2$, and $\rho_3$, we see that $z_1=z_2=z_3=0$. Hence we may write
    \[
    (\mathrm{id}+p)\left( x+\sum_{|I|=1} \rho_I w_I \right) = \sum_{|I|\ge 2} \rho_I (z'_I+w'_I)
    \]
    for some $z'_I \in \mathrm{Im}(\hat{p})$ and $w_I \in \ker(\hat{p})$. By applying the same arguments again, we see that $z'_I=0$ for all Reeb chords $I$ with $|I|=2$, and thus we get
    \[
    (\mathrm{id}+p)\left( x+\sum_{|I|=1} \rho_I w_I + \sum_{|I|=2} \rho_I w'_I\right) = \rho_{123}(z''_{123}+w''_{123})
    \]
    for some $z''_I \in \mathrm{Im}(\hat{p})$ and $w_I \in \ker(\hat{p})$. By applying the same arguments again, we see that $z''_{123}=0$. Therefore we get
    \[
    (\mathrm{id}+p)\left( x+\sum_{|I|=1} \rho_I w_I + \sum_{|I|=2} \rho_I w'_I + \rho_{123}w''_{123} \right) = 0,
    \]
    as desired.

    Now suppose that $x\in \ker(\hat{p})$. Then we may write
    \[
    p(x) = \sum_{|I|\ge 1} \rho_I (z_I+w_I)
    \]
    for some $z_I \in \mathrm{Im}(\hat{p})$ and $w_I \in \ker(\hat{p})$. Applying $\mathrm{id}+p$ then gives
    \[
    0 = (p+p^2)(x) = \sum_{|I|=1} \rho_I w_I + \sum_{|I|\ge 2} \rho_I (z'_I+w'_I)
    \]
    for some $z'_I \in \mathrm{Im}(\hat{p})$ and $w'_I \in \ker(\hat{p})$. Hence iterating the arguments that we have used in the case of $x\in \mathrm{Im}(\hat{p})$, while replacing $p$ with $\mathrm{id}+p$, gives us the following fact: there exists some elements $z_I,z'_I,z''_I\in \mathrm{Im}(\hat{p})$ such that
    \[
    p\left( x+\sum_{|I|=1} \rho_I z_I + \sum_{|I|=2} \rho_I z'_I + z''_{123} \right) = 0,
    \]
    as desired.
\end{proof}

The following corollary is now a direct consequence of \Cref{lem: Reeb chord correction terms}.

\begin{cor} \label{cor: nice basis for kernel and image}
    There exists an $\mathcal{A}$-linear basis $\mathcal{B}$ of $\mathcal{A}\otimes_\mathcal{I} M$ such that, for any $b\in \mathcal{B}$, we have either $p(b)=0$ or $p(b)=b$. Furthermore, if $p(b)=0$, then $\delta(b)$ is contained in the $\mathcal{A}$-linear span of elements of $\mathcal{B}\cap \ker (p)$; if $p(b)=b$, then $\delta(b)$ is contained in the $\mathcal{A}$-linear span of elements of $\mathcal{B}\cap \ker(\mathrm{id}+p)$.
\end{cor}
\begin{proof}
Choose $\mathbb{F}_2$-linear bases $B_0$ and $B_1$ of $\iota_0 M$ and $\iota_1 M$, respectively, so that for each $b\in B_0 \sqcup B_1$, either $\hat{p}(b)=0$ or $\hat{p}(b)=b$. Then \Cref{lem: Reeb chord correction terms} tells us that, for each $b\in B_0 \sqcup B_1$, there exists some element $\varphi_b\in \sum_{|I|>0} \rho_I \otimes M$ such that:
\begin{itemize}
    \item If $\hat{p}(b)=0$, then $p(b+\varphi_b)=0$;
    \item If $\hat{p}(b)=b$, then $p(b+\varphi_b) = b+\varphi_b$.
\end{itemize}
Thus, for each $b\in B_0 \sqcup B_1$, we define the corresponding element $\tilde{b}\in \mathcal{A}\otimes_\mathcal{I} M$ as follows.
\begin{itemize}
    \item If $b\in B_0$, then we define $\tilde{b} = \iota_0(b+\varphi_b)$;
    \item If $b\in B_1$, then we define $\tilde{b} = \iota_1(b+\varphi_b)$.
\end{itemize}
Then we define the subset
\[
\mathcal{B} = \{\tilde{b}\,\vert\,b\in B_0 \sqcup B_1 \} \subset \mathcal{A}\otimes_\mathcal{I} M.
\]
It is clear from constructions that $\mathcal{B}$ is an $\mathcal{A}$-linear basis of $\mathcal{A}\otimes_\mathcal{I} M$. The final statement of the corollary also follows immediately.
\end{proof}

We are now ready to define the kernel and image of a type D projection endomorphism.

\begin{defn} \label{defn: kernel and image of type D projections}
    We define the kernel and image of $p$ as
    \[
    \ker(p) = \mathbb{F}_2 \langle b\in \mathcal{B}\,\vert\,p(b)=0\rangle,\quad \mathrm{Im}(p) = \mathbb{F}_2 \langle b\in\mathcal{B}\,\vert\,p(b)=b\rangle.
    \]
    Here, the differential maps
    \[
    \delta^1:\ker(p)\rightarrow\mathcal{A}\otimes_\mathcal{I}\ker(p),\quad \delta^1:\mathrm{Im}(p)\rightarrow\mathcal{A}\otimes_\mathcal{I}\mathrm{Im}(p)
    \]
    is given by the ones that are naturally induced by the differential map $\delta^1:M\rightarrow \mathcal{A}\otimes_\mathcal{I} M$ of $M$; note that this is possible by the final statement of \Cref{cor: nice basis for kernel and image}.
\end{defn}

\begin{rem}
    While the definition of \Cref{defn: kernel and image of type D projections} relies on the choice of $\mathcal{B}$, their isomorphism classes are independent of such choices. This is because they are precisely the kernel and image of the map
    \[
    \mathrm{id}\otimes p:\mathcal{A}\otimes_\mathcal{I} M\rightarrow \mathcal{A}\otimes_\mathcal{I} M.
    \]
    In fact, what we have shown in \Cref{cor: nice basis for kernel and image} is that the kernel and image of $\mathrm{id}\otimes p$ can be identified with some $\mathcal{A}$-module induced by a finitely generated type D structure over $\mathcal{A}$. Once this existence is established, uniqueness up to isomorphism is clear.
\end{rem}

We now discuss the homotopy invariance of kernels and images of type D projections.

\begin{lem}
    Let $p$ and $p'$ be two projections of $M$ which are homotopic. Then $\ker(p)$ and $\ker(p')$ are homotopy equivalent, and similarly, $\mathrm{Im}(p)$ and $\mathrm{Im}(p')$ are homotopy equivalent.
\end{lem}
\begin{proof}
    Since $\ker(p)=\mathrm{Im}(\mathrm{id}+p)$, it suffices to show that $\mathrm{Im}(p)$ and $\mathrm{Im}(p')$ are homotopy equivalent. To show this, we choose a homotopy $H$ between $p$ and $p'$, Then we have
    \[
    p+p' = \partial H + H\partial
    \]
    as $\mathcal{A}$-linear endomorphisms of $\mathcal{A}\otimes_\mathcal{I} M$. Consider the type D endomorphisms
    \[
    f = p'\vert_{\mathrm{Im}(p)}:\mathrm{Im}(p)\rightarrow \mathrm{Im}(p'),\quad g=p\vert_{\mathrm{Im}(p')}:\mathrm{Im}(p')\rightarrow \mathrm{Im}(p).
    \]
    Then we have
    \[
    \begin{split}
    \mathrm{id}_{\mathrm{Im}(p)} +(g\circ f) &= (p+pp')\vert_{\mathrm{Im}(p)}  \\
    &= (p+p(p+\partial H+H\partial))\vert_{\mathrm{Im}(p)} \\
    &= (\partial pH + pH\partial) \vert_{\mathrm{Im}(p)}.
    \end{split}
    \]
    Hence we see that $pH\vert_{\mathrm{Im}(p)}:\mathrm{Im}(p)\rightarrow \mathrm{Im}(p)$ gives a homotopy between $\mathrm{id}_{\mathrm{Im}(p)}$ and $g\circ f$. Similarly, $p'H\vert_{\mathrm{Im}(p')}$ is a homotopy between $\mathrm{id}_{\mathrm{Im}(p')}$ and $f\circ g$. Therefore $f$ and $g$ are homotopy equivalences, and thus $\mathrm{Im}(p)$ and $\mathrm{Im}(p')$ are homotopy equivalent.
\end{proof}

\bibliographystyle{amsalpha}
\bibliography{mathbib}

\end{document}